\documentclass{compositio}

\theoremstyle{plain}
\newtheorem{theorem}{Theorem}[section]

\theoremstyle{definition}
\newtheorem{remark}[theorem]{Remark}

\usepackage{xcolor}

\newcommand{\qbinom}[2]{\left(\begin{matrix}
#1\\#2
\end{matrix}\right)_{\mkern-9mu q}}

\usepackage{amssymb}
\usepackage{amsmath}
\usepackage{cancel}
\usepackage{graphicx}
\usepackage{comment}

\DeclareMathOperator{\Tr}{Tr}
\DeclareMathOperator{\tr}{tr}
\DeclareMathOperator{\sign}{sign}
\DeclareMathOperator{\chara}{char}

\def\I{{I}} 
\def\e{\mathbf{e}}  
\def\Up{U^{\sharp}}     
\def\Um{U^{\flat}}     
\def\Ua{U_a}
\def\Ub{U_b}
\def\Upb{\Up_b} 
\def\Uma{\Um_a} 
\def\Umb{\Um_b} 
\def\Umd{\Um_o} 
\def\L{V}
\def\n{{N}}

\def\Ubar{\bar{U}}  
\def\Ubara{\Ubar_a}
\def\Upbar{{\bar{U}^{\sharp}}} 
\def\Umbar{\bar{U}^\flat} 

\def\Mn{M_n}
\def\cI{{\mathcal I}}
\def\cJ{{\mathcal J}}

\newcommand{\diag}{\mathrm{diag}}

\newcommand{\GL}{\mathrm{GL}}

\newcommand{\id}{\mathrm{id}}

\newcommand{\rk}{\mathrm{rk}}

\newcommand{\Frob}{\mathrm{Frob}}

\newcommand{\vgen}[1]{\langle #1 \rangle}

\newcommand{\upr}[1]{\bar{#1}}

\newtheorem{thm}{Theorem}[section]
\newtheorem{pro}[thm]{Proposition}
\newtheorem{lem}[thm]{Lemma}
\newtheorem{cor}[thm]{Corollary}
\newtheorem{df}[thm]{Definition}
\newtheorem{ex}[thm]{Example}

\newtheorem{con}[thm]{Conjecture}

\theoremstyle{definition}
\newtheorem{rem}{Remark}

\def\p{E}
\def\Wr{\overline{W}_r}

\def\Z{\mathbb{Z}}
\def\Q{\mathbb{Q}}
\def\cC{\mathcal{C}}

\title{Matrix Kloosterman sums}

\author{M\'arton Erd\'elyi}
\email{merdelyi@math.bme.hu}
\address{Department of Algebra, BME, Egry J\'ozsef u. 1., 1111 Budapest, Hungary}
\author{\'Arp\'ad T\'oth}
\email{arpad.toth@ttk.elte.hu}
\address{Analysis Department, ELTE, P\'azm\'ny P\'eter s\'et\'any 1/c, 1117 Budapest and Alfr\'ed R\'enyi Institute of Mathematics, Re\'altanoda u. 13-15, H-1053 Budapest, Hungary}


\begin{abstract}
We study a family of exponential sums that arises in the study of the horocyclic flow on \(\GL_n\). We prove an explicit version of general purity and find optimal bounds for these sums.
\end{abstract}

\thanks{Erd\'elyi is supported by NKFIH (National Research, Development and Innovation Office) grants FK127906, FK135885 and the MTA–RI Lendület "Momentum" Analytic Number Theory and Representation Theory Research Group. T\'oth is supported by the MTA R\'enyi Int\'ezet Lend\"ulet Automorphic Research Group and by NKFIH grants K119528 and FK135885.}

\begin{document}
\begin{abstract}
We study a family of exponential sums that arises in the study of expanding horospheres on \(\GL_n\). We prove an explicit version of general purity and find optimal bounds for these sums.
\end{abstract}
\maketitle

\vspace*{6pt}\tableofcontents

\section{Introduction}\label{sec-intro}

\subsection{The subject of the paper}
The goal of this paper is to derive non-trivial bounds for certain exponential sums that are natural generalizations of the classical Kloosterman sum to the non-commutative algebra \(M_n(\mathbb{F}_{q})\) of
\(n\times n\) matrices over a finite field $\mathbb{F}_q$ with $q=p^f$ elements.

To define these sums let $\mathbb{F}_{p}$ be the prime field of \(\mathbb{F}_q \),  and $\overline{\mathbb{F}}$ an algebraic closure of \(\mathbb F_q\) so that for $m\geq1$  $\mathbb{F}_{q^m}\subset\overline{\mathbb{F}}$ is the unique degree $m$ extension of $\mathbb{F}_q$. Let $$\varphi_0:\mathbb{F}_{p}\to\mathbb{C}^*$$ be the additive character which maps $1\in\mathbb{F}_{p}$ to $\zeta=\exp(1/p)=e^{2\pi i/p}$, and  fix the additive characters
$$
\varphi=\varphi_0  \circ \mathrm{Tr}_{\mathbb{F}_q/\mathbb{F}_{p}}  \qquad\text{ and  }\qquad \varphi_m=\varphi_0  \circ\mathrm{Tr}_{\mathbb{F}_{q^m}/\mathbb{F}_p}
$$ of $\mathbb{F}_q$ and $\mathbb{F}_{q^m}$.

Let $M_n(\mathbb{F}_{q^m})$ be the algebra of $n \times n$ matrices over $\mathbb{F}_{q^m}$, and $\mathrm{\GL}_n(\mathbb{F}_{q^m})=M_n^*(\mathbb{F}_{q^m})\subset M_n(\mathbb{F}_{q^m})$ be the general linear group. Let $\psi$ (resp. $\psi_m$) be the additive character of $M_n(\mathbb{F}_q)$ (resp. $M_n(\mathbb{F}_{q^m})$) defined by
\[
\psi= \varphi \circ \tr
\]
(resp. $\psi_m=\varphi_m \circ \tr$), where $\tr=\tr_n:M_n(\mathbb{F}_{q^m})\to\mathbb{F}_{q^m}$ is the matrix trace. For $a\in M_n(\mathbb{F}_{q^m})$ define the sum
\begin{equation}\label{gen-kl-def}
K_n(a,\mathbb{F}_{q^m})=\sum_{x\in\mathrm{\GL}_n(\mathbb{F}_{q^m})}\psi_m(ax+x^{-1}).
\end{equation}
generalizing  the classical Kloosterman sum
\begin{equation}\label{kl-def}
K_1(\alpha,\mathbb{F}_{q^m})=K(\alpha,\mathbb{F}_{q^m}^*)=\sum_{x\in\mathbb{F}_{q^m}^*}\varphi_m(\alpha x+x^{-1}).
\end{equation}
When the field in question is clear, or when the arguments used do not depend on it, we will write $M_n$ and
\( K_n(a) \) for $M_n(\mathbb{F}_q)$ and \( K_n(a,\mathbb{F}_{q})\).

The interest in these sums arose in connection with a conjecture of Marklof concerning the equidistribution of certain special points associated to expanding horospheres. This conjecture, originally motivated by Marklof's work on Frobenius numbers \cite{Marklof}, was proved by Einsiedler, Moses, Shah and Shapira \cite{EMSS} using methods of ergodic theory. In the case of \(SL_2\) the connection to classical Kloosterman sums was known already to Marklof (see Section 2.1 of \cite{EMSS}) and together with Lee they proved an effective version of the conjecture for \( SL_3 \). This proof strongly hinted that non-trivial bounds of the sums  in (\ref{gen-kl-def}) could yield a proof of Marklof's conjecture with an effective power saving for higher rank situations as well. One of the main purposes of this paper is to provide such bounds, they are formulated in Theorems \ref{thm-general-bound} and \ref{thm-degenerate}. These bounds, together with further extensions in \cite{ETZ}, were then recently used by El-Baz, Lee and Strömbergsson for realizing the above goal in  \cite{EBLS}.

There is however also intrinsic interest in these sums as natural generalizations\footnote{The special case when \(a\) is a scalar matrix was first considered in a 1956 paper by Hodges \cite{Hodges} and reappeared again in various other contexts. See for example the papers \cite{Kim, Fulman, ChaeKim}. We thank Ofir Gorodetsky for bringing our attention to these earlier works.} of Kloosterman's sum. The relevance  and widespread use  in analytic number theory of \(K_1(\alpha)\) (see for example \cite{H-B}) is immediate from the fact that it is the additive Fourier transform of the function \(x\mapsto \varphi(x^{-1})\) on \(\mathbb{F}_{q}^*\) (extended by 0)
\[
\varphi(x^{-1})=\frac{1}{q}\sum_{\alpha\in \mathbb{F}_{q}} K_1(-\alpha)\varphi(\alpha x)
\]
and that suitable estimates are available for \(K_1(\alpha)\). In fact
one knows \cite{Weil-Exp-sum} (also \cite{Carlitz}) that if $\alpha$ is not 0, then the associated zeta-function is rational,
\[Z(T)=\exp \left(-\sum_{m \geqslant 1} \frac{K_1(\alpha, \mathbb{F}_{q^m}^*)}{m} T^{m}\right)=\frac{1}{1-K_1(\alpha)T+qT^2}
\]
from which
$$
K_1(\alpha,\mathbb{F}_{q^m}^*)= -(\lambda_1^m+\lambda_2^m)
$$
for some $\lambda_1, {\lambda}_2\in\mathbb{C}$. Moreover Weil's proof \cite{Weil-RH} of the Riemann hypothesis over function fields shows \cite{Weil-Exp-sum} that $|\lambda_i|=\sqrt{q}$ which gives the optimal bound
$$|K_1(\alpha, \mathbb{F}_{q^m}^*)|\leq2q^{m/2}$$
for $\alpha$ not 0. (The explicit description of the connection between exponential sums of this type and the Riemann Hypothesis for curves over function fields goes back to \cite{Hasse}.)

There are a number of extensions of these results in the commutative setting especially the so called hyper-Kloosterman sums \cite{Mordell,Deligne-Hyper-K, LRS, KMS}, and both these and the classical Kloosterman sums are ubiquitous in the theory of exponential sums \cite{Katz-monodromy}. There is also a deep connection between Kloosterman sums and modular forms  \cite{Poincare, Petersson,Linnik, Selberg, Des-Iw,Goldfeld-Sarnak}, and the notion of Kloosterman sum is extended to \(\GL_n\) \cite{Friedberg, Stevens}, as well as to other algebraic groups \cite{Dabrowski}, with many applications.

The sums \(K_n(a)\) considered in this paper are more natural from a ring theoretic point of view. If \(A\) is a finite-dimensional algebra over a finite field, then by the Wedderburn-Artin theorem the additive Fourier transform of \(\psi(x^{-1})\) (extended from \(A^*\) to \(A\) by
\( 0 \)) leads naturally to the matrix Kloosterman sums of (\ref{gen-kl-def}). These of course are also related to the group \(\GL_n\) but at the same time they are very strongly tied to the standard representation of this group.
From this ring theoretic point of view we have again \[\psi(x^{-1})=\frac{1}{q^{n^2}}\sum_{a\in M_n(\mathbb{F}_q)} K_n(-a)\psi(a x)\]
in the simple ring of \(n\times n\) matrices over a finite field.

The other main goal of the paper is then to generalize the classical results above from the Kloosterman sums \(K_1\) to \(K_n\), especially to understand the associated cohomology. The difficulty of this task stems from the fact that when \(K_n(a)\) is viewed as an exponential sum on the affine variety
\[
V=\{ (x,\Delta)\in M_n(\mathbb{F}_q)\times\mathbb{F}_q\,:\, \det(x)\Delta=1\},
\]
the part at infinity of the projective closure of \(V\), defined by the equation \(\det x=0\), is singular. However the sums \(K_n(a)\) provide a rare example for which their cohomology and so their 
zeta function is explicitly expressible in terms of one dimensional exponential sums and so the weights in the sense of  Deligne \cite{Deligne-weight} can be understood in elementary terms. This realization that exponential sums on algebraic groups can be treated rather explicitly is the other main achievement of the paper. The evaluations for the matrix Kloosterman sums in concrete terms, especially the semisimple case in Theorem \ref{thm-split-semisimple} is reminiscent of Herz's work on Bessel functions of matrix arguments in \cite{Herz}. This link to transcendental special function continues a long line of similar connections, for example those of the Gauss, Jacobi, and Kloosterman sums to the Gamma, Beta and Bessel functions. As an important by-product the concrete nature of these evaluations lead automatically to the estimates required for the equidistribution problem mentioned above.

\subsection{Statements of the results}

The statements below refer to a fixed finite field \(\mathbb{F}_q\), and so we will write \(K_n(a)\) for the sum \(K_n(a,\mathbb{F}_q)\) in (\ref{gen-kl-def}). We start with the following reduction. Let \(a_1,a_2\) be matrices of size \(n_1\times n_1\) and \(n_2\times n_2\), and let \(a_1\oplus a_2\) denote the block matrix \(\left(\begin{smallmatrix} a_1 & 0 \\0 & a_2\end{smallmatrix}\right)\) of size \((n_1+n_2) \times (n_1+n_2)   \).

\begin{thm}\label{thm-split-semisimple}
\begin{enumerate}
\item Assume that \(a_1\in M_{n_1}(\mathbb{F}_q)\), \(a_2\in M_{n_2}(\mathbb{F}_q)\)  and that their characteristic polynomials \(p_{a_i}(\lambda)\) 
are relatively prime. Then
\[
K_{n_1+n_2}(a_1\oplus a_2)=q^{n_1n_2}K_{n_1}(a_1)K_{n_2}(a_2).
\]
\item Assume that \(a\in M_n(\mathbb{F}_q)\) has characteristic polynomial $p_a(\lambda)=\prod_{i=1}^n(\lambda-\alpha_i)$, with \(\alpha_i\in \mathbb{F}_q\) all different. Then
\[
K_n(a) = q^{n(n-1)/2}\prod_{i=1}^n K_1(\alpha_i),
\]
where \(K_1(\alpha_i)\) is as in (\ref{kl-def}).
\end{enumerate}
\end{thm}

Now assume that all roots of the characteristic polynomial of \(a\) are in \(\mathbb{F}_q\). By the theorem above  we may assume that \(a\) has a unique eigenvalue \(\alpha\). Our first result in this direction is for nilpotent matrices.
\begin{thm}\label{thm-nil}
Assume that \(a\in M_n(\mathbb{F}_q\) is nilpotent. Then
\begin{equation}\label{eqn-nil}
K_n(a)=K_n(0)=(-1)^{n}q^{n(n-1)/2}.
\end{equation}
\end{thm}

In general we have
\begin{thm}\label{thm-motivic}
Assume that \(a\in \GL_n(\mathbb{F}_q)\) has a unique eigenvalue \(\alpha\neq 0\). Denote \(\lambda \) the partition of $n$ consisting of the sizes of the blocks in the Jordan normal form of \(a\). There exists a polynomial \(P_\lambda(A,G,K)\in \mathbb{Z}[A,G,K]\), that depends only on the block partition \(\lambda\), such that
\[K_n(a)=P_\lambda(q,q-1,K_1(\alpha)).\]
\end{thm}


\begin{rem}
While irrelevant for estimates, the separation of \(q\) and \(q-1\) in the above polynomial is natural from the cohomological point of view, as these correspond to sums over the additive group \(\mathbb{A}^1\) and the multiplicative group \(\mathbb{G}_m\).
\end{rem}

The proof of Theorem~\ref{thm-motivic} is constructive and allows one to express the Kloosterman sums \(K_n(a)\) recursively as a polynomial in \(q,(q-1)\) and \(K_1(\alpha_i)\), where \( \alpha_i\) runs through the eigenvalues of \(a\). 
For example if $\I_n$ denotes the identity matrix of size \(n\times n\) then we have
\begin{thm}\label{thm-recursion}
Assume that \(a=\alpha \I_n\), \(\alpha \neq 0\). Then

\begin{equation}\label{recursion-eq}
K_n(\alpha \I_n)=q^{n-1}K_1(\alpha) K_{n-1}(\alpha \I_{n-1}) + q^{2n-2}(q^{n-1}-1)\, K_{n-2}(\alpha \I_{n-2}).
\end{equation}
\end{thm}

The recursion formulas for a general partition are somewhat complicated to state but easy to describe algorithmically. See section~\ref{sec-ex-recursion}, which also contains further examples. These formulas are based on a parabolic Bruhat decomposition (section \ref{sec-parabolic-Bruhat}). Using the finer decomposition via a Borel subgroup, one can derive closed form expressions. For example we have
\begin{thm}\label{thm-scalar-closed-form}
If \(\alpha \in \mathbb{F}_q^*\) then \begin{equation*}
K_n(\alpha \I_n)= \sum_{w\in S_n,w^2=\I} q^{n(n-1)/2+N_w}(q-1)^{e_w} K_1(\alpha)^{f_w}
\end{equation*}
where \(S_n\) is the symmetric group and where for an involution \(w\in S_n\), $f_w$ (resp. $e_w$) is the number of fixed points (resp. involution pairs) of $w$ (i. e. $n=2e+f$) and
\[
N_w= |\{ (i,j)\,|\, 1\leq i < j \leq n, \,w(j)<w(i)\leq j\}|.
\]
\end{thm}

One can similarly express \(K_n(a)\) for \(a=\alpha\I+\sum_{i=1}^{n-1} \e_{i,i+1} \), when \(\alpha\neq 0\) and where
$\e_{i,j}$ is the elementary matrix with 1 in the $(i,j)$-th position and $0$ everywhere else, see (\ref{eq-K_n-U_n}) in section~\ref{sec--ex-purity}.

The use of the Borel subgroup Bruhat decomposition is also very suitable for deriving estimates for these generalized Kloosterman sums. As a first step we have the following.
\begin{thm}\label{thm-split-estimate}
If \(a\) has a unique eigenvalue \(\alpha\) then
\[
|K_n(a)|\leq |K_n(\alpha \I)|\leq 4 q^{(3n^2-\delta(n))/4}
\] where $\delta(n)=0$ if $n$ is even and $\delta(n)=1$ if $n$ is odd.
\end{thm}

Thus if the characteristic polynomial of $a$ splits over
$\mathbb{F}_q$, then the estimates do not require much input from \'etale cohomology. However to bound the sum \(K_n(a)\) in general it seems unavoidable to use cohomological methods. The main input from \'etale cohomology is Lemma \ref{thm-Hfg} from which we can derive the following.

\begin{thm}\label{thm-semisimple-purity}
Let \(a\in \mathrm{\GL}_n(\mathbb{F}_q)\) be a regular semi-simple element (i. e. the characteristic polynomial $p_a$ has no multiple roots over \(\overline{\mathbb{F}}\)). Then the exponential sum \(K_n(a)\) is cohomologically pure -- that is all the cohomology groups are trivial but the middle one and all the weights are $n^2$.
\end{thm}
In particular for these regular semi-simple elements,
we have "square-root cancellation" \[|K_n(a)|\leq  2^n q^{n^2/2}.\]
\begin{rem}
Note that as the conditions on the multiplicities of the roots of $p_a$ can be formulated as polynomial equations with integral coefficients in the variables $a_{i,j}$, thus Theorem~\ref{thm-semisimple-purity} is a concrete illustration of the theorems on "generic purity" of Katz-Laumon and Fouvry-Katz  (\cite{Katz-Laumon},\cite[ Theorem 1.1]{Fouvry-Katz}).

It is an intriguing question if the set of regular semi-simple elements is the actual purity locus. The methods of Theorem~\ref{thm-motivic} at least for low ranks allow one to see that matrices whose Jordan normal form over the algebraic closure has an eigenvalue with more than one Jordan block are too large for square root cancellations as can be seen by inspection. Interestingly matrices with only one block for each eigenvalue do exhibit square root cancellation. The methods in our paper do not yield enough information to decide whether these sums are cohomologically pure, see subsection~\ref{sec--ex-purity}.
\end{rem}

\begin{rem}
Will Sawin suggested to us a geometric approach that could shed even more light on these sums. The sum \(K_n(a)\) may be viewed as the trace of Frobenius on the stalk at \(a\) of a complex of sheaves (in fact, a perverse sheaf). The geometry of this complex is linked to the Springer correspondence, in the sense that the stalk at \(a\) should be related to the cohomology of the fixed points of \(a\) acting on various flag varieties. This is an interesting and promising approach whose exact shape can conjectured from the recursion formulas. While the paper was going through publication this suggestion was elaborated in \cite{EST}, settling the purity question in the previous remark  in the positive.
\end{rem}

Our main theorem for bounding \(K_n(a)\) is as follows.
\begin{thm}\label{thm-general-bound}
For all $a\in \mathrm{\GL}_n(\mathbb{F}_q)$ we have $$|K_n(a)|\ll_n q^{(3n^2-\delta(n))/4},$$ where $\delta(n)=0$ if $n$ is even and $\delta(n)=1$ if $n$ is odd.
\end{thm}

\begin{rem}
It is possible to refine the statement based on the characteristic polynomial of \(a\). If the characteristic polynomial is $p_a(t)=\prod_{i=1}^r(t-\alpha_i)^{n_i}$ for some pairwise distinct $\alpha_i\in\overline{\mathbb{F}}$, then
\[
|K_n(a)|\ll_n\prod_{i=1}^rq^{(3n_i^2-\delta(n_i))/4}\prod_{1\leq i<j\leq r}q^{n_i n_j}.
\]

\end{rem}

In the classical picture it is natural to look at the more general expression
\begin{equation}\label{kl-def-ab}
K_1(\alpha,\beta)=K(\alpha,\beta,\mathbb{F}_q^*)=\sum_{\gamma\in\mathbb{F}_q^*}\varphi(\alpha\gamma+\beta\gamma^{-1}).
\end{equation}
which in case \(\beta \neq 0\) immediately reduces to \(K_1(\alpha\beta^{-1})\)  by the simple identity
\[
K_1(\alpha,\beta)=K_1(\alpha\delta,\beta \delta^{-1})
\]
valid for any \(\delta \in \mathbb{F}_q^*\). However the case \(\beta=0\) is again interesting in its own way as it is the Fourier transform of the characteristic function of the set of invertible elements. While trivial to evaluate it is also unavoidable in the analytic applications.

We will look at these sums for \(n\times n\) matrices and so we let
\begin{equation}\label{gen-kl-def-ab}
K_n(a,b)=K_n(a,b,\mathbb{F}_q)=\sum_{x\in\mathrm{\GL}_n(\mathbb{F}_{q})}\psi(ax+bx^{-1}).
\end{equation}
It is easy to see that just like as above, we have \(K_n(a,b)=K_n(b,a)\) and for any  invertible \(g,h\)
\begin{equation}\label{eq:g-h-action}
    K_n(gah^{-1},hbg^{-1})=K_n(a,b).
\end{equation}
Therefore in case $\det a, \det b$ not both 0, we may use the results above to obtain the bound \(|K_n(a,b)|\ll q^{3n^2/4}\). However unlike in dimension 1 the other cases are not settled completely by Theorem~\ref{thm-nil} as the orbit structure of pairs of matrices \( (a,b) \) under the \(\GL_n\times \GL_n\)-action given in (\ref{eq:g-h-action}) gets more intricate when \(n>1\). We start with the most natural case of \(b=0\).

\begin{thm}\label{thm-projections} Assume that \(a\in M_n(\mathbb{F}_{q})\) has rank \(r\). Then
\[K_n(a,0,\mathbb{F}_{q})=(-1)^{r}q^{-r(r+1)/2} q^{rn}\vert \GL_{n-r}(\mathbb{F}_q)\vert.
\]
\end{thm}
Since for an \(r\)-dimensional subspace the Moebius function of the poset of subspaces evaluates to \((-1)^rq^{r(r-1)/2}\) by the $q$-binomial theorem (\cite{Stanley} Formula (1.87)), this result is vaguely reminiscent to the evaluation of the Ramanujan sums
\[
\sum _{\substack{1\leq x\leq q \\ (x,q)=1}} e^{2\pi i a  {x}/{q}}
=\sum _{d\mid (a,q)}\mu \left({d}\right)\frac {q}{d}.
\]

One can see that the sums $K_n(a,b)$ can get significantly larger than \(q^{3n^2/4}\) since for rank \(a=1\) we have \(|K_n(a,0)|\gg q^{n^2-n}\). The next theorem shows that this is close to the worst case scenario.

\begin{thm}\label{thm-degenerate}
Let $a,b\in M_n(\mathbb{F}_{q})$ be singular $n\times n$ matrices such that \[r=\rk(b)\geq s=\rk(a).\]	Then
\[
|K_n(a,b,\mathbb{F}_{q})|\leq 2q^{n^2-rn+r^2+\binom{\min(r,n-r)}{2}}.
\]
\end{thm}

\begin{cor} If \(a,b\in M_n(\mathbb{F}_{q})\) are not both \(0\), we have the general estimate
\[
|K_n(a,b,\mathbb{F}_{q})|\leq 2q^{n^2-n+1}.
\]
\end{cor}
\begin{rem}
Apart from the constant \(2\) which can probably be replaced by \(1+o(1)\) as \(q^n\to \infty\), this bound is sharp, since
\begin{equation}\label{eq-e_{1,n},e_{1,n}}
K_n(\e_{1,n},\e_{1,n})=q^{2n-2}|\GL_{n-2}(\mathbb{F}_q)|+(q-1)q^{n-1}|\GL_{n-1}(\mathbb{F}_q)|\sim q^{n^2-n+1}.
\end{equation}
\end{rem}
(Here again $\e_{i,j}$ stands for the elementary matrix with 1 in the $(i,j)$-th position and $0$ everywhere else.)

\subsection{The organization of the paper}

The paper naturally splits into four parts. The first deals with matrices \(a\) all of whose eigenvalues are in the ground field \(\mathbb{F}_q\) with the aim of evaluating the generalized Kloosterman sums in terms of classical ones. Later sections deal with the non-split case using  cohomology, and the entirely different and more combinatorial case of degenerate matrices. These are included because they are needed for the equidistribution problem in \cite{EBLS}. The final part provides some examples and highlights some open questions. 
The material bifurcates on several occasions and this may somewhat obscure the insight one gains from the results. Therefore we first highlight the generic case of regular semisimple matrices before further details.

To treat regular semisimple elements in concrete terms we need to assume that they are split over the ground field. The calculation to reduce to block upper diagonal $a$  with the blocks having no common eigenvalue is presented in subsection~\ref{sec-blocks}. We show here that in that case the Kloosterman sum \(K_n(a)\) factors as a product of Kloosterman sums of smaller ranks corresponding to the diagonal blocks.  This is elementary and leads immediately to Theorem~\ref{thm-split-semisimple}. If one is merely interested in this generic case then one can jump to subsection \ref{sec-semisimple} which shows how to circumvent the problem when the eigenvalues are only defined in a field extension. In this regard one is also naturally lead to conjecture \ref{conj:reg-ss} which suggests that an evaluation might be possible even in the non-split case.

The finer picture of the non semisimple case is both of natural interest and required by the applications. We spend a great deal of the paper on them. After the proof of Theorem~\ref{thm-split-semisimple} section 2 and parts of section~3 handle this case still under the assumption that the eigenvalues are in the ground field \(\mathbb{F}_q\). To reach our goals we will evaluate the sub-sums of \( K_n(a) \) restricted to Bruhat cells in various decompositions.

While the calculations in section 2 are somewhat lengthy the overall structure is simple. We use Bruhat decomposition with respect to a maximal parabolic fixing a line. In subsections~\ref{sec-parabolic-Bruhat} and \ref{sec-recursion} we introduce the necessary notation for this task. This decomposition is the first step in the plain old Gauss-Jordan elimination and it is natural to expect that this process should lead to an inductive structure for matrix Kloosterman sums. An elementary computation in subsection \ref{sec-recursion} justifies this expectation.

Since nilpotent matrices can be put into Jordan normal form over any field this step immediately shows that Theorem~\ref{thm-motivic} holds, at least if the polynomial in the statement is allowed to depend on the characteristic \(p\). It is easy to see that in the simplest case of Theorem~\ref{thm-recursion} there is a recursion that works over all primes simultaneously. The bounds in section 3 are based only on this result and if one is only interested in the estimates the rest of the section can be skipped.

However from the exponential sum point of view independence on the characteristic is of great interest in itself. Therefore it is important that the recursions to lower rank can be done across all finite fields universally for matrices whose eigenvalues are in that field. The rest of section 2 is then devoted to show this. Subsection~\ref{sec-red-2-GL_n-2} presents the technical core by showing that one may restrict to matrices with entries 1 and 0. These matrices can be lifted to \(\Z\) and can be put into Jordan normal form over \(\Q\). In effect this establishes that Theorem~\ref{thm-motivic} holds for almost all primes simultaneously. The final step in the proof is then a simple criterion for the existence of a Jordan normal form over \(\mathbb{Z}\) in subsection~\ref{sec-proof-motivic}. While this Jordan normal form reduction can be made explicit, the rather technical calculations are postponed to subsection \ref{sec-ex-recursion}.

The second main part of the paper, section 3, is about estimates. This again starts with an elementary but structural observation that repeated applications of the reduction in section 2 is equivalent to the use of the Bruhat decomposition with respect to a Borel subgroup.  We are able to refine the usual Bruhat decompositions just enough to establish the first bound in Theorem~\ref{thm-split-estimate}.

While the proof of the above estimates are still group-theoretic in nature, in general one needs methods from cohomology. Conversely the cohomological apparatus relies on these more classical arguments. The connection is given by a key lemma, Lemma~\ref{thm-Hfg}, whose simple proof somewhat disguises its importance. This statement allows one to push "trivial cancellations" over a field extension back to the original field.

In subsection~\ref{sec-cohomology} we review the cohomological apparatus by listing the necessary tools for the linear algebra that follows. Subsections~\ref{sec-semisimple} and \ref{sec-non-split-wt-bound} then give the proofs of the main theorems, including the special case that the sums $K_n(a)$ are "pure" when $a$ is a regular semisimple element. These results are based on the fact that the cohomology groups attached to the subsum restricted to a Bruhat cell vanishes in large enough degree.

The third part deals with \(K_n(a,b)\) in the degenerate cases. After dealing with the case \(b=0\) the general situation is reduced to this special case when \(b=0\). Again one needs a double coset decomposition suitable for this task. The machinery here is combinatorial in nature using Gaussian binomials \cite{Stanley}.

In the last part of the paper we provide examples to illustrate our results. We give explicit evaluations for low ranks. The case $K_2(a,b)$ is explicitly computed for any $a,b\in M_2(\mathbb{F}_q)$ including the non-split semisimple case. There are further calculations that illustrate the difficulties to get explicit results for $n\geq 4$.
These sections also contain some observations and open questions that 
are interesting on their own.

\subsubsection*{Notational conventions.}

Throughout the paper letters of the Greek alphabet \(\alpha, ...,\lambda, ...,\xi \) will denote scalars, that is elements of the field \(\mathbb{F}_q\), or its algebraic closure. Lower case letters of the Latin alphabet \(a,...,g,...,x\) will denote matrices, but their entries, while scalars, will usually be denoted with Latin characters \(a_{ij},...,g_{ij},...,x_{ij}\) as well.
Upper case letters \(A,...,G,...,X\) will denote sets of matrices, usually but not necessarily subgroups. While these depend on \( n \) that dependence is usually suppressed for better readability.
There are a few exceptions that should not cause confusion, a general partition will be denoted \(\lambda\), \(\I\) will denote the identity matrix (\(\I_n\) if the size is not clear from the context), \(K\) will denote Kloosterman sums, and on occasion \(S_1, S_2,...\) will denote some auxiliary sums.

For the cohomological arguments in subsections 3.5-3.7 we have to work with algebraic varieties, so then by $A,\dots,G,\dots,X$ we will also denote the  affine algebraic varieties of $M_n$ defined by simple algebraic equations in the matrix entries.

There are two notions of a trace in the paper, \(\Tr=\Tr_{\mathbb{F}_{q^m}/\mathbb{F}_q}\) will denote the trace from an overfield to the ground field, while \(\tr=\tr_n\) will denote the trace of a matrix of size \(n \times n\).

If \(u\) is a unipotent matrix that is upper triangular we will denote \(\bar{u}=u-\I\) the strictly upper triangular part of \(u\). Similarly, if \(U\) is the group of upper triangular unipotent matrices, and \(V \subset U\) then \(\bar{V}\) denotes the set (or variety) \(\{\bar{v}\,|\, v\in V\}\).

An important role is played by the Weyl group \(W\simeq N(T)/T\) of \(\GL_{n}\). Here \(T\) is the set of diagonal matrixes in \(\GL_n\), and \(N(T)\) is its normalizer.  If \(S_n\) is the group of permutations on the letters \(\{1,...,n\}\) then \(W \simeq S_n\), and we will choose a specific isomorphism in subsection~\ref{sec-borel+inv}.

Cohomology will mean \(\ell\)-adic cohomology with compact support, with some unspecified \(\ell\neq p\).

\begin{acknowledgements}
We thank Will Sawin and Min Lee for suggesting the problem of matrix Kloosterman sums. We thank the anonymous referees whose reports helped in improving the paper. We also thank Sawin and Gergely Z\'abr\'adi for helpful conversations and the R\'enyi Institute of Mathematics for ongoing support.
\end{acknowledgements}

%

\section{Identities via reduction of rank}


\subsection{Splitting to primary components}\label{sec-blocks}

We start with some elementary but important observations. Recall from (\ref{kl-def-ab}) and (\ref{kl-def}) that \(K_n(a,b)=\sum_{x\in \GL_n(\mathbb{F}_q)}\psi(ax+bx^{-1})\), and \(K_n(a)=K_n(a,\I)\). The sum is clearly unchanged if \(x\) is replaced by \(x^{-1}\), or if \(x\) is replaced by \(cx\) for some \(c\in \GL_n(\mathbb{F}_q)\). This leads immediately to
\begin{lem}
If $a,b\in M_n(\mathbb F_q)$ and $c\in \mathrm{\GL}_n(\mathbb{F}_q)$ then we have
\[
K_n(a,b)=K_n(b,a),\  
K_n(ac,b)=K_n(a,cb)
\text{ and } K_n(a)=K_n(cac^{-1}).
\]
\end{lem}

Therefore, \(K_n(a)\) only depends on
the conjugacy class of \(a\) and so so using conjugation we can put $a$ into Jordan normal form over $\overline{\mathbb{F}}$. To exploit this we recall the following result of Sylvester \cite{Syl} whose proof is given here for  the convenience of the reader.

\begin{lem}\label{Fkxl}
Let \(M_{k,l}\) be the vector space of \(k\times l\) matrices over an arbitary field $F$. If \(a_k\in M_k(F)\) and \(a_l\in M_l(F)\) are matrices with no common eigenvalue over the algebraic closure \(\overline{F}\) of \(F\) then the linear endomorphism of $M_{k,l}$ given by \[v \mapsto va_l-a_kv\] is an isomorphism.
\end{lem}

\begin{proof}
It suffices to prove that this map is injective. Assume that $$va_l-a_kv=0$$ for some $v$. Then for any polynomial \(p(t)\)
\[
vp(a_l)=p(a_k)v.
\]
Let  $p_k(t)\in F[t]$ (resp. $p_l(t)$) be the characteristic polynomial of $a_k$ (resp. $a_l$). Then by the Cayley-Hamilton theorem we have
$$
0=p_k(a_k)v=vp_k(a_l) \quad\text{ and also }\quad 0=vp_l(a_l).
$$
By our assumption on the eigenvalues we have that $p_k(t)$ and $p_l(t)$ are relatively prime in $F[t]$, thus there are \(r_k(t), r_l(t)\) polynomials such that \(p_k(t)r_k(t)+p_l(t)r_l(t)=1\) which implies $0=v$, hence our claim.
\end{proof}

To use this observation let  $U_{[k,l]}$ be the linear subgroup of \(\GL_n\) whose set of points for any ring \(R\) is
\( U_{[k,l]} (R) = \left\{\left(\begin{array}{c|c}\I_k & v \\ \hline 0 & \I_l\end{array}\right)\bigg| v\in M_{k,l}(R) \right\}\).
The fact that this is a sub-variety will play a role in the cohomological arguments, but for now we will only use the particular subgroup \( U_{k,l}(\mathbb{F}_q) \subset\GL_n(\mathbb{F}_q)\). As usual since the field \( \mathbb{F}_q\) is fixed, for  easier reading we will write \(K_n(a)\) for \( K_n(a,\mathbb{F}_q) \), and \(G\) and  \( U_{k,l} \) for the sets \(\GL_n(\mathbb{F}_q)\)  and \( U_{k,l}(\mathbb{F}_q) \).

\begin{pro}\label{felsoblokk}
Assume that $a=\left(\begin{array}{c|c}a_k & b \\ \hline 0 & a_l\end{array}\right)$ with $a_k\in M_k(\mathbb{F}_q), a_l\in M_l(\mathbb{F}_q), b\in M_{k,l}(\mathbb{F}_q)$ for some $k,l\in\mathbb{N}$ such that $k+l=n$ and $a_k$ and $a_l$ have no common eigenvalue in $\overline{\mathbb{F}}$. Then
\[K_n(a)=q^{kl}K_k(a_k)K_l(a_l).\]
\end{pro}

\begin{proof}
Since $\tr$ is invariant under conjugation, we have

\begin{multline*}
K_n(a)=\sum_{x\in G}\psi(ax+x^{-1})=\\
\frac{1}{q^{kl}}\sum_{x\in G}\sum_{u\in U_{[k,l]}}\psi(a(u^{-1}xu)+(u^{-1}xu)^{-1})=\\
\frac{1}{q^{kl}}\sum_{x\in G}\sum_{u\in U_{[k,l]}}\psi((uau^{-1})x+x^{-1}).
\end{multline*}

Now \[
uau^{-1}=\left(\begin{array}{c|c}a_k & b+va_l-a_kv \\ \hline 0 & a_l \end{array}\right)
\]
and so
\begin{equation*}
K_n(a)=
\frac{1}{q^{kl}}\sum_{x\in G}\psi(ax+x^{-1})\left( \sum_{v\in\mathbb{F}_q^{k\times l}}\psi_k((va_l-a_kv)x')\right),
\end{equation*}

where $x'$ is the $l\times k$ matrix which we get by deleting the first $k$ rows and last $l$ columns of $x$ and $\psi_k$ is the $k\times k$ matrix trace composed with $\varphi$.

From Lemma~\ref{Fkxl} we have
\[\sum_{v\in\mathbb{F}_q^{k\times l}}\psi_k((va_l-a_kv)x')=\sum_{v\in\mathbb{F}_q^{k\times l}}\psi_k(vx')=\left\{\begin{array}{ll}0, & \mathrm{if~} x'\neq 0 \\ q^{kl}, & \mathrm{if~} x'=0.\end{array}\right.\]
This immediately yields
\[
K_n(a)=\sum_{\substack{x\in G\\x'=0}}\psi(ax+x^{-1})= q^{kl}K_k(a_k)K_l(a_l)
\]

\end{proof}

\begin{proof}[Proof of Theorem~\ref{thm-split-semisimple}] The first claim was proved above. For the second using the invariance under conjugation we may assume that \(a=\diag(\alpha_1,...,\alpha_n)\). By Proposition~\ref{felsoblokk} we have
\[
K_n(a)=q^{n-1}K_1(\alpha_n)K_{n-1}(a')\]
where \(a'=\diag(\alpha_1,...,\alpha_{n-1})\). The result follows by induction.
\end{proof}


\subsection{A parabolic  Bruhat decomposition}\label{sec-parabolic-Bruhat}

In this section we prepare the proofs of Theorem~\ref{thm-motivic} and \ref{thm-recursion}. Our goal is to express the Kloosterman sum $K_{n}(a)$ in terms of sums $K_{{n-1}}(a')$ and $K_{{n-2}}(a'')$ where the matrices $a'$ and $a''$ are derived from $a$ by deleting one or two rows and columns.

When $a$ has a single eigenvalue $\alpha$, it is conjugate to
\begin{equation}\label{a-def}
a=\alpha \I + \overline{a}
\end{equation}  where \(\overline{a}  \) is strictly upper triangular.
Since \(K_n(a)\) is conjugacy invariant, we will assume that \(a\) itself is in the above form. Our reductions are then based on a parabolic Bruhat decomposition for $\GL_n$. While it can be deduced from general facts (\cite[Theorem 8.3.8]{Springer} and\break \cite[Proposition IV.14.21(iii)]{Borel} ) it is easier to work them out explicitly for the special case at hand.

Let \(P\) be the closed subgroup of \(\GL_n\)
defined by the vanishing of \(g_{n,1},...,g_{n,n-1}\). If \(F\) is a field \(P\) may be described alternatively as follows. Let  \(\GL_n(F)\) act on row vectors by multiplication on the right, \(v\mapsto vg\). Then \(P(F)\) is the stabilizer of the line \( \{ \lambda \,\mathbf{e}_n: \lambda \in F \}\)
where
$\mathbf{e}_n$ is the row vector $(0,...,0,1)$,
\begin{equation}\label{eq-P-def}
P(F)=\{ g \in \GL_n({F})\,|\, \mathbf{e}_n g  = \lambda \mathbf{e}_n , \lambda \in {F}^*\}.
\end{equation}

Since the arguments in this and the following sections will not be used in the cohomological proofs we will only concentrate on the set
\(P(\mathbb{F}_q)\).

Then as sets \(G= \GL_n(\mathbb{F}_q)=\bigsqcup_{k=1}^n P(\mathbb{F}_q)w_{(kn)}P(\mathbb{F}_q)\), where 	$w_{(kn)}$ is the permutation matrix corresponding to the transposition $(kn)$.
To make this a parameterization let
\begin{equation}\label{Uk}
U_k=\{\I + \sum_{j=k+1}^{n} u_j \e_{k,j} \mid u_{k,j}\in F\}
\end{equation}
be the set of unipotent matrices with nonzero elements only in the $k$-th row. (While this is again an algebraic group scheme, this fact will not play any role.)

We will only deal with \(F=\mathbb{F}_q\) and from here on we will write \(P\) and \(U_k\) for \(P(\mathbb{F}_q)\) and \(U_k(\mathbb{F}_q)\).
We then have the following decomposition into disjoint sets.
\begin{lem}\label{P-bruhat}
Let $ X_k=U_k w_{(kn)} P $. The map $U_k \times P  \, \to \;  X_k,
(u,g)   \, \mapsto \; uw_{(kn)} g$ is a bijection. Moreover \( G= \bigsqcup_{k=1}^n U_k w_{(kn)}P\).
\end{lem}

\begin{proof}
Let $x$ be a matrix in $G=\GL_n(\mathbb{F}_q)$ with rows $\mathbf{x}_{i}$, and write
$$
\mathbf{e}_n=\sum_{j=1}^{n} u_j \mathbf{x}_j,
$$
where $\mathbf{e}_n=(0,\cdots,0,1)$. We claim that
\[
X_k=\{x\in G\mid \min\{ j\mid u_j\neq 0\}=k \}.
\]
It is clear that \(X_kP=X_k\) and that \( U_kw \subset X_k\). Conversely if we let
$$
u=\I + \sum_{j=k+1}^{n} (u_j/u_k) \e_{k,j}
$$
then we have $ux\in wP$, with $w=w_{(kn)}$.

Finally it is enough to show that if \(u_1wp_1=u_2wp_2\) with \(w=w_{(kn)}\), then \(u_1=u_2\). To this effect note that the above implies that \( \mathbf{e}_nw u_2^{-1}u_1 =\mathbf{e}_n p_2p_1^{-1}w\). However \(\mathbf{e}_nw=\mathbf{e}_k\) and so the \(k\)-th rows of \(u_1\) and \(u_2\) are the same, which implies  \( u_1=u_2\).
\end{proof}

By the lemma we have
\begin{equation}\label{eq-X-k-sum}
K_n(a)=\sum_{g \in \GL_n(\mathbb{F}_q)} \psi(ag+g^{-1})=\sum_{k=1}^n \sum_{x\in X_k} \psi(ax+x^{-1}).
\end{equation}

When summing over \(X_k\), the case of \(k=n\), when \(X_n=P\) is trivial. To see this we will give an explicit Levi decomposition of \(P=P(\mathbb{F}_q)\). This fixes notation and will also be used in our further calculations  on \(X_k\) for \(k<n\). Note that again we will be working not with the algebraic groups but the fixed finite groups that they give rise to for \(\mathbb{F}_q\).

For $h \in \GL_{n-1}(\mathbb{F}_q)$ and $\lambda \in \GL_1(\mathbb{F}_q)$ let
\begin{equation}\label{[h,l]-def}
[h,\lambda]= \begin{bmatrix}
h & 0 \\0 & \lambda
\end{bmatrix}\in \GL_n(\mathbb{F}_q)
\end{equation}
and let
\begin{equation}
\label{L-def}
L=\{ [h,\lambda]: h \in \GL_{n-1}(\mathbb{F}_q), \lambda \in \GL_1(\mathbb{F}_q) \}\subset \GL_n(\mathbb{F}_q).
\end{equation}
Also let
\begin{equation}\label{V-def}
V=\{ \I +\sum_{l=1}^{n-1} v_l \e_{l,n} : v_l \in \mathbb{F}_q  \}
\end{equation}
then
$$
P=LV=VL.
$$

\begin{pro}\label{pro-triv-cell} If \(a\) is as in (\ref{a-def}) then
\[
\sum_{x\in X_n} \psi(ax+x^{-1})=
q^{n-1}K_1(\alpha)K_{n-1}(a')
\]
where \(a'\) is the matrix one gets by deleting the last row and column of \(a\).
\end{pro}

\begin{proof} Since
\[
\sum_{x\in X_n} \psi(ax+x^{-1})=\sum_{g\in L,v\in V} \psi(agv+(gv)^{-1})=
q^{n-1}\sum_{g\in L} \psi(ag+g^{-1}),
\]
the claim follows from the description of \(L\) in (\ref{L-def}) and that \[\tr(agv+(gv)^{-1})=\tr(ag+g^{-1})=tr(a'h+h^{-1})+\alpha\lambda+\lambda^{-1}.\]
\end{proof}


\subsection{The sum over the non-trivial cells.}\label{sec-recursion}

We continue to assume that \(a=\alpha\I_n+\bar{a} \), with \(\bar{a}\) strictly upper triangular, so that \(a\) has a unique eigenvalue \(\alpha\). In this section we will show
that for \(k<n\) the sum \[ \sum_{x\in X_k} \psi(ax+x^{-1})
\] can be expressed as a sum over the subvariety
\begin{equation}\label{def-L_k}
L_k(\alpha)=\{ g \in L\,|\,g_{k,j}=0 \text{ for all } j\neq k, \text{ and }\alpha g_{k,k}=g_{n,n}^{-1} \}.
\end{equation}
However we will only work over the set of points in \(\mathbb{F}_q\) and write \(L_l(\alpha)\) for \(L_k(\alpha)(\mathbb{F}_q)\). For \(\alpha=0\) these sets are empty, while for \(\alpha \in \mathbb{F}_q^*\) they are  subvarieties of \(L\) isomorphic to \(\GL_{n-2}\times \GL_1\times \mathbb{A}^{n-2}\) that can be visualized as elements \(g\in L\)  of the form
\begin{equation}\label{eq-g-block}
g=\left[ \begin{smallmatrix}
h_{11} & * & h_{12} & 0 \\0 & \lambda & 0 & 0 \\
h_{21} & * & h_{22} & 0 \\
0& 0 & 0 & 1/(\alpha\lambda)
\end{smallmatrix}\right].
\end{equation}
Here the blocks correspond to the partition \[\{1,...,n\}= \{1,...,k-1\} \sqcup \{k\} \sqcup \{k+1,...,n-1\} \sqcup\{n\}\] for \(1<k<n-1\), while
for \(k=1\) and \(n-1\) we have to adapt (\ref{eq-g-block}) to \(3 \times 3\) blocks
\begin{equation}\label{eq-g-block-k=1}
g=\left[ \begin{smallmatrix}
\lambda& 0 &  0 \\
*& 	h'' &  0 \\
0 & 0 & 1/(\alpha\lambda)
\end{smallmatrix}\right],\qquad g=\left[ \begin{smallmatrix}
h''& 	* &  0 \\
0 &	\lambda&   0 \\
0 & 0 & 1/(\alpha\lambda)
\end{smallmatrix}\right].
\end{equation}

This is merely a preliminary step in the reduction to rank \(n-2\), but is already quite useful, a fact that we will illustrate by proving Theorems~\ref{thm-nil} and \ref{thm-recursion}.

The reduction to the special form in (\ref{eq-g-block}) is based on the following calculation.

\begin{pro}\label{pro-red-tech}
For a fixed \(g\in L\), \[\sum_{v\in V} \psi(\alpha w_{(kn)}gv+(w_{(kn)}gv)^{-1})=0
\]
unless
\[ g_{k,j} = 0 \text{ for all } j\neq k, \text{ and }\alpha g_{k,k}=g_{n,n}^{-1}.\]
When these conditions hold
\[
\sum_{v\in V} \psi(\alpha w_{(kn)}gv+(w_{(kn)}gv)^{-1})=
q^{n-1} \psi(\alpha w_{(kn)}g+(w_{(kn)}g)^{-1}).
\]
\end{pro}

\begin{proof}
Let \(\bar{V}=\{\bar{v}=\sum_{l=1}^{n-1} v_l \e_{l,n}\,:\, v_{l}\in \mathbb{F}_q\}\). If \(v=\I+\bar{v} \in V\)
then \(v^{-1}=\I-\bar{v}\). The sum in question then becomes
\begin{multline*}
    \sum_{v\in V} \psi(\alpha w_{(kn)}gv+(w_{(kn)}gv)^{-1})=
    \psi(\alpha w_{(kn)}g+(w_{(kn)}g)^{-1})\sum_{\bar{v}\in \bar{V}} \psi(\alpha w_{(kn)}g\bar{v}-\bar{v}(w_{(kn)}g)^{-1})
\end{multline*}\
which vanishes unless the linear function
\[
v\mapsto 	\tr(\alpha w_{(kn)}g\bar{v}-\bar{v}(w_{(kn)}g)^{-1})=\alpha \sum_{l=1}^{n-1} g_{k,l} v_l - g_{n,n}^{-1} v_k
\]
is trivial.
\end{proof}

We can now prove the following claim about the sum over \(X_k\).

\begin{pro}\label{pro-red-X_k} Let \(a=\alpha \I +\bar{a}\), where \(\bar{a}\) strictly upper triangular, \(k<n\), \(X_k\) as in Lemma~\ref{P-bruhat} and \(L_k(\alpha)\) as in (\ref{def-L_k}). Then we have
\[
\sum_{x\in X_k} \psi(ax+x^{-1})=
q^{n-1}\sum_{g \in L_k(\alpha)}\sum_{u\in U_k}
\psi\left(a^u w_{(kn)}g+(w_{(kn)}g)^{-1}\right)
\]
where \(a^u=u^{-1}au\).


\end{pro}
\begin{proof}
Since \(Pu=P\) for any \(u\in U_k\), we have \[X_k=\bigsqcup_{u\in U_k} uw_{(kn)}P= \bigsqcup_{u\in U_k} uw_{(kn)}Pu^{-1}\] and so by \(u^{-1}au=\alpha\I +u^{-1}\bar{a}u\) we have
\[
\sum_{x\in X_k} \psi(ax+x^{-1})=
\sum_{\substack{g \in L,\; v\in V\\u\in U_k}} \psi(\alpha w_{(kn)}gv+(w_{(kn)}gv)^{-1})
\psi(u^{-1}\bar{a} uwgv).
\]
A direct  calculation, based on the fact that the last row of \(\bar{a}=a-\alpha \I \) is identically 0, then shows that
\begin{equation}\label{eq:trace-direct}
  \tr( u^{-1}\bar{a}uwgv) = \tr( u^{-1}\bar{a}uwg)
\end{equation}
is independent of \(v\). Therefore
\[
\sum_{x\in X_k} \psi(ax+x^{-1})=\sum_{\substack{g \in L\\u\in U_k}} \psi(u^{-1}\bar{a}uwg)\sum_{v\in V} \psi(\alpha w_{(kn)}gv+(w_{(kn)}gv)^{-1}).
\]
The inner sum is identical to the one in Proposition~\ref{pro-red-tech}, thus the proposition is proven.
\end{proof}

\begin{rem}
We comment briefly on identity (\ref{eq:trace-direct}).
The calculations are simplified by using \(M_n(\mathbb{F}_q)\) writing \(v=I+\overline{v}\) and observing that \(\tr x\overline{v}=\sum_{l=1}^n x_{n,l}v_l\) which clearly vanishes if the last row of the matrix \(x\) is identically 0.

However  one may argue alternatively via interpreting these matrices as linear transformations as follows. The group \(P\) is the parabolic subgroup fixing the 1-dimensional subspace \(M=\{\lambda \mathbf{e}_n\mid \lambda \in \mathbb{F}_q\}\) and so its elements also preserve the flag \(\{0\} \subset M \subset \mathbb{F}_q^n\). Any element \(g\) of \(P\) then gives rise to a linear transformation \(g'\) of \( M'=\mathbb{F}_q^n/M\).
The subgroup \(V\) itself is characterized by the property that its elements act trivially both on \(M\) and on \(M'\). Let \(x=  u^{-1}\bar{a}uwg\). Since  \( \mathbf{e}_nx=0\) the linear transformation \(x\)  also induces a map \(x'\) on \(M'\) and \(\tr x=\tr x'\). Since \(\mathbf{e}_nxv=0\) as well,
\(\tr xv = \tr(xv)'=\tr x'v'=\tr x'\).

In general all the calculations in the paper can easily be proved using block matrices, either by hand or  by using a symbolic algebra package. Since this gives an easy way to check the validity of these statement we will present most of the identities in this matrix interpretation.
\end{rem}

As a corollary to Proposition \ref{pro-red-X_k} we immediately have
\begin{proof}[The Proof of Theorem~\ref{thm-nil}]
If \(\alpha=0\) then the set \(L_k(\alpha)\) is empty, and so for \(k<n\)
\[
\sum_{x\in X_k} \psi(ax+x^{-1})=0.
\]
Since \(K_1(0)=-1\), Proposition~\ref{pro-triv-cell} gives
\[
K_n(a)=-q^{n-1}K_{n-1}(a')
\]
where \(a'\) is the matrix one gets by deleting the last rows and columns of \(a\). Since by assumption \(a\) is upper triangular nilpotent, so is \(a'\) and the theorem follows by induction.
\end{proof}

We finish the section with

\begin{proof}[The Proof of Theorem~\ref{thm-recursion}] If \(a=\alpha \I\), with
\(\alpha \in \mathbb{F}_q^*\) is a scalar matrix then \(\bar{a}=0\) and \(a^u=\alpha\I\). If \(1<k<n-1\) and \(g \) is as in (\ref{eq-g-block}) then \(g\) is invertible if and only if \(g''=\left[ \begin{smallmatrix}
g_{11} & g_{12} \\
g_{21} & g_{22}
\end{smallmatrix}\right]\) is invertible, in which case \((g^{-1})''=(g'')^{-1}\). It follows that for such \(k\)
\[
\sum_{x\in X_k} \psi(ax+x^{-1})= q^{2n-3}(q-1) K_{n-2}(\alpha\I)q^{n-k},
\]
and it is easy to see that this relation holds for \(k=1,n-1\) as well. This gives
\begin{equation*}
K_n(\alpha \I)=\sum_{k=1}^n \sum_{x\in X_k} \psi(\alpha x+x^{-1})=
q^{n-1}K_1(\alpha) K_{n-1}(\alpha \I) + q^{2n-3}(q-1) K_{n-2}(\alpha \I)\sum_{k=1}^{n-1} q^{n-k}
\end{equation*}
from which the desired formula follows.
\end{proof}


\subsection{Reduction to \( \GL_{n-2}\) } \label{sec-red-2-GL_n-2}

Assume that \(a=\alpha\I+\bar{a}\) has a unique eigenvalue \(\alpha\neq 0\), and that \(\bar{a}\) is strictly upper triangular. Since the results of the previous section take care of the case when \(n=2\) or \(a=\alpha \I\), we will assume that \(n\geq 3\) and that \(\bar{a}\neq 0\).

Recall that \( 	\sum_{x\in X_k} \psi(ax+x^{-1}) =q^{n-1} \sum_{u,g} 	\psi\left(a^u w_{(kn)}g+(w_{(kn)}g)^{-1}\right) \), the sums over \({u\in U_k}\) and \( {g \in L_k(\alpha)}	\). We will use the fact that as a variety
\( L_k(\alpha)\) is isomorphic to \( \GL_{n-2}(\mathbb{F}_q)\times \mathbb{F}_q^* \times \mathbb{F}_q^{n-2}\) to express \mbox{\(\sum_{x\in X_k} \psi(ax+x^{-1}) \)} as a linear combination of Kloosterman sums of rank
\( n-2 \). 

To state the reduction step we will denote \(m''_{\not{k},\not{n}}\) the matrix one gets by deleting the \(k\)-th and \(n\)-th rows and columns of a matrix \(m\). For us \(n\) will be fixed, and the value of \(k\) will be clear from the context, in which case we will often simply write \(m''\).
Also for any matrix \(m\) let \(m_{(k)}\) denote its \(k\)-th row, and \(m^{(l)}\) denote its \(l\)-th column.
We have

\begin{pro}\label{pro-red-using-blocks}
Assume that \(n>2\),  \(a=\alpha\I+\bar{a}\) with \(\bar{a}\) strictly upper triangular and let \(u=\I+\bar{u}\in U_k\). Then we have that
\[
\sum_{g\in L_k(\alpha)}
\psi\left(a^u w_{(kn)}g+(w_{(kn)}g)^{-1}\right)
=0\]
unless \(\bar{u}_{(k)}\bar{a}^{(j)}=\bar{a}_{k,j}
\) for \(j=k+1,...,n-1\). When this condition holds
\[
\sum_{g\in L_k(\alpha)}
\psi\left(a^u w_{(kn)}g+(w_{(kn)}g)^{-1}\right)
=\\q^{n-1}K_{n-2}\left(a''+z\right) \sum_{\lambda\in \mathbb{F}_q^*} \varphi(\lambda \xi ).
\]
where \(z=(\bar{a}^{(k)}\bar{u}_{(k)})''\in M_{n-2}\) and \(\xi=a_{k,n}- \bar{u}_{(k)} \bar{a}^{(n)}\).

\end{pro}

\begin{rem}
Note that when \(k=1\) or \(n-1\) the perturbation \(z\) vanishes for any \(u\).
\end{rem}

\begin{proof}
A direct calculation shows that
\begin{equation}\label{eq-a^u}
a^u=(\I-\bar{u})a(\I+\bar{u})={a}-\bar{u}\bar{a}+ \bar{a}\bar{u}.
\end{equation}

First assume that \(1<k<n-1\) and that
\[ g=\left[ \begin{smallmatrix}
g_{11} & y_1 & g_{12} & 0 \\0 & \lambda & 0 & 0 \\
g_{21} & y_2 & g_{22} & 0 \\
0& 0 & 0 & 1/(\alpha\lambda)
\end{smallmatrix}\right]\in L_k(\alpha)
\]  as in (\ref{eq-g-block}), in which case \(g''=\left[ \begin{smallmatrix}
g_{11} & g_{12}\\
g_{21} &  g_{22}
\end{smallmatrix}\right]\) is invertible and with $\tr_n$ being denoting the $n\times n$ matrix trace one also has that
\[
\tr_n ((wg)^{-1})=\tr_{n-2}( (g'')^{-1}).
\]
Moreover
\[
\tr_n(awg)=\tr_{n-2} (a'' g'')+\lambda a_{k,n}
\]
and
\[
\tr_n(\bar{a}\bar{u}wg)=\tr_{n-2}\left( (a_{(k)}u^{(k)})''g''\right)
\]
where we have used the fact that \(\bar{u}\) has only one non-zero row \(\bar{u}_{(k)}\).


Finally note that \(\tr({a}^uwg+(wg)^{-1})\) does not depend on the \((k-1)\times 1\) matrix \(y_1\), and as a function of \(y_2\) only depends on \(\tr(\bar{u}\bar{a}wg)\).  The function \[y_2 \mapsto \tr( \bar{u}\bar{a}wg)\] is constant if and only if
\[
\bar{u}_{(k)}\bar{a}^{(j)}=0
\]
for \(j=k+1,...,n-1\), and if this condition does not hold the sum over \(g\in L_k(\alpha)\) vanishes. This proves the claim for \(1<k<n-1\). The remaining cases are treated similarly.
\end{proof}

We will now specify the result in case \(a=\alpha\I+\bar{a}\) is in Jordan normal form. There is a partition \(\lambda\) associated to \(a\), i.e. a sequence of positive integers \(n_1\leq n_2 \leq ...\leq n_l\), such that \(n_1+n_2+..+n_l=n\).  Conversely, given a partition \(\lambda=[n_1,...,n_l]\) let  \begin{equation}\label{def-N_i}
N_i=n_1+...+n_i
\end{equation} so that \(1\leq N_1<N_2<...<N_l=n,\) and form
\begin{equation}\label{eq-lambda-2-b}
\bar{a}(\lambda)=\sum_{j=1}^{n-1} \varepsilon_j(\lambda) \e_{j,j+1},\ \text{  where }\
\varepsilon_j(\lambda)=\begin{cases}
0 & \text{ if } j=N_i, \text{ for some }i, \\
1 & \text{ otherwise}.
\end{cases}
\end{equation}
Any \(a\) with a single eigenvalue \(\alpha\) is conjugate to one of the matrices \(\alpha \I +\bar{a}(\lambda)\) and we will assume from now on that \(a\) is already in that form. Since scalar matrices were already dealt with, we will also assume that \(\lambda\neq [1,1,...,1]\), which ensures that \(\varepsilon_{n-1}=1\).

\begin{thm}\label{thm-begin-red}
Assume that \(n> 2\), \(\alpha\neq 0\),
\(a=\alpha \I + \bar{a}(\lambda)
\) whith $\varepsilon_j=\varepsilon_j(\lambda) \in \{0,1\}$ as in (\ref{eq-lambda-2-b}) with \(\varepsilon_{n-1}=1\). Then

\begin{enumerate}
\item\label{X_n-1}
\[
\sum_{x\in X_{n-1}}   \psi(ax+x^{-1}) = -q^{2n-2} K_{n-2}(a'')
\]
where \(a''=a''_{\cancel{n\text{-}1},\not{n}}\) -- the matrix obtained by deleting the last two rows and columns of $a$.

\item\label{X_k-vanishing} If \(k\leq n-2\) then
\[\sum_{x\in X_k} \psi(ax+x^{-1})=0\]
unless 
\(k=N_i\) for one of the \(N_i\) in (\ref{def-N_i}) for which \(n_i>1\).

\item\label{X_k-red-2-z} When  \(k=N_i=n_1+..+n_i<n-1\), and \(n_i>1\) we have
\begin{equation}
\sum_{x\in X_k} \psi(ax+x^{-1})=
q^{2n-2} \sum_{z\in Z,\lambda \in \mathbb{F}_q^*} K_{n-2}(a''+z)\phi(\lambda\xi_l )
\end{equation}

where \(a''=a''_{\not{k},\not{n}}\) and
\[
Z=\left\{ \textstyle \sum\limits_{j=i+1}^{l-1} \xi_j \e_{k-1,N_j-1} +\xi_l \e_{k-1,n-2} \,|\,\xi_j \in \mathbb{F}_q \text{ for }i+1\leq j\leq l \right\} \subset M_{n-2}.
\]
In \(Z\) the elementary matrices \(\e_{*,*}\) are of size \((n-2)\times(n-2)\).
\end{enumerate}
\end{thm}
\begin{proof}
The statements are easy corollaries of Proposition~\ref{pro-red-using-blocks} except for the fact in (ii) that \(n_i\) must be greater than 1, which is equivalent to \(\varepsilon_{k-1}\neq 0\). Since \(k<n-1\) and \(\varepsilon_{k-1}= 0\) imply that the parameters in Proposition~\ref{pro-red-using-blocks} are very simple, all \(z=0\), and \(\xi=-\bar{u}_{k,n-1}\). Therefore
\[
\sum_{x\in X_k} \psi(ax+x^{-1}) =q^{2n-2} \sum_{u} K_{n-2}\left(a''\right) \sum_{\lambda\in \mathbb{F}_q^*} \varphi(-\lambda \bar{u}_{k,n-1} )
\]
vanishes.
\end{proof}

While the matrices \(a''+z\) are not in Jordan normal form, they are again matrices with a single eigenvalue \(\alpha\). Therefore collecting them according their conjugacy classes gives a reduction algorithm, in fact a characteristic \(p\) version of Theorem~\ref{thm-motivic} (see Proposition \ref{pro-intermed-red-jordan}). In the next two sections we will explicate this strategy and prove that the polynomials that arise this way do not depend on \(p\).


\subsection{Jordan normal forms over \(\mathbb{Z}\)}

The proof of Theorem~\ref{thm-motivic} will be based on proving that the perturbations arising from the reduction to \(M_{n-2}\) can be collected into Jordan normal forms that do not depend on the characteristic \(p\). For this we will need some details about Jordan normal forms for integral nilpotent matrices. This of course is a trivial task over \(\mathbb{Q}\), but requires a little care when one works over \(\mathbb{Z}\).

Assume for example that \(x\) is an \(n \times n\) nilpotent matrix, and \(g\in \GL_n(\Z)\) is such that \(gxg^{-1}\) is in Jordan normal form as in (\ref{eq-lambda-2-b}). A moment's thought reveals that \(\{ vx\,|\, v\in \mathbb{Z}^n\}\), the row space of \(x\), must be a direct summand of \(\Z^n\), which we also think of as row vectors. This by itself is not sufficient for the existence of a Jordan normal form over \(\Z\) but we have the following.

\begin{thm}\label{thm-jordan}
Let \(\bar{a}\) be an \(n \times n\) nilpotent matrix with integral entries. There exist \(g\in \GL_n(\Z)\) such that \(g\bar{a}g^{-1}=\sum_{j=1}^{n-1}\varepsilon_j \e_{j,j+1}\), \(\varepsilon_j\in \{0,1\}\) if and only if
\[
\{ v \bar{a}^k\,|\, v\in \mathbb{Z}^n\}
\]
is a direct summand of \(\mathbb{Z}^n\) (as an abelian group) for any \(k \in \mathbb{N}\).

This is equivalent to the conditions that
\[
\{ \bar{a}^k v^T\,|\, v\in \mathbb{Z}^n\}
\]
is a direct summand of \((\mathbb{Z}^n)^T\) for any \(k\) (here $\cdot^T$ is the matrix transpose).
\end{thm}

The following examples illustrate the situation.
\begin{ex}
Let \( x=\left[\begin{smallmatrix}
0&2\\0 &0
\end{smallmatrix}\right]\), its Jordan normal form over \(\mathbb{Q}\) is \(y=\left[\begin{smallmatrix}
0&1\\0 &0
\end{smallmatrix}\right]\). If \(g=\left[\begin{smallmatrix}
a&b\\c&d
\end{smallmatrix}\right]\), the equation \(gx=yg\) leads to \(c=0\) and \(d=2a\). Therefore the equation \(gxg^{-1}=y\) has no solution in \(SL_2(\Z)\), or even \(SL_2(\Q)\).
\end{ex}

\begin{ex}
Let \( x=\left[\begin{smallmatrix}
0 & 1 & 0 & 0 \\
0 & 0 & 0 & 2 \\
0 & 0 & 0 & 1 \\
0 & 0 & 0 & 0
\end{smallmatrix}\right]\) with normal form \( y=\left[\begin{smallmatrix}
0 & 0 & 0 & 0 \\
0 & 0 & 1 & 0 \\
0 & 0 & 0 & 1 \\
0 & 0 & 0 & 0
\end{smallmatrix}\right]\). The \(\Z\)-span of the rows of \(x\) is clearly a direct summand. If \(gxg^{-1}=y\) then also \(gx^2=y^2g\), but \( x^2=\left[\begin{smallmatrix}
0 & 0 & 0 & 2 \\
0 & 0 & 0 & 0 \\
0 & 0 & 0 & 0 \\
0 & 0 & 0 & 0
\end{smallmatrix}\right]\) and \( y^2=\left[\begin{smallmatrix}
0 & 0 & 0 & 0 \\
0 & 0 & 0 & 1 \\
0 & 0 & 0 & 0 \\
0 & 0 & 0 & 0
\end{smallmatrix}\right]\), and so a Jordan normal form over \(\Z\) does not exist.
\end{ex}

Theorem~\ref{thm-jordan} follows along the lines of the standard proofs in the case of a vector space over a field. For completeness we present such a proof below, but
for both this proof and the applications of the theorem it is more convenient to work with linear transformations than matrices.

Let \(R\) be either \(\mathbb{F}_q\) or \(\mathbb{Z}\)
and \(\L \) a free \(R\)-module of finite rank \(n\). If \(A:\L \to \L\) is an \(R\)-homomorphism, then it gives rise to an \(R[T]\)-module structure on \(\L\), where \(R[T]\) is the polynomial ring over \(R\), and \(T v:=A(v)\). If needed we will denote these \(R[T]\)-modules by \(\L_A\) to distinguish modules corresponding to different transformations.

We will be interested in the situation when \(A\) is nilpotent, \(A^n=0\). If \(v\in \L\) let \(k\) be minimal such that \(A^kv=0\) and let
\begin{equation}\label{eq-v-gen}
\vgen{v} = Rv+R(Av)+\dots+R(A^{k-1}v) \end{equation}  denote the cyclic sub-module generated by \(v\). In this case we will call \(\L\) cyclic if \(\L_A=\vgen{v}\) for some \(v\in \L_A\). This happens exactly when the \(R[T]\) module \(\L_A\) is isomorphic to \(\mathcal{C}_n=R[T]/(T^n)\). If \(R\) is a field, then any \(\L_A\) is a direct sum of cyclic modules, but this fails for \(R=\Z\), and in general when \(R[T]\) is not a principal ideal domain.

\begin{proof}[The Proof of Theorem~\ref{thm-jordan}]
The theorem is equivalent to the following statement: if \(\L \simeq \mathbb{Z}^n\), and \(A: \L \to \L \) is a nilpotent homomorphism, then the \(\mathbb{Z}[T]\) module \(\L_A\) is a direct sum of cyclic modules if and only if \(A^k(\L)\) is a direct summand of \(\L\) (as an abelian group) for all \(k\in \mathbb{N}\).

Note that by the structure theorem for finitely generated abelian groups a subgroup \(\L'\) of \(\L\) is a direct summand if and only if for \(k \in \mathbb{Z}, v \in \:\L\), \(kv \in \L'\) implies that \(v \in \L'\). This immediately shows that \(\ker A\) is a direct summand of \(\L\), and moreover \(A(\L)\cap \ker A\) is a direct summand of \(\ker A\). Let \(\L_0\) be a complementary  direct summand so that
\begin{equation}\label{eq-Lambda_0}
\ker A = \L_0 \oplus (A(\L)\cap \ker A).
\end{equation}

Since \(A\) is nilpotent,  the rank of \(A(\L)\) is strictly less than that of \(\L\). The condition on \(A\) descends to \(A(\L)\), and so by induction we have that \(A: A(\L) \to A(\L)\) has a cyclic basis, i.e. there are \(v_1,...,v_l\) such that
\(A(\L) \simeq \vgen{Av_1}\oplus ...\oplus \vgen{Av_l}\). If we let \(d_i\) be the smallest integer \(k\) such that \(A^{k}v_i=0\), then we have that the set
\begin{equation}\label{eq-A(Lambda)-basis}
\{ A^jv_i\,|\, i=1,...,l, j=1,...,d_i-1\}
\end{equation}
is a basis of the free abelian group \(A(\L)\).

Let \(v_{l+1}, ..., v_{r}\) be such that \(\L_0=\oplus_{i=l+1}^r \mathbb{Z}v_i\), where \(\L_0\) is as in (\ref{eq-Lambda_0}).  Extending the notation from above we let \(d_i=1\) for \(i=l+1,...,r\).  

We claim that
\[
\L_A \simeq \oplus_{j=1}^{r} \vgen{v_j}.
\]

We need to prove that for each \(v \in V\), there is a unique choice of \(\alpha_{i,j}\in \mathbb{Z}\) such that
\begin{equation}\label{eq-Z-lin-comb}
v=\sum_{i=1}^r \sum_{j=0}^{d_i-1} \alpha_{i,j} A^jv_i.
\end{equation}
To see uniqueness assume that \(v=0\) is expressed this way. Then \(Av=0\) as well, and
the linear independence of the set in (\ref{eq-A(Lambda)-basis}) shows that \(\alpha_{i,j}=0\) for all \(i=1,...,l\) and \(j=0,...,d_i-2\). Since by (\ref{eq-Lambda_0}) we have that \(A^{d_i-1}v_i\), for \(i=1,...,l\) and \(v_{l+1}, ..., v_{r}\) are linearly independent this shows that \(\alpha_{i,d_i-1}=0\) as well for all \(i=1,...,r\).

It remains to show that every \(v \in \L \) can be expressed as an integral linear combination as in (\ref{eq-Z-lin-comb}).
Note that by (\ref{eq-A(Lambda)-basis}) this is clearly true for \(Av\),
\[
Av=\sum_{i=1}^l \sum_{j=0}^{d_i-2} \alpha_{i,j} A^{i} (Av_j)
\]
for some \(\alpha_{i,j}\in \mathbb{Z}\). Let \(v'=\sum_{i=1}^l \sum_{j=0}^{d_i-2} \alpha_{i,j} A^{i} v_j\). Then \(v-v'\in \ker A\), thus proving the claim.
\end{proof}


\subsection{The Proof of Theorem~\ref{thm-motivic}} \label{sec-proof-motivic}

The proof of Theorem~\ref{thm-motivic} relies on the following two propositions. 
\begin{pro}\label{pro-intermed-red-jordan}
Assume that \(n\geq 2\), \(\alpha \in \mathbb{F}_q^*\),
\(a=\alpha \I + \bar{a}(\lambda)
\) with $\varepsilon_j\in \{0,1\}$ as in (\ref{eq-lambda-2-b}) with \(\varepsilon_{n-1}=1\). Also assume that 	\(z_1= \sum_{j=i+1}^{l-1} \xi_j \e_{k-1,N_j-1}+\xi_l \e_{k-1,n_2}\), and that
\[
z_2=\sum_{j=i+1}^{l-1} \eta_j \e_{k-1,N_j-1}+\eta_l \e_{k-1,n_2}, \text{ where } \eta_j=\begin{cases} 0 & \text{ if } \xi_j = 0\\
1 & \text{ if } \xi_j \neq 0.\end{cases}
\]
Then \(a''+z_1\) and \(a''+z_2\) are conjugate in \(\GL_{n-2}(\mathbb{F}_q)\), and so \( K_{n-2}(a''+z_1)=K_{n-2}(a''+z_2)\).
\end{pro}

As a corollary of Proposition~\ref{pro-intermed-red-jordan} one immediately has that for \(k=N_i\) as above
\begin{multline}
\label{eq-red-try}
\sum_{x\in X_k} \psi(ax+x^{-1})=
q^{2n-2} \sum_{z\in Z_0}(q-1)^{J+1} \left( K_{n-2}(a''+z)-K_{n-2}(a''+z+\e_{k-1,n-2})\right)
\end{multline}
where \(a''=a''_{\not{k},\not{n}}\), and where \(z\) runs over
\begin{equation}\label{eq-Z-0}
Z_0=\left\{ z= \sum_{j=i+1}^{l-1} \eta_j \e_{k-1,N_j-1}\,|\, \eta_j \in \{0,1\} \text{ for } i+1\leq j\leq l\right\}
\end{equation}
with \(J=J(z)=\vert\{ j\,|\, \eta_j=1\}\vert\).

This in itself proves a version of  Theorem~\ref{thm-motivic} valid for almost all primes. The matrices \(a''+z\), \(a''+z+\e_{k-1,n-2}\)  can obviously be lifted to \(\mathbb{Z}\) where they can be put into Jordan normal form after a rational change of basis. This shows there are well defined partitions \(\lambda''(z), \lambda''(z+\e_{k,n-2})\) of \(n-2\) associated to the partition \(\lambda\), the value \(k=N_i\) and \(z\) which are independent of \(p\) except for the finitely many primes dividing the determinant of the change of base matrix. The next proposition 
shows that such exceptions do not arise.

\begin{pro}\label{pro-jordan-indep}
Let $\bar a\in M_n(\mathbb{Z})$ be as in Proposition~\ref{pro-intermed-red-jordan} and  \(z\in Z_0\) be as in (\ref{eq-Z-0}).
There exists a unique partition \(\lambda''\vdash n-2\) such that for any prime $p$ and finite field $\mathbb{F}_q$ of characteristic $p$ the partition $\lambda''(z)\vdash n-2$ associated to the nilpotent matrix \(\bar{a}''+z\) equals to $\lambda''$.

Similarly  there exist a partition \(\mu''\vdash n-2\) such that the block partition of the matrix \(\bar{a}''+z+\e_{k-1,n-2}\) is \(\mu''\).
\end{pro}

\begin{proof}[The Proof of Proposition~\ref{pro-intermed-red-jordan}]
Again we will use the language of linear transformations. Let \(V\) be a finite dimensional vector-space over \(\mathbb{F}_q\), and \(A: \L \to \L\) a nilpotent linear transformation such that
\[
\L_A \simeq \vgen{v_0}\oplus \vgen{v_1} \oplus ...\oplus \vgen{v_l}.
\]
Let \(d_i = \dim \vgen{v_i}\).
We are interested in perturbations \(A+Z_1\), \(A+Z_2\) of \(A\), where \(Z_1, Z_2\) are
such that
\begin{align*}
Z_1(A^j v_i)=Z_2(A^j v_i) &=0,  \;\text{for all } i=1,...,l, j=0,...,d_i-1,\\ Z_1(A^j v_0)=Z_2(A^j v_0) &=0 \;\text{ for } j=0, ...,n_0-2.
\end{align*}
and that furthermore
\[
Z_1(A^{d_0-1} v_0)= \sum_{i=1}^l \xi_ i A^{d_i-1} v_i 
\quad \text{  and } \quad Z_2(A^{d_0-1} v_0)= \sum_{i=1}^l \eta_ i A^{d_i-1} v_i
\]
where \(\eta_i=0\) or \(1\) depending on whether \(\xi=0\) or not.

Let \(\phi: V_A \to V_A\) be the \(A\)-linear  (\(\phi\circ A=A\circ \phi\)) isomorphism for which
\[
\phi(v_i)=\begin{cases}\frac{1}{\xi_i}v_i, & \text{ for } 1\leq i \leq l,  \xi_i\neq 0\\ v_i, & \text{ for } 1\leq i\leq l, \xi_i=0\end{cases}.
\]
Then \(\phi(Z_1(v))=Z_2(\phi(v))\) as well, showing that the modules \(V_{A+Z_1}\) and \(V_{A+Z_2}\) are isomorphic.
\end{proof}

\begin{proof}[The Proof of Proposition~\ref{pro-jordan-indep}]
The statement is only meaningful if \(n>2\). Lift the matrix 
\(\bar{a}''\) 
which only has entries \(0\) and \(1\) to a matrix \(\tilde{a}\) over \(\Z\) by lifting \(0_{\mathbb{F}_q}\) to \(0_{\Z}\) and \(1_{\mathbb{F}_q}\) to \(1_{\Z}\). Identify \(\mathbb{Z}^{n-2}\) with \(1\times (n-2)\) matrices (row vectors) and let \(A:\mathbb{Z}^{n-2} \to \mathbb{Z}^{n-2}\) be the linear transformation
\[
v \mapsto v\tilde{a}.
\]
In a similar fashion we may lift the matrix \(z\) or \(z+\e_{k-1,n-2}\) to a linear transformation \(Z:\mathbb{Z}^{n-2} \to \mathbb{Z}^{n-2}\).

In both cases a simple change of the standard generators shows that \((A+Z)(\mathbb{Z}^{n-2})=A(\mathbb{Z}^{n-2})\) is a direct summand. One also has that \(AZ= Z^2=0\) and so
\[
(A+Z)^k=(A+Z)A^{k-1}.
\]
It also follows that \((A+Z)^k(\mathbb{Z}^{n-2})\) is a direct summand and by Theorem~\ref{thm-jordan} has a Jordan normal form over \(\mathbb{Z}\).
%
%
%
%
%
\end{proof}

\begin{proof}[The Proof of Theorem~\ref{thm-motivic}]
Assume that \(a=\alpha \I+\bar{a}\), where \(\bar{a}=\bar{a}(\lambda)\)
as in (\ref{eq-lambda-2-b}).
By (\ref{eq-X-k-sum}) \(K_n(a)=\sum_{k=1}^n\sum_{x\in X_k} \psi(ax+x^{-1})\).
It is enough to prove that each of the sums \(\sum_{x\in X_k} \psi(ax+x^{-1})\) is expressible as a polynomial in \(q,q-1\) and \(K=K_1(\alpha)\).This was already established when \(\bar{a}=0\), so we may assume, that \(\varepsilon_{n-1}\neq 0\).

Proposition~\ref{pro-triv-cell} and (i) of Theorem~\ref{thm-begin-red} take care of the cells \(X_n\) and \(X_{n-1}\). For the other cells we may refer to statements (ii) and (iii) in Theorem~\ref{thm-begin-red} which reduce the sum down to the case when \(k=N_i\) for one of the \(N_i=n_1+...+n_i\), in which case  Proposition~\ref{pro-intermed-red-jordan} gives (\ref{eq-red-try}). Induction on the rank and Proposition~\ref{pro-jordan-indep} then shows that the resulting  Kloosterman sums of rank \(n-1\) and \(n-2\)
can be expressed as polynomials in \(q,q-1\) and \(K\) independently of \(p=\chara \mathbb{F}_q\).

\end{proof}

It is possible to make this recursion more concrete, see subsection~\ref{sec-ex-recursion} for examples.

%

\section{Estimates}\label{sec-estimates}


\subsection{Kloosterman sums over Bruhat cells: reduction to involutions}\label{sec-borel+inv}

We will move on to setup the technical background for the proof of Theorem~\ref{thm-split-estimate}.
We will pursue a path which establishes both of the estimates in Theorem~\ref{thm-split-estimate}  as well as Theorem~\ref{thm-scalar-closed-form} simultaneously and which will also allow us to analyze these sums cohomologically.

This approach is based on the Bruhat decomposition of the algebraic group \(\GL_n\) as well as its specialization in the finite group $G=\mathrm{GL}_n(\mathbb{F}_q)$ with respect to the Borel subgroup \(B_n\) of upper triangular matrices, with
$$B_n=T_nU_n,$$ where $U_n$ is the algebraic subgroup of upper triangular unipotent matrices, and $T_n$ is the maximal torus (diagonal matrices). Since in this section \( n\) is fixed we will simply write \(B,T\) and \(U\). Let $W$ be the Weyl group of $\GL_n$, we then have
\begin{equation}\label{Bruhat}
G=\GL_n(\mathbb{F}_q)=\bigsqcup_{w\in W} C_w(\mathbb{F}_q), \qquad C_w(\mathbb{F}_q)=\Um(\mathbb{F}_q) wB(\mathbb{F}_q)
\end{equation}
with \( \Um =\{u\in U|w^{-1}uw\in U^{T}\}\), where $U^{T}$  is the unipotent subgroup of the opposite Borel subgroup of lower triangular matrices. It is clear from Gaussian elimination that for a field \(F\) all $x\in C_w(F)$ have a unique decomposition $x=uwb$ with $u\in \Um(F) $ and $b\in B(F)$. This remains true for the algebraic variety $\GL_n$ canonically, \cite[Chapter 8]{Springer}. (Alternatively, given any field \(F\) one may work over an algebraic closure of \(F\) say \(\overline{F}\) as in \cite{Borel} IV.14.12 Theorem (a). It is clear from the uniqueness that once representatives for \(W\) as permutation matrices are fixed, the decomposition above is invariant under any Galois automorphism of \(\overline{F}\) fixing \(F\). Moreover the same argument shows that the map \(\Um \times B\to C_w\) is an isomorphism of algebraic varieties, defined over \(F\).)

We will work with
\begin{equation}\label{K^w}
K_n^{(w)}(a,\mathbb{F}_q)=
\sum_{x\in C_w(\mathbb{F}_q)}\psi(ax+x^{-1}).
\end{equation}

As a first step we will prove in this section  that the above sum vanishes unless \(w^2=\I\) and analyze the cells \(C_w\) to simplify these sums for \(w^2=\I\).

The Weyl group, $W=W_n$ is isomorphic to \(S_n\), the permutation group on \(n\) letters. For later calculations we make this identification explicit as follows. For a permutation matrix \(w\) we associate the permutation \(\pi\) such that \(i=\pi(j)\) if \(w_{i,j}=1\). Conversely given \(\pi\), we let
\begin{equation}\label{w-def}
w_\pi=\sum_{j=1}^{n} \e_{\pi(j),j}
\end{equation}
so that \(w_{\pi_1}w_{\pi_2}=w_{\pi_1\pi_2}\). To ease reading the arguments that will follow we overload the notation and write \(w(i)\) for \(\pi(i)\) if \(w=w_\pi\). This convention leads to
\begin{equation}\label{w-act}
(gw)_{ij}=g_{i,w(j)}  \quad \text{ and } \quad (wg)_{ij}=g_{w^{-1}(i),j}
\end{equation}
for any $g\in G$, which will be frequently used without further mention.

\begin{pro}\label{w2-uj} Let \(a\in M_n(\mathbb{F}_q)\) be an upper triangular matrix, such that \(\det(a)\neq 0\), and assume $w\in W$ is such that $w^2\neq \I$. Then
\[
K_n^{(w)}(a,\mathbb{F}_q) = 0.
\]
\end{pro}

\begin{proof} Let \(i\) be minimal such that $i\neq w^2(i)$ and let \(j=w(i)\). Note that if \(k=w(j)\) then \(w^2(k)\neq k\) and so \(k>i\).

Consider now the one parameter subgroup $$\{ x_{i,j}(s)=\I +s \e_{i,j}
: \,s\in \mathbb{F}_q\} \subset B.$$
Clearly we have \(x_{ij}(s)B(\mathbb{F}_q)=B(\mathbb{F}_q)\) and so
\begin{multline*}
K_n^{(w)}(a,\mathbb{F}_q) = \sum_{\substack{u\in \Um(\mathbb{F}_q)\\ b\in B(\mathbb{F}_q)}}
\psi(auwb+(uwb)^{-1}) =
\frac{1}{q}\sum_{s\in \mathbb{F}_q} \sum_{\substack{u\in \Um(\mathbb{F}_q)\\ b\in B(\mathbb{F}_q)}} \psi(auwx_{i,j}(s)b+(uwx_{i,j}(s)b)^{-1}).
\end{multline*}

Note that
\[
\frac{\partial}{\partial s}\Tr(auwx_{i,j}(s)b) = \Tr(bauw \e_{i,j})=(bauw)_{j,i}= (bau)_{j,j}
\]
since \(j=w(i)\). By the assumption \(\det(a)\neq 0\) we have \((bau)_{j,j}\neq 0\).

On the other hand
\begin{multline*}
\frac{\partial}{\partial s}\Tr(b^{-1}x_{i,j}(-s)w^{-1}u^{-1})=
-\Tr(\e_{i,j}w^{-1}u^{-1}b^{-1})=
-(w^{-1}u^{-1}b^{-1})_{j,i}=-(u^{-1}b^{-1})_{w(j),i}=0
\end{multline*}
since \(w(j)>i\). Since $\psi(b^{-1}x_{i,j}(-s)w^{-1}u^{-1})$ is linear in $s$, it is equal to $\psi(b^{-1}w^{-1}u^{-1})$ for any $s$.

Writing $\psi(auwx_{i,j}(s)b)=\psi(auwb)\psi(auwe_{i,j}bs)$ and using that $\psi$ is invariant under conjugation we have $\psi(auwe_{i,j}bs)=\psi(bauwe_{i,j}s)=\psi((bau)_{j,j}s)$. This gives
\begin{multline*}
\sum_{s\in \mathbb{F}_q} \psi(a uwx_{i,j}(s)b+b^{-1}x_{i,j}(-s)w^{-1}u^{-1})=
\psi(a uwb+b^{-1}w^{-1}u^{-1})\sum_{s\in \mathbb{F}_q} \psi((bau)_{j,j}s)=0.
\end{multline*}

\end{proof}

We can also prove Theorems~\ref{thm-split-semisimple} and \ref{thm-nil} this way. For example for nilpotent matrices we may  prove that $K_n^{(w)}(a)=0$, unless $w=\I$. First we may assume that $a$ is strictly upper triangular.

Let $j=\max(k|w(k)\neq k)$ and $i=w(j)<j$ and consider the one parameter subgroup $\{x_{i,j}(s)=\I+s\e_{i,j}|s\in\mathbb{F}_q\}\leq B(\mathbb{F}_q)$. We have $x_{i,j}(s)B(\mathbb{F}_q)=B(\mathbb{F}_q)$ and so
\begin{multline*}
K_n^{(w)}(a) = \sum_{\substack{u\in U_w(\mathbb{F}_q)\\
b\in B(\mathbb{F}_q)}}
\psi(auwb+(uwb)^{-1}) =
\frac{1}{q}\sum_{s\in \mathbb{F}_q} \sum_{\substack{u\in U_w(\mathbb{F}_q)\\
b\in B(\mathbb{F}_q)}}
\psi(auwx_{i,j}(s)b+(uwx_{i,j}(s)b)^{-1}).
\end{multline*}

Note that by the definition of $j$ for any $u\in U_w(\mathbb{F}_q)$ and $b\in B(\mathbb{F}_q)$ \[\frac{\partial}{\partial s}\tr(auwx_{i,j}(s)b)=\tr(\e_{i,j}bauw)=0.\]

On the other hand
\[\frac{\partial}{\partial s}\tr((uwx_{i,j}(s)b)^{-1})=\tr(\e_{i,j}w^{-1}u^{-1}b^{-1})=b_{i,i}^{-1}\neq0.\]

This gives

\begin{equation*}
\sum_{s\in \mathbb{F}_q} \psi(a uwx_{i,j}(s)b+b^{-1}x_{i,j}(-s)w^{-1}u^{-1})=
\psi(a uwb+b^{-1}w^{-1}u^{-1})\sum_{s\in \mathbb{F}_q} \psi(b^{-1}_{i,i}s)=0.
\end{equation*}

So
\begin{multline*}
K_n(a,\mathbb{F}_q)=K_n^{(\I)}(a,\mathbb{F}_q)= \sum_{b\in B(\mathbb{F}_q)} \psi(ab+b^{-1}) =
\sum_{b\in B(\mathbb{F}_q)}\psi(b^{-1})= \\
|U|\sum_{t\in T(\mathbb{F}_q)} \prod_{i=1}^n \varphi(t_{i,i}^{-1}) = (-1)^nq^{n(n-1)/2}
\end{multline*}
as $\psi(b^{-1}) = \psi(t^{-1}) = \displaystyle{\prod_{i=1}^n} \varphi(t_{i,i}^{-1})$ if $b=tu\in TU=B$.


\subsection{Finer decomposition of individual Bruhat cells}\label{sec-N_w}

In this section we will give a decomposition of the algebraic group \(U\) of unipotent upper triangular matrices in \(\GL_n\). To show that the underlying maps are morphisms we will work
over a general commutative ring \(R\). Therefore the letter \(U\)  will denote the algebraic group itself, and not its set of points \(U(\mathbb{F}_q)\). Similarly further subsets of \(U\) denoted
with various markings will define affine sub-varieties of \(U\) and not their set of points in \( \mathbb{F}_q \).

The motivation for this refinement of the Bruhat decomposition is as follows. The Bruhat cells \( C_w(\mathbb{F}_q)=\Um(\mathbb{F}_q) wB(\mathbb{F}_q)\) are already of the form \( \mathbb{F}_q^d \times (\mathbb{F}_q^*)^e\) when using the entries of the matrices in \(\Um(\mathbb{F}_q)\), \(T(\mathbb{F}_q)\) and \(U(\mathbb{F}_q)\) as coordinates, and exponential sums over such spaces have a well established theory  \cite{Adolphson-Sperber,Denef-Loeser}.
Also note that \(\tr (ax+x^{-1})\) as a function on \(\mathbb{F}_q^d\) (using the entries as coordinates of \(x\)) is degree 1 in each of these variables. However they have to be collected the right way to use this observation. Therefore we will fiber up the cells \(C_w\) into finer (affine) subspaces that are more suitable for this purpose.
For later arguments that rely on cohomology it will be important that this refinement gives an isomorphism of affine varieties. While for this purpose one could restrict to \(\mathbb{F}_q\)-algebras, the results are true over an arbitrary commutative ring. Also this refinement has potential applications elsewhere and so we first set it up for a general element \(w\in W\) which is assumed to be fixed.

This section mainly consists in developing a nomenclature for these simple but numerous subsets (many of them subgroups). These subsets are defined under various actions of \(W\) on the affine variety \(\Ubar\) of strictly  upper triangular nilpotent matrices, whose set of points in any ring is
\begin{equation}\label{N-def}
\Ubar(R)=\{ y\in \Mn(R)  \,|\, \I+y\in U(R)\}
\end{equation}
Clearly \(\Ubar\) is isomorphic to \(\mathbb{A}^{n(n-1)/2}\) as an algebraic variety.

The subsets we are going to define depend on \(w\) but we will not work with more than one \(w\) at a time, so we drop $w$ from the notation. For example
the subgroups
\[
\Um =\{u\in U|w^{-1}uw\in U^T\} \quad  \text{ and }\quad  \Up =\{u\in U|w^{-1}uw\in U\}\] that satisfy \(\Up\cap\Um=\{\I\}\) and \(U=\Up\Um=\Um\Up\) can be defined as \[\Um=\I+\Umbar, \quad \Up=\I+\Upbar\] where
\[
\Umbar= \Ubar\cap w\Ubar^T w^{-1} \quad \text{ and } \quad \Upbar = \Ubar \cap w\Ubar w^{-1}.
\]
The further refinement  comes from exploiting the action of \(W\), the group of permutation matrices on
\(M_n\) via left multiplication and so it is tied to the standard representation of \(\GL_n\).
\begin{df}\label{def-Ubar-Ubar-a}
For a fixed \(w\in W\) let
\begin{equation}\label{eq-N_a}
\Ubar_a=\Ubar\cap w^{-1}\Ubar, \quad  	\Ubar_b=\Ubar\cap w^{-1}\Ubar^T \text{ and } 	\Ubar_o=\Ubar\cap w^{-1}D
\end{equation}
where \(D\) is the set of possibly singular diagonal matrices . Using these sub-varieties define
\begin{equation}\label{eq-U_a}
\Ua=\I+\Ubara, \quad  	\Ub=\I+\Ubar_b \quad \text{ and } 	U_o=\I+\Ubar_o.
\end{equation}
\end{df}
\begin{rem}
The subscripts ``\(a\)'', ``\(b\)'' and ``\(o\)''  denote the fact that these are
the elements \(\bar{u}\in \Ubar\) for which the non-zero entries in \(w\bar{u}\) are strictly ``above'',  ``below'' or ``on'' the diagonal.
Clearly for any ring  \(\Ubar(R)=\Ubar_a(R)\oplus \Ubar_b(R) \oplus \Ubar_o(R)\).
\end{rem}

We will see below that the map \(x\mapsto\tr(ax+x^{-1})\) is linear on \(\Ubara\) which leads to immediate cancellations. However for this purpose we first need to setup further notation.

\begin{df}\label{def-finer-subgroups}
Let
\[
\Up_a=\Up\cap \Ua, \quad \Up_b=\Up\cap \Ub,\quad \Um_a=\Um\cap \Ua, \quad \Um_b=\Um\cap \Ub.
\]
\end{df}

\begin{rem}\label{rem:J-set-def}
These sub-varieties are defined via their non-vanishing entries. In general let \[\cI=\{(i,j)\,|\, 1\leq i < j \leq n\},\] and for \(\cJ\subset \cI\) define the affine variety
\[\Ubar_{\cJ}=\{\bar{u}\in \Ubar\, | \, \bar{u}_{i,j}\neq 0 \implies (i,j)\in \cJ\}.\]
Note that \(U_\cJ=\I+\Ubar_{\cJ}\) is an algebraic subgroup of \(U\) if and only if \(\cJ\) is transitive in the sense that
\[
(i,j), (j,k) \in \cJ \implies (i,k) \in \cJ.
\]

For example \(\Ua=U_{\cJ_a}\) where \(\cJ_a=\{(i,j) \in \cI \,|\, w(i)<j \}\), which is transitive, and so \(\Ua\) is a subgroup.
\end{rem}

\begin{ex}
Consider the case \(n=5\). We indicate below the indices with the different properties for two involutions.
\[
w=(14)(25) \text{ gives }:\begin{bmatrix}
    & b& b& o& a\\
    & & b& b& o\\
    & & & a& a\\
    & & & & a\\
    & & & &
\end{bmatrix},\quad\text{while } w=(15)(23) \text{ gives }:\begin{bmatrix}
    & b& b& b& o\\
    & & o& a& a\\
    & & & a& a\\
    & & & & a\\
    & & & &
\end{bmatrix}
\]
\end{ex}


To state the main result of this section we need to introduce one more piece of notation. Let \[\Ubar_{b/o}=\Ubar_b \oplus \Ubar_o,\]
be the subset of \(\Ubar\) of elements \(\bar{u}\) for which \(w\bar{u}\) has non-zero elements only below or on the diagonal, and let \[
\Up_{b/o}=\I+\Upbar_{b/o}, \quad \text{ and } \Um_{b/o}= \I+ \Umbar_{b/o}
\]
where \( \Upbar_{b/o}=\Upbar \cap \Ubar_{b/o}\),  and \(\Umbar_{b/o}= \Umbar \cap \Ubar_{b/o}\).

\begin{pro}\label{pro-ref-Bruhat}
\begin{enumerate}
\item \(\Ua\), \( \Up\Ua \) and \(\Um_{b/o}\) are algebraic subgroups of \(U\) and \( \Up\Ua \cap \Um_{b/o}=\{\I\}\).
\item The morphism
\begin{align*}
\Up_{b/o} \times \Ua \times \Um_{b/o}  \, \to& \;\quad   U\\
(u_1,y,u_2)  \quad \, \mapsto &\; u_1 y u_2
\end{align*}
is an isomorphism.
\item If \(w^2=\I\), then \(U_o \subset \Um\) is a subgroup, and the morphism
\begin{align*}
\Upb \times U_{a} \times U_o \times \Umb \, \to& \;\quad  U\\
(u_1,y,u_o, u_2) \quad \, \mapsto &\; u_1 y u_o u_2
\end{align*}
is an isomorphism.
\end{enumerate}

\end{pro}

There are many alternative versions of the statement in (i) as for
example \( U_{a/o}=\Ua U_o = U_o\Ua\) is a  subgroup
of \(U\) as well. Also note that if one is merely interested in a bijection over \(\mathbb{F}_q\), the statements in (ii) and (iii) can easily be proved by a counting argument. One may argue
similarly to see that the set \(\Up_{b/o}\) is a complement of \(\Ua\)
in  \(\Up \Ua\). To prove that these maps are isomorphisms  it is possible to adapt the reasoning of Lemma 8.2.2 in \cite{Springer} but given the concrete nature of the statement we give a self-contained proof here, based on the fact that affine varieties and their maps are determined by their functor of points.

\begin{lem}\label{lem-U_a}
We have the following
\begin{enumerate}
\item\label{lem-U_a-(1)} If \(u\in U(R)\) and \(x\in  \Ubara(R)\) then \(xu \in  \Ubara(R)\).

\item\label{lem-U_a-(2)} If \(u\in  \Up(R) \) and \(x\in  \Ubara(R)\) then \(ux \in  \Ubara(R)\).

\item\label{lem-U_a-(3)} If \(u_1\in  \Up(R)\) and \(u_2\in \Um(R) \) then
\[
u_1u_2+ \Ubara(R)	= u_1 \Ua(R) u_2.
\]
\end{enumerate}
\end{lem}

\begin{proof}
Since \(\Ubar_a(R)=\Ubar(R) \cap w^{-1}\Ubar(R)\), the first and second claims are obvious from the fact that \(U(R) \Ubar(R)=\Ubar(R) U(R)=\Ubar(R)\), and that for \(u \in \Up(R)\),  \( w^{-1}uw \in U(R) \).

By the first two claims, if \(u_1\in  \Up(R) \) and \(u_2\in \Um(R) \) then \(u_1 \Ubar_a(R) u_2= \Ubar_a(R)\), from which the third claim is obvious.
%
\end{proof}

\begin{proof}[The Proof of Proposition~\ref{pro-ref-Bruhat}]
It is enough to prove that  for any commutative ring \(R\) the sets \(\Ua(R)\), \( \Up(R)\Ua(R) \) and \(\Um_{b/o}(R)\) are subgroups of \(U(R)\), and that the maps
\begin{align*}
\Up_{b/o}(R) \times \Ua(R) \times \Um_{b/o}(R)  \, \to& \;\quad   U(R)\\
(u_1,y,u_2)  \quad \, \mapsto &\; u_1 y u_2
\end{align*}
and
\begin{align*}
\Upb(R) \times U_{a}(R) \times U_o(R) \times \Umb(R) \, \to& \;\quad  U(R)\\
(u_1,y,u_o, u_2) \quad \, \mapsto &\; u_1 y u_o u_2
\end{align*}
are bijections.

Define the following equivalence relation on \(U(R)\):
\begin{equation}\label{Nw-equi}
u_1\sim_{\Ubar_a} u_2\ \iff u_1-u_2\in  \Ubara(R).
\end{equation}
We start with the proof of the claim about the group property of the three sets in (i). By (\ref{lem-U_a-(1)}) of the previous lemma, if \(u_1 \sim_{\Ubar_a} u_2\), and \(u \in U(R)\), then \( u_1u \simeq_{\Ubar_a} u_2u\). Therefore \(\Ua(R)\), the stabilizer of the equivalence class of the identity \(\I\),
is a subgroup.

By (\ref{lem-U_a-(2)}) of the same lemma \(\Up(R)\) normalizes \(\Ua(R)\) and so \(\Up(R)\Ua(R)\) is a subgroup. Finally  note that
\[
\cJ=\{ (i,j) \, | \, j\leq w(i), w(j) \leq  w(i)\}
\]
and
\[
\cJ'=\cI\setminus \cJ= \{ (i,j)\in \cI\,|\, w(i) < j \text { or } w(i)<w(j)\},
\]
are disjoint transitive subsets. This shows that \(\Um_{b/o}(R)=U_\cJ(R)\) is a subgroup, and since \( \Up(R), \Ua(R) \subset U_{\cJ'}\) we have that \( \Up(R) \Ua(R) \cap \Um_{b/o}(R) = \{\I\}\).

Now to prove the claim that every \(u\in U(R)\) can be represented in a unique way as
\[
u=u_1y u_2, \quad u_1 \in \Up_{b/o}(R), \;y\in U_a(R), \; u_2 \in \Um_{b/o}(R)
\]
we will first show that \(U(R)=\Up_{b/o}(R) \Ua(R)  \Um_{b/o}(R)\). Let \(u\in U(R)\). We know (\cite[Chapter 8., Proposition 8.2.1]{Springer}) that there are \(v_1\in  \Up(R) \) and \(v_2\in \Um(R) \)
such that \(u=v_1v_2\).

Let \(u_1\in \Up_{b/o}(R)\) be the matrix whose entries agree with \(v_1\) for \((i,j)\in \cI\) when \(w(i)\geq j\) and are 0 otherwise.
Then we have that
\[
v_1\in u_1+\Ubar_a(R), \text{ and so } v_1 \Ua(R)= u_1 \Ua(R).
\]
Similarly let \(u_2\in  \Umb(R)\) be such that \(v_2\in u_2+ \Ubara(R)\), so that \(\Ua(R) u_2 = \Ua(R) v_2\).
We have that
\[
u = v_1 v_2 \in v_1 \Ua(R) v_2=  u_1 \Ua(R) u_2.
\]

Now for the injectivity assume that \(u_1,u_1'\in \Up_{b/o}(R)\), \(y,y'\in \Ua(R), \) and \(u_2,u_2' \in \Um_{b/o}(R)\) are such that
\[
u_1 y u_2 = u_1' y' u_2 '.
\]
Since \( \Up(R) \Ua(R) \cap \Um_{b/o}(R) = \{\I\}\) we have that \(u_2=u_2'\) and so, \(u_1y=u_1'y'\). By (\ref{lem-U_a-(3)}) of the previous Lemma \(u_1\sim_{\Ubar_a} u_1'\) which can only happen in \(\Up_{b/o}(R)\) if \(u_1=u_1'\).

The very last claim about the case \(w^2=\I\) is elementary.
\end{proof}

For convenience we will parameterize \(C_w\) using the following lemma.
\begin{lem}\label{lem-Bruhat-param} Assume that \(w^2=\I\). Then \(U_o=\Umd\) and the morphisms
\begin{enumerate}
\item \(\Um \times T\times U \to C_w, \quad
(v,t,u)\mapsto vwtu v^{-1}
\)
\item \( \Upb\times \Ubara \times \Umd\times \Umb \to U, \quad (u_1,y,u_o,u_2)\mapsto u_1(\I+y)^{-1}u_o u_2
\)
\end{enumerate}
are isomorphisms.
%
\end{lem}
\begin{proof}
The first claim follows from (\ref{Bruhat}) and the fact that the morphism
\(\Um \times T\times U \to \Um \times T\times U\), \(	(v,t,u)\mapsto (v,t,u v^{-1}) \)
is an isomorphism of varieties. The second claim is merely a restatement of (iii) of Proposition \ref{pro-ref-Bruhat}.
\end{proof}
The advantage of the parameterization in (i) is obvious: if
\(x=vwtuv^{-1}\in C_w(\mathbb{F}_q)\) then
\[
\tr(a x +x^{-1})=\tr (a^v wtu+(wtu)^{-1})
\]
where \(a^v=v^{-1}{a}v=\alpha \I+\bar{a}^v\), with \(\bar{a}^v=v^{-1}\bar{a}v\) still strictly upper triangular. 


\subsection{Some trace calculations}\label{sec-trace-calc}

From now on we will specify to the finite field \(\mathbb{F}_q\) and so \(B, U, \Ubar\) etc. will mean \(B(\mathbb{F}_q), U(\mathbb{F}_q), \Ubar(\mathbb{F}_q)\) etc. Assume again that $a$ is a matrix with a unique eigenvalue $\alpha\in\mathbb{F}_q^*$,
\begin{equation}\label{strictly-upper-b}
a=\alpha\I+\bar{a}\in M_n(\mathbb{F}_q),
\text{ where } \bar{a}\text{ is strictly upper triangular.}
\end{equation}

With (ii) of \ref{lem-Bruhat-param} it is now easy to bound 
\(
\sum_{x\in C_w} \psi(ax+x^{-1}).
\)
However the cohomological methods used in the proof of Theorem~\ref{thm-semisimple-purity} and \ref{thm-general-bound} require repeated use of Lemma~\ref{thm-Hfg}, which in turn relies on understanding
\[
x \mapsto \tr(ax+x^{-1})
\]
itself as a function on \( C_w\). This is achieved in Propositions~\ref{pro-Mw-sum} and \ref{pro-U^flat_d} whose proofs will make up the rest of this section.

\begin{pro}\label{pro-Mw-sum}
Assume that \(w\in W\) satisfies \(w^2=\I\) and let \( \Ubara, \Upb,  \Umb, U_o\)
be as in Definitions~\ref{def-Ubar-Ubar-a} and \ref{def-finer-subgroups} with values in \(\mathbb{F}_q\).

Assume that \(a^v=\alpha \I + \bar{a}'\in M_n(\mathbb{F}_q)\) with \(\alpha \in \mathbb{F}_q^*\), \(\bar{a}' \in \Ubar \),
and that \(u_1 \in \Upb\), \(u_o \in {U_o}\)
and \(u_2 \in  \Umb\). If  \( y \in \Ubara \), let \(g(y)=u_1(\I+y)^{-1} u_o u_2\). Then the function
\begin{equation*}
y\mapsto  \tr(a^v  wt g(y)  + (wt g(y))^{-1})
\end{equation*}
is affine linear on \(\Ubara\). This map is non-constant unless \(u_2 = \I\).

When the map is constant its value is
\[
\tr(a^v wt u_1 u_o  + (wtu_1 u_o)^{-1}).
\]

%

\end{pro}

Further reductions are given by

\begin{pro}\label{pro-U^flat_d}
Assume \(w\in W\) satisfies \(w^2=\I\), and that
\( a'=\alpha \I + \bar{a}'\in M_n(\mathbb{F}_q)\) with \(\alpha \in \mathbb{F}_q^*\), \(\bar{a}' \in \Ubar \). Let \(u_1\in \Upb\) and  \(t=\diag(t_1,...,t_n) \in T\). For  \(u_o=\I+\bar{u}_o \in {U_o}\)
the map
\begin{equation*}
u_o \mapsto \tr( a' wt u_1(\I +\bar{u}_o) + (wtu_1(\I +\bar{u}_o))^{-1})
\end{equation*}
is affine linear on $U_0$
. This map is non-constant unless \(t_{w(i)}^{-1}=\alpha t_i \) for all \(i\neq w(i)\). When the map is constant its value is

\[
\sum_{i=w(i)} (\alpha t_i+ t_i^{-1}) + \tr(\bar{a}'wdu_1)
\]
where \(d=\sum_{i<w(i)} t_i \e_{i,i}\).
\end{pro}

The proof of the above facts rely on some simple lemmas that we present first.
\begin{lem}\label{lem-y-not-in-b'}
If \({a}'\)  is  upper triangular, \(y \in  \Ubara \),  \(u_1\in \Up\) and \(u_2\in U\), 
then
\begin{equation}
\tr({a}' wtu_1 (\I+y)^{-1} u_2)=\tr({a}'wt u_1 u_2)
\end{equation}
is independent of \(y\).
\end{lem}
\begin{proof} Note that by Proposition~\ref{pro-ref-Bruhat} 
\((\I+y)^{-1}=\I+y'\) for some \(y'\in \Ubara\) and so \(wy'\) is strictly upper triangular.
By the assumptions on \(a'\), \(u_1\) and \(u_2\) we have that \(a'\), \(wtu_1w^{-1}\) and \(u_2\) are upper triangular,
so
\[
\tr({a}'wt u_1 y' u_2)=\tr({a}' (wt u_1 w^{-1}) (wy') u_2)=0.
\]
\end{proof}

\begin{lem}\label{lem-lin-aut}
For any \(u \in U, t \in T\)  
\[
y \mapsto wyu^{-1}t^{-1}
\]
is a linear automorphism of \( \Ubara \).
\end{lem}
\begin{proof}
The map \( y \mapsto wyu^{-1}t^{-1} \) is clearly linear, and a bijection to its image. Since \(w \Ubara = \Ubara \) and \( \Ubara ut= \Ubara \) this image is \(\Ubara\).
\end{proof}

\begin{proof}[The Proof of Proposition~\ref{pro-Mw-sum}]
First by Lemma~\ref{lem-y-not-in-b'} we have
\begin{equation*}
\tr(a^v wt g(y)  + (wt g(y))^{-1})=\tr(a^vwtu_1u_o u_2)+\tr((wtu_1 u_o u_2)^{-1})+
\tr((u_2 u_o)^{-1}yu_1^{-1}t^{-1}w)
\end{equation*}
showing instantly that the map $y\mapsto  \tr(a^v wt g(y)  + (wt g(y))^{-1})$ is affine linear. Also
\[
\tr((u_2 u_o)^{-1}yu_1^{-1}t^{-1}w)=\tr(w(u_2u_o)^{-1}w y')
\]
where \(y'= wyu_1^{-1}t^{-1}\in U\), and so by Lemma~\ref{lem-lin-aut} it is enough to show that the map
\[
y'\mapsto \tr(w(u_2u_o)^{-1}w y')
\]
is non-constant unless \(u_2 = \I\).

Assume that \(u_2\neq \I\), and let \(z=(u_2 u_o )^{-1}\). Since \( \Umb\) and \(\Umd\) are subgroups, \(z\) is in \( \Umb \Umd\), but not in \(\Umd\) and so there exist \((i,j)\) such that
\[
i<j, \quad w(i)> j \text{ and } z_{i,j}\neq 0.
\]

Also note that by the fact that \( \Umb\subset \Um \) we have that \(w(i)>w(j)\),
and so
\[
\e_{w(j),w(i)}\in  \Ubara .
\]
Now let
\[
Y_0=\{ y\in  \Ubara : y_{w(j),w(i)}=0\}
\] so that any element \(y\in  \Ubara \) may be written as  \(y=y_0+s\e_{w(j),w(i)}\), where \(y_0\in Y_0\). Then
\begin{equation*}
\tr(w z w(y_0+s\e_{w(j),w(i)} ))=
\tr(w z wy_0)+ s\,\tr(w z w\e_{w(j),w(i)})=
\tr(w z wy_0) + s z_{i,j}.
\end{equation*}
This proves the proposition.
\end{proof}

For the  proof of Proposition~\ref{pro-U^flat_d} we need an explicit evaluation that we present in the following

\begin{lem}\label{lem-T-separate}
Assume \(w\in W\) satisfies \(w^2=\I\), \(\alpha \in \mathbb{F}_q^*\), \(\bar{a}' \in \Ubar \) and that \(t=\diag(t_1,...,t_n) \in T\) satisfies \(t_{w(i)}^{-1}=\alpha t_i\) for all \(i<w(i)\). Then
\[
\tr(w(\alpha t+t^{-1}))=\sum_{i=w(i)} \alpha t_i+ t_i^{-1} .
\]
Moreover if \(u \in U\), then
\[
\tr(\bar{a}'wtu)=\tr(\bar{a}'wdu)
\]
where \(d=\sum_{i<w(i)} t_i \e_{i,i}\).
\end{lem}
\begin{proof}
First observe that
\[
w t = \sum_{i=w(i)} t_i \e_{i,i}+\sum_{i<w(i)} (t_i \e_{w(i),i}+(\alpha t_i)^{-1}\e_{i,w(i)} ),
\]
from which the first claim follows immediately. The second is a slight variant, using that \(\bar{a}'_{i,i}=0\), from the assumption \(\bar{a}' \in \Ubar \).

\end{proof}

\begin{proof}[The Proof of Proposition~\ref{pro-U^flat_d}]
It is easy to check that \(U_o=\Umd=\{I+\sum_{i<w(i)} s_i \e_{i,w(i)}|s_i\in\mathbb{F}_q\}\) and it is abelian. Using that we have that \((\I+ \bar{u}_o)^{-1}=\I- \bar{u}_o\). Therefore
\begin{multline*}
\tr(a' wt u_1 (\I + \bar{u}_o ) + (wt u_1 (\I + \bar{u}_o ))^{-1})= \\
\tr(a' wtu_1 + (wtu_1)^{-1})+
\tr((a' wtu_1-(wtu_1)^{-1})  \bar{u}_o)
\end{multline*}
clearly affine linear as a function of \(\bar{u}_o\). To see when it is non-constant write \( \bar{u}_o\) as \( \sum_{i<w(i)} s_i \e_{i,w(i)}\) so that by the conventions in (\ref{w-act})
\[
\bar{u}_o w=\sum_{i<w(i)} s_i \e_{i,i} \text{ and } w \bar{u}_o  = \sum_{i<w(i)} s_i \e_{w(i),w(i)}.
\]
This leads to
\begin{equation*}
\tr( a' wtu_1  \bar{u}_o)=
\tr\left( (a' wtu_1 w)w \bar{u}_o \right)=
\sum_{i<w(i)} s_i (a' wtu_1w)_{w(i),w(i)}=\sum_{i<w(i)} s_i \alpha t_{i}
\end{equation*}
since \(wu_1w\in U\), and \(w\diag(t_i)w=\diag(t_{w(i)})\). In a similar manner
\begin{equation*}
\tr((wtu_1)^{-1} \bar{u}_o)=\tr( (t u_1)^{-1}  w \bar{u}_o) = \sum_{i<w(i)}  s_i ((tu_1)^{-1})_{w(i),w(i)}=\sum_{i<w(i)} s_i t_{w(i)}^{-1}.
\end{equation*}

Finally, since \( u_1 \in \Upb \subset  \Up \), \(u_1^{-1}\in  \Up \) as well, and for any \(u \in  \Up \) we have
\[
u_{i,w(i)}=0 \text{ in case } i<w(i).
\]
Hence for \(u = \I + \bar{u}\in  \Up \), \(\tr(wt \bar{u})= 0\), and  \(\tr(tw \bar{u})=0\). 
This gives
\[ \tr(a' wt u_1 + wu_1 ^{-1}t^{-1})=\tr(\alpha wt + wt^{-1}) + \tr(\bar{a}'wtu_1).
\]

This finishes the proof of the proposition in view of Lemma~\ref{lem-T-separate}.
\end{proof}

\subsection{The Proofs of Theorems~\ref{thm-scalar-closed-form} and \ref{thm-split-estimate}}\label{sec-pf-thm-split-estimate}

In this section we again interpret the notation for all affine varieties as the set of their \( \mathbb{F}_q\)-rational points, so \( C_w\) stands for \( C_w(\mathbb{F}_q)\) etc. Recall that
\[
K_n^{(w)}(a)= \sum_{x\in C_w} \psi(ax+x^{-1}).
\]

\begin{pro} Assume that \( a=\alpha \I + \bar{a}\in M_n(\mathbb{F}_q)\) with \(\alpha \in \mathbb{F}_q^*\), \(\bar{a} \in \Ubar \). Then
\[
K_n^{(w)}(a)= q^{n_a} \sum_{v}\sum_{t,u_o,u_1}  \psi(a^vwtu_1u_o +(wt u_1u_o)^{-1})
\]
where \(v \in \Um\), \(a^v=v^{-1}av\), \(t\in T\),
{$u_o\in U_o$},
\(u_1\in \Upb\) and
\[
n_a=\dim \Ua=|\{ \cJ_a \}|=|\{(i,j)\,|\, i, w(i)<j \}|.
\]
\end{pro}

\begin{proof}
By Lemma~\ref{lem-Bruhat-param}	(ii)
\[
K_n^{(w)}(a)=\sum_{v}\sum_{t,u_o,u_1,u_2} \sum_{y} \psi(a'wtg(y)+(wtg(y))^{-1}),
\]
where $g(y)=u_1(I+y)^{-1}u_0u_2$ and the inner sum is
\[
\sum_{y} \psi(a'wtg(y)+(wtg(y))^{-1})=\begin{cases}
0 & \text{  if } u_2\neq \I \\
q^{n_a}\psi(a'wtu_1u_o +(wt u_1u_o)^{-1}) & \text{ if } u_2=\I
\end{cases}
\]
by Proposition~\ref{pro-Mw-sum}.

\end{proof}
\begin{pro}\label{pro-K_n^w-eval}
\[
K_n^{(w)}(a)= q^{n_a+n_o} K_1(\alpha)^f \sum_{v,d,u} \psi(\bar{a}^v wdu)
\]
where \(v \in \Um\), \(u \in \Upb\), \(\bar{a}^v=v^{-1}\bar{a}v\), \(n_o=e=\vert \{i\,|\, i<w(i)\}\vert\) is the number of involution pairs in \(w\), \(f=|\{i\,|\,w(i)=i\}|\) is the number of fixed points of \(w\) and \(d\in D(w)=\{ \sum_{i<w(i)} t_i \e_{i,i}\,|\, t_i\in \mathbb{F}_q^*\}\).
\end{pro}
\begin{proof}
Let \( T(w)=\{t \in T\, | \, t_{w(i)}^{-1}=\alpha t_i \text{  if } i<w(i) \} \)
By Proposition~\ref{pro-U^flat_d}
\[
\sum_{u_o}  \psi(a^v wtu_1u_o +(wt u_1u_o)^{-1})=	0
\]
unless \(t \in T'(w)\), in which case
\[
\sum_{u_o}  \psi(a^v wtu_1u_o +(wt u_1u_o)^{-1})=q^{n_o}\prod_{i=w(i)}\phi(\alpha t_i + t_i^{-1}) \psi(\bar{a}^v wdu_1),
\]
with \(d =\sum_{i<w(i)} t_i \e_{i,i}\).

The proposition follows after summing over \(t \in T(w)\).
\end{proof}

\begin{proof}[The Proof of Theorem~\ref{thm-scalar-closed-form}]
We will show that for \(\alpha \neq 0\) and \(w^2=\I\) we have
\begin{equation}\label{cor-scalar-explicit}
K_n^{(w)}(\alpha \I)=q^{n(n-1)/2+N}(q-1)^{e} K_1(\alpha)^{f} \end{equation}
where \(N = n^\flat_{a/o}= \dim(\Uma)+\dim(\Umd)\).

Since \(\bar{a}=0\) now, \(\bar{a}^v=v^{-1}\bar{a}v=0\) as well, and so  \(\psi(\bar{a}^v wdu_1)=\psi(0)=1\). To get the exponent of \(q\) note that $n_a+n_o+n^\flat+n_b^\sharp=n_a+n_o+n_b+n^\flat_{a/o}$, where these are denoting the dimension of the corresponding subspaces of $U$. Then we also have $n_a+n_o+n_b=n(n-1)/2$.
\end{proof}



\begin{proof}[The Proof of Theorem~\ref{thm-split-estimate}]
\quad

\noindent
If \(a=\alpha \I +\bar{a}\), \(\bar{a}\in \Ubar\), \(w^2=\I\), then by Proposition~\ref{pro-K_n^w-eval}
\[
|K_n^{(w)}(a)| =q^{n_a+n_o}\left|K_1(\alpha)^f\right|\left|\sum_{v,d,u}\psi(\bar a^vwdu)|\right|\leq q^{n_a+n_o}\left|K_1(\alpha)^f\right|\sum_{v,d,u}\left|\psi(\bar a^vwdu)\right|= |K_n^{(w)}(\alpha \I)|
\]
since
\(|\psi(\bar{a}^v wdu_1)|= 1\).

Since \(w^2=\I\) every element is either fixed by \(w\) or is in an involution pair, and so \(n=f +2 n_o\). So while \(f\) depends on \(w\)
\[
K_1(\alpha)^f=K_1(\alpha)^{n} 	K_1(\alpha)^{-2 n_o}=
\sign(K_1(\alpha))^n |K_1(\alpha)|^{n} 	K_1(\alpha)^{-2 n_o} = \varepsilon |K_1(\alpha)|^f,
\]
with the sign
\[\varepsilon = \left( \sign( K_1(\alpha)\right)^n\] independent of \(w\). Here we have used the fact that \(K_1(\alpha)\) is real. Thus we immediately have that
\( |K_n^{(w)}(\alpha \I)|= \varepsilon K_n^{(w)}(\alpha \I)\). Therefore
\begin{multline*}
|K_n(a)|\leq \sum_{w^2=\I} \left| K_n^{(w)}(a) \right|	\leq \sum_{w^2=\I} \left| K_n^{(w)}(\alpha \I) \right|
=
\varepsilon  \sum_{w^2=\I} K_n^{(w)}(\alpha \I) =
\varepsilon K_n(\alpha \I)= |K_n(\alpha \I)|.
\end{multline*}

It remains to prove that \(c_n=|K_n(\alpha \I)|/q^{{(3n^2-\delta(n))/4}}\leq 4\). By the recursion formula (\ref{recursion-eq}) we have
\[
c_{n+1} \leq c_n |K_1(\alpha)| q^{-n/2}q^{(\delta(n+1)-\delta(n)-3)/4}+c_{n-1}.
\]
Recall that $\delta(n)=\begin{cases}0, & \text{if }n\text{ is even}\\1, & \text{if }n\text{ is odd}\end{cases}$, thus $\delta(n+1)-\delta(n)=\begin{cases}1, & \text{if }n\text{ is even}\\-1, & \text{if }n\text{ is odd}\end{cases}$.

If \(n=2k\) this gives \(c_{2k+1}-c_{2k-1}\leq 2c_{2k}q^{-k}\), and so
\[
c_{2k+1}\leq c_1 + 2\sum_{j=1}^k c_{2j}q^{-j}.
\]
Similarly
\[
c_{2k} \leq 2\sum_{j=1}^{k-1}c_{2j+1}q^{-j}
\]
from which the claim follows easily for \(q\geq 3\). The case \(q=2\) is easily checked by hand.
\end{proof}
The sum \(\sum_{v,d,u}\psi(\bar{a}^v wd u)
\) gives rise to some intriguing questions on its own, see the problems mentioned sections \ref{sec-ex-borel-poly}.


\subsection{Review of cohomology}\label{sec-cohomology}

With the results of the previous section, for a fixed $a\in M_n$ the bounds can be proven over those extensions of $\mathbb{F}_q$ in which $a$ can be conjugated to Jordan normal form. However, to get the general result, we need to understand certain cohomology groups attached to the sum -- which are independent of the field extension. In the rest of Section we will view \ref{sec-estimates} a subset of the matrix group $X\subset M_n$ defined by algebraic equations of the matrix entries as the corresponding algebraic variety.

We first introduce the notation and the main tools, then prove cohomological versions of Propositions \ref{felsoblokk} and \ref{pro-K_n^w-eval}. This enables us to prove Theorems~\ref{thm-semisimple-purity} and \ref{thm-general-bound}.

Let $\ell\neq p$ be a prime, and $\overline{\mathbb{Q}}_\ell$ be an algebraic closure of the field $\mathbb{Q}_\ell$ of $\ell$-adic numbers, such that there is an $p$-th primitve root of unity $\zeta$ contained in  $\overline{\mathbb{Q}}_\ell$. Fix the field embedding $\iota_0:\mathbb{Q}(\zeta)\to\mathbb{C}$ which sends $\zeta$ to $e(1/p)$ and let $\mathcal{L}_\varphi$ be the Artin-Schreier sheaf on $\mathbb{A}^1=\mathbb{A}^1_{\mathbb{F}_q}$ corresponding to the additive character $\varphi$.

For a quasi-projective scheme $X/\mathbb{F}_q$ and a morphism $f:X\to\mathbb{A}^1$ the Grothendieck trace formula (\cite{Gr}) yields
\[\sum_{x\in X(\mathbb{F}_{q^m})}\varphi(f(x))=\sum_{i=0}^{2\dim(X)}(-1)^i\Tr(\Frob_q^m,H^i_c(\overline{X},f^*\mathcal{L}_\varphi)),\]
where $\overline{X}=X\otimes_{\mathbb{F}_q} \overline{\mathbb{F}}$,
$H^i_c$ is the $\ell$-adic cohomology group with compact support in degree $i$. We use the notation $H^\bullet_c$ for the "complex" of cohomologies. These cohomology groups are finite dimensional $\overline{\mathbb{Q}}_\ell$-vector spaces and $\Frob_q\in\mathrm{Gal}(\overline{\mathbb{F}}/\mathbb{F}_q)$ is the geometric Frobenius acting on them. By Deligne's work \cite{Deligne-weight} (see also \cite{K-W,Milne,Milne-RH})  we know that each Frobenius eigenvalue $\lambda^i_k$ on $H^i_c$ (for $1\leq k\leq d_i=\dim(H^i_c)$) is a Weil number of weight $j$ for some $j\leq i$ that is
\begin{equation}\label{eq:weil-number}
 |\iota(\lambda^i_k)|=q^{j/2}
\end{equation}
for all embeddings \( \iota: \mathbb{Q}(\lambda) \to \mathbb{C} \), and thus
\[
\sum_{x\in X(\mathbb{F}_{q^m})}\varphi(f(x))=\sum_{i=0}^{2\dim(X)}(-1)^i\sum_{j=1}^{d_i}(\lambda^i_j)^m.
\]

To simplify the notation we write $H^i_c(Y,f)=H^i_c(\overline{Y},(f^*\mathcal{L}_\varphi)|_{\overline{Y}})$ for arbitrary subschemes $Y\leq X$. As a corollary of the above we have that if \[H^i_c(Y,f)=0  \text{ for } i>d\] then
\begin{equation}\label{eq:deligne-bound}
\left| \sum_{x\in X(\mathbb{F}_{q^m})} \varphi(f(x)) \right| \leq C q^{md/2}
\end{equation}
with \(C=\sum_{i=0}^d \dim H^i_c(Y,f)\).
We will consider the cohomologies of the sums in the previous section: the sum $K_n(a)$ corresponds to the scheme $X=G=\mathrm{\GL}_n$ and the morphism $f=g:x\mapsto\tr(ax+x^{-1})$ (and also the embedding of $\iota_0:\mathbb{Q}(\zeta)\to\mathbb{C}$).

In the previous sections we derived bounds for the general Kloosterman sums over a finite extension where the eigenvalues of the coefficient matrix are defined. However bounds over an extension field do not imply that the weights are small. Consider for example  $X=(\mathbb{A}^1\setminus\{1\})\sqcup\mathbb{A}^0$ and the regular function $f:X\to\mathbb{A}^1$ defined by
\[f(x)=\left\{\begin{array}{ll}x, & \mathrm{if~}x\in\mathbb{A}^1\setminus\{1\},\\ 0, & \mathrm{if~} x\in\mathbb{A}^0\end{array}\right.\]
then $\displaystyle{\sum_{x\in X(\mathbb{F}_{q^m})}\varphi(f(x)) = 1-\zeta^m}$ which vanishes if  $p|m$ but only in that case.

The reason for this phenomenon is that the Frobenius eigenvalues on different cohomologies differ by a multiple of a root of unity and thus cancel in some extensions: here $\dim(H^1_c(X,f))=\dim(H^0_c(X,f))=1$, $\dim(H^2_c(X,f))=0$ and the Frobenius eigenvalues are $\lambda^1_1=\zeta$ and $\lambda^0_1=1$.

Thus, to get the general bound, we will prove that the cohomologies $H^i_c(G,g)$ vanish if $i$ is large enough, hence the weights are not too large.

We will use the following properties of $H^\bullet_c$ (for an overview see e.g. \cite{Katz-Asterisque} especially chapters 3.5. and 4.1-3 and  \cite{Fresan-Jossen}.):

\smallskip
\noindent
\textit{Excision.} If $f:X\to\mathbb{A}^1$ is a regular function, $Z\to X$ is a closed immersion and $U\to X$ is the complementary open immersion, then there exists a long exact sequence in the form
\[\dots\to H^i_c(U,f)\to H^i_c(X,f)\to H^i_c(Z,f)\to H^{i+1}_c(U,f)\to\dots\]

\smallskip
\noindent
\textit{K\"unneth formula.} If $f_i:X_i\to\mathbb{A}^1$ for $i=1,2$, $\pi_i$ is the canonical map $X=X_1\times_{\mathrm{Spec}(\mathbb{F}_q)} X_2\to X_i$, and $f_1+f_2:= \pi_1^*f_1+\pi_2^*f_2$, then \[(f_1^*\mathcal{L}_\varphi)\boxtimes(f_2^*\mathcal{L}_\varphi)\simeq (f_1+f_2)^*\mathcal{L}_\varphi\] and \[H^\bullet_c(X,f_1+f_2)\simeq H^\bullet_c(X_1,f_1)\otimes H^\bullet_c(X_2,f_2),\] that is for all $i$
\[H^i_c(X,f_1+f_2)\simeq\bigoplus_{j+k=i}H^j_c(X_1,f_1)\otimes H^k_c(X_2,f_2).\]

\smallskip
\noindent
We will also need some knowledge of the cohomologies in simple situations. They are listed in the following theorem.

\begin{thm}\label{thm:basic-H} \textit{Cohomology of some basic sheaves.}

\begin{enumerate}
    \item \label{item:A^1}
\textit{Cohomology of some basic sheaves on \(\mathbb{A}^1\).}
\begin{enumerate}
\item $H^i_c(\mathbb{A}^1,\id)=0$ for all $i$.
\item If $0:\mathbb{A}^1\to\mathbb{A}^1$ is the zero map, then $\mathcal{L}_0=f_0^*\mathcal{L}_\varphi$ is the constant sheaf and
\[\dim H^i_c(\mathbb{A}^1,0)=\left\{\begin{array}{ll}1, & \mathrm{if~} i=2,\\ 0, & \mathrm{if~} i\neq 2\end{array}\right.\]
The Frobenius eigenvalue on $H^2_c$ is $q$ (which is of weight 2).
\end{enumerate}

\smallskip
\noindent
\item\label{item:A^1-A^0} \textit{Cohomology of some basic sheaves on \(\mathbb{A}^1\setminus \mathbb{A}^0\).}
\begin{enumerate}
\item \(\dim H^i_c(\mathbb{A}^1\setminus\mathbb{A}^0,\id)=\left\{\begin{array}{ll}1, & \mathrm{if~} i=1,\\ 0, & \mathrm{if~} i\neq 1.\end{array}\right.\).

The Frobenius eigenvalue on $H^1_c$ is $1$ (which is of weight 0).
\item
\(\dim H^i_c(\mathbb{A}^1\setminus\mathbb{A}^0,0)=\left\{\begin{array}{ll}1, & \mathrm{if~} i=1,2,\\ 0, & \mathrm{if~} i\neq 1,2.\end{array}\right.\).

The Frobenius eigenvalue on $H^2_c$ is $q$ (which is of weight 2) and on $H^1_c$ is $1$ (weight 0).

\end{enumerate}

\smallskip
\noindent
\item\label{item:K} \textit{Cohomology of sheaves corresponding to Kloosterman sums.} (see \cite{Weil-Exp-sum}) If  \(\alpha \in \mathbb{F}_q^*\) and $f_\alpha: \mathbb{G}_m=\mathbb{A}^1\setminus\mathbb{A}^0 \to\mathbb{A}^1$ is the morphism which corresponds to the map
\[f_\alpha(t)=\alpha\cdot t+1/t\]
then\[\dim H^i_c(\mathbb{A}^1\setminus\mathbb{A}^0,f_\alpha)=\left\{\begin{array}{ll}2, & \mathrm{if~} i=1,\\ 0, & \mathrm{if~} i\neq 1\end{array}\right.\]
and on $H^1_c$ both weights are 1.
\medskip
\end{enumerate}

\end{thm}

The following observation is essential in what follows.

\begin{lem}\label{thm-Hfg}
Let $f,g:X\to\mathbb{A}^1$ be regular functions, $X_0=f^{-1}(\{0\})$ and consider $f\cdot\id_{\mathbb{A}^1}+g:X\times_{\mathrm{Spec}(\mathbb{F}_q)}\mathbb{A}^1\to\mathbb{A}^1$. Then
\begin{equation*}
H^\bullet_c(X\times_{\mathrm{Spec}(\mathbb{F}_q)}\mathbb{A}^1,f\cdot\id_{\mathbb{A}^1}+g)  \simeq  H^\bullet_c(X_0,g)\otimes H^\bullet_c(\mathbb{A}^1,0)
\end{equation*}
thus
\begin{equation}
H^{i+2}_c(X\times_{\mathrm{Spec}(\mathbb{F}_q)}\mathbb{A}^1,f\cdot\id_{\mathbb{A}^1}+g)  \simeq  H^i_c(X_0,g)\otimes H^2_c(\mathbb{A}^1,0) \mathrm{~for~all~}i.
\end{equation}
\end{lem}

\begin{proof}
Let $V=X\setminus X_0$ and consider the morphism $j=\id_V\otimes(\id_{\mathbb{A}^1}-g)/f:V\times\mathbb{A}^1\to V\times\mathbb{A}^1$. This is clearly an isomorphism and  $j\circ (f\cdot\id_{\mathbb{A}^1} +g)=0_V+\id_{\mathbb{A}^1}$ thus by the K\" unneth formula $H^\bullet_c(V\times\overline{\mathbb{A}^1},j\circ (f\cdot\id_{\mathbb{A}^1}+g))\equiv0$.

Let $Z=X_0\times\mathbb{A}^1$ and $U=X\times\mathbb{A}^1\setminus Z$. From the previous argument we get $H^\bullet_c(U,f\cdot\id_{\mathbb{A}^1}+g)\equiv0$. By excision $H^i_c(X\times_{\mathrm{Spec}(\mathbb{F}_q)}\mathbb{A}^1,f\cdot\id_{\mathbb{A}^1}+g)\simeq H^i_c(Z,f\cdot\id_{\mathbb{A}^1}+g)$, but $(f\cdot\id_{\mathbb{A}^1})|_{Z}\equiv0$, hence $(f\cdot\id_{\mathbb{A}^1}+g)|_Z=g|_Z=g|_{X_0}+0|_{\mathbb{A}^1}$ and applying the K\"unneth formula we get the lemma.
\end{proof}

\begin{remark}
This lemma is the cohomological form of the straightforward computation
\[\sum_{(x,t)\in X(\mathbb F)\times F}\varphi(tf(x)+g(x))=\sum_{x\in X(\mathbb{F})}\varphi(g(x))\sum_{t\in\mathbb{F}}\varphi(tf(x))=q\sum_{x\in X_0(\mathbb{F})}\varphi(g(x)).\]
A similar argument as in the proof appears in motivic context in \cite{Fresan-Jossen} Lemma 6.5.3 and Remark 6.5.4.

\end{remark}
Applying the lemma repeatedly we get:

\begin{cor}\label{Hfgx}
Let $\pi_j:\mathbb{A}^m\to\mathbb{A}^1$ be the projection \[(x_1,x_2,\dots, x_m)\mapsto x_j.\]
For $f_j,g:X\to\mathbb{A}^1$ for $1\leq j\leq m$ let \( h:X\times_{\mathrm{Spec} (\mathbb{F}_q)} \mathbb{A}^m\to\mathbb{A}^1 \) be defined by
\[h=\sum_{j=1}^mf_j\cdot\pi_j+g
.\]
Consider $X_0=\{x\in X|h(x,\cdot)\equiv 0\}\leq X$ as a subscheme. Then
\begin{equation*}
H^\bullet_c(X\times_{\mathrm{Spec}(\mathbb{F}_q)}\mathbb{A}^m,h)  \simeq  H^\bullet_c(X_0,g)\otimes \left(\bigotimes_{j=1}^mH^\bullet_c(\mathbb{A}^1,0)\right)
\end{equation*}
thus
\begin{equation*}
H^{i+2m}_c(X\times_{\mathrm{Spec}(\mathbb{F}_q)}\mathbb{A}^m,h)  \simeq  H^i_c(X_0,g)\otimes \left(\bigotimes_{j=1}^mH^2_c(\mathbb{A}^1,0)\right) \mathrm{~for~all~}i.
\end{equation*}
\end{cor}

We will now show the vanishing of cohomologies of high enough degree for the  exponential sums that were used in the previous sections.

\subsection{The proof of Theorem \ref{thm-semisimple-purity}} \label{sec-semisimple}

We first start with the reduction to Jordan blocks as in Proposition \ref{felsoblokk}.

Let $G=\mathrm{\GL}_n$, $G_k=\mathrm{\GL}_k$, $G_l=\mathrm{\GL}_l$ for some $n=k+l$. Let $a=\left(\begin{array}{c|c}a_k & b \\ \hline 0 & a_l\end{array}\right)\in \mathbb{A}^{n\times n}$ be block upper triangular with $a_k\in\mathbb{A}^{k\times k}, a_l\in \mathbb{A}^{l\times l}$. Let $g:G\to\mathbb{A}^1$, $x\mapsto\tr(ax+x^{-1})$ and for the diagonal blocks denote $g_k$ and $g_l$ the morphisms $x_k\mapsto \tr_k(a_kx_k+x_k^{-1})$ and $x_l\mapsto \tr_l(a_lx_l+x_l^{-1})$ respectively. Let $H^\bullet=H^\bullet_c(G,g)$ and similarly $H^\bullet_k=H^\bullet_c(G_k,g_k)$, $H^\bullet_l=H^\bullet_c(G_l,g_l)$.

\begin{pro}\label{felsoblokkH}
If $a_k$ and $a_l$ have no common eigenvalues, then
\[H^\bullet\simeq H^\bullet(\mathbb{A}^{k\times l},0)\otimes H^\bullet_k\otimes H^\bullet_l\text{ that is }
H^i\simeq H^{2kl}(\mathbb{A}^{k\times l},0)\otimes\left(\bigoplus_{j+j'=i-2kl}H^j_k\otimes H^{j'}_l\right),\]
where \(H^\bullet(\mathbb{A}^{k\times l},0)=\left(\bigotimes_{i=1}^{kl}H^\bullet_c(\mathbb{A}^1,0)\right)\)
\end{pro}

\begin{proof}
The morphism $U_{[k,l]}\to\mathbb{A}^{kl}$, $u\mapsto(u_{i,j})_{1\leq i\leq k, 1\leq j\leq l}$ is an isomorphism, so we apply Corollary \ref{Hfgx} with $m=kl$ and $X=G$ on $X\times_{\mathrm{Spec}(\mathbb{F}_q)}U_{[k,l]}$ and \[h:(x,u)\mapsto\tr\left(a(u^{-1}xu)+(u^{-1}xu)^{-1}\right).\]
From the proof of Proposition \ref{felsoblokk} it is clear that $h(x,\cdot)$ is cohomogically non-trivial if and only if $x'=0$, that is $x\in X_0$ with $X_0\leq X$ the subscheme of "block upper triangular" matrices and by Corollary \ref{Hfgx}
\[
H^\bullet_c(X\otimes_{\mathrm{Spec}(\mathbb{F}_q)}U_{[k,l]},h)\simeq H^\bullet_c(X_0,g)\otimes H^\bullet(\mathbb{A}^{k\times l},0) \simeq
H^\bullet_c(G_k,g_k)\otimes H^\bullet_c(G_l,g_l)\otimes H^\bullet(\mathbb{A}^{k\times l},0),
\]
where the second isomorphism is a consequence of $a$ being a block matrix thus
\[X_0\simeq G_k\times_{\mathrm{Spec}(\mathbb{F}_q)} G_l\times_{\mathrm{Spec}(\mathbb{F}_q)} \mathbb{A}^{k\times l},\quad g|_{X_0}=g_k+g_l+0_{\mathbb{A}^{k\times l}}\]
and the K\"unneth formula.

On the other hand $j:X\to X, (u,x)\mapsto (u,uxu^{-1})$ is an isomorphism and $j^*h=0_{U_k}+g$ so
\[H^\bullet_c(X\otimes_{\mathrm{Spec}(\mathbb{F}_q)}U_{[k,l]},h)\simeq H^\bullet_c(G,g)\otimes\left(\bigotimes_{i=1}^{kl} H^\bullet_c(\mathbb{A}^1,0)\right),\]
hence the Proposition.
\end{proof}

\begin{proof}[The proof of Theorem~\ref{thm-semisimple-purity}]
Since cohomology does not depend on the finite field in question we may assume that \(a\) is a diagonal matrix with non-zero and unequal entries \(\alpha_i\) on the the diagonal. As above let $g:G\to\mathbb{A}^1$, $x\mapsto\tr(ax+x^{-1})$ . 
Also let  $H^\bullet_{K(\alpha_i)} = H^\bullet_c(\mathbb{A}^1\setminus\mathbb{A}^0,f_\alpha) $. Repeated applications of the proposition above gives
\[
H^\bullet_c(G,g)=  H^{n(n-1)}(\mathbb{A}^{n(n-1)/2},0)\otimes \bigotimes_{i=1}^n  H^\bullet_{K(\alpha_i)} .
\]
Note that by Theorem \ref{thm:basic-H} (\ref{item:A^1}) and the Künneth formula \( \dim( H^{n(n-1)}(\mathbb{A}^{n(n-1)/2},0)) =1\), with Frobenius eigenvalue \( q^{n(n-1)} \). Similarly by Theorem \ref{thm:basic-H} (\ref{item:K}) \(\bigotimes_{i=1}^n  H^\bullet_{K(\alpha_i)}\) is concentrated in degree \(n\) where it equals \( \bigotimes_{i=1}^n  H^1_{K(\alpha_i)}\).

The claim now follows from the purity of the classical Kloosterman sums \(K_1(\alpha_i)\).
\end{proof}

\subsection{Bounding the weights in the non-split case}\label{sec-non-split-wt-bound}

For the proof of Theorem \ref{thm-general-bound} we need to bound the degrees of the non-trivial cohomology groups. Recall that for $w^2=I$  we defined in Theorem \ref{thm-scalar-closed-form}
\[
N=\n(w)= |\{ (i,j)\,|\, 1\leq i < j \leq n, \,w(j)<w(i)\leq j\}|.
\]
\begin{thm}\label{thm-weights}
Assume that \(a=\alpha \I+ \overline{a}\), with with $\alpha \in \mathbb{F}_q^*$ and
\( \overline{a}\) nilpotent,  and consider the cohomology $ H^\bullet_c(C_w,x\mapsto\tr(ax+x^{-1})) $ associated to the exponential sum (\ref{K^w}).
\begin{enumerate}
\item If $w^2\neq I$, then $H^\bullet_c(C_w,x\mapsto\tr(ax+x^{-1}))\equiv0$.
\item 	If \(w^2=\I\) then 	\(H^i_c(C_w,x\mapsto\tr(ax+x^{-1}))=0\)  for \(i>n^2+2\n(w)\),
thus all weights of the sum $\sum_{x\in C_w}\psi(ax+x^{-1})$ are at most $n^2+2\n(w)$.
\end{enumerate}
\end{thm}
\begin{proof}
(i) We may assume that \(\bar{a}\) is upper triangular. The case $w^2\neq I$ follows from the proof of Proposition \ref{w2-uj}. Let the pair \((i,j)\) be chosen as in there, \(i\) is minimal such that $i\neq w^2(i)$ and \(j=w(i)\). Consider the subvarieties $Y=\{I+s\e_{i,j}\mid s\in \mathbb{F}_q\}\simeq \mathbb{A}^1$, \(X_1=\Ub\) and \(X_2=\{x\in B\mid x_{ij}=0\}\). Then we have the decomposition $C_w=X_1\times Y\times X_2$ by mapping \(x_1\in X_1, x_2\in X_2, s\in \mathbb{F}_q\) to \( g = x_1 w (\I+s\e_{i,j}) x_2 \) (this is indeed an isomorphism of algebraic varieties). If \( X=X_1\times X_2\) then the proof of Proposition \ref{w2-uj} shows that the map \( g\tr a +(uwb)^{-1}\) is of the form \(f(x)s+g(x)\) with $X_0=\{ x\in X\mid f(x)=0\}$ empty. Hence Lemma~\ref{thm-Hfg} imply the vanishing of cohomology.

Assume now that \(w^2=\I\). By Lemma \ref{lem-Bruhat-param}, $C_w\simeq \Um\times T\times \Upb\times \Ubara \times U_o \times \Umb$ and this is an isomorphism of algebraic varieties.

We can apply Corollary \ref{Hfgx} in the setting of Proposition \ref{pro-Mw-sum}. We have $\mathbb{A}^m=\Ubara$ with $m=n_a$, $h:x\mapsto\tr(ax+x^{-1})$ and
\[X_0=\Ub\times T\times \Upb \times U_o\hookrightarrow X=\Um\times T\times \Upb \times U_o \times \Umb,\]
where the embedding maps the last coordinate to $I$. We obtain
\begin{equation*}
H_c^\bullet(C_w,x\mapsto\tr(ax+x^{-1}))\simeq
H_c^\bullet(\Um\times T\times \Upb \times U_o,g_1)\otimes\left(\bigotimes_{i=1}^{n_a} H^\bullet_c(\mathbb{A}^1,0)\right),
\end{equation*}
where $g_1=\tr(a^v wt u_1 u_o  + (wtu_1 u_o)^{-1})$ with the notation of subsection \ref{sec-trace-calc}.

Applying Corollary \ref{Hfgx} again in the setting of Proposition \ref{pro-U^flat_d} with
$\mathbb{A}^m\simeq U_o$, $m=n_o$, $h:x\mapsto\tr(ax+x^{-1})$ and
\[X_0=\Um\times T(w)\times \Upb\hookrightarrow X=\Um\times T\times \Upb,\]
where $T(w)$ is as in the proof of Proposition \ref{pro-K_n^w-eval}, thus we obtain

\begin{equation*}
H_c^\bullet(C_w,x\mapsto\tr(ax+x^{-1}))\simeq\\
H_c^\bullet(\Um\times T(w)\times \Upb,g_2)\otimes\left(\bigotimes_{i=1}^{n_a+n_o} H^\bullet_c(\mathbb{A}^1,0)\right),
\end{equation*}
where $g_2=\sum_{i=w(i)} (\alpha t_i+ t_i^{-1}) + \tr(\bar{a}^vwdu_1)$.

The K\"unneth formula yields
\begin{multline*}
H_c^\bullet(C_w,x\mapsto\tr(ax+x^{-1}))\simeq\\
H_c^\bullet(\Um\times \Upb,g_3)\otimes\left(\bigotimes_{i=1}^{n_a+n_o} H^\bullet_c(\mathbb{A}^1,0)\right)\otimes 
\left(\bigotimes_{j=1}^eH^\bullet_c(\mathbb{A}^1\setminus\mathbb{A}^0,0)\right)\otimes\left(\bigotimes_{k=1}^fH^\bullet_c(\mathbb{A}^1\setminus\mathbb{A}^0,f_\alpha)\right),
\end{multline*}
with $g_3=\tr(\bar{a}^vwdu_1)$ and $f_\alpha:x\mapsto \alpha x+x^{-1}$.

Now it is enough to observe that the maximal nontrivial cohomology group is of degree
\[2\dim(\Um\times \Upb)+2(n_a+n_o)+2e+f\]
where \(e\) in the number of involution pairs, and \(f\) is the number fixed elements in \(w\). It is clear that $n=2e+f$ and the calculation in the proof of (\ref{cor-scalar-explicit}) shows that the rest is equal to $n(n-1)+2\n(w)$.
\end{proof}

To show the vanishing of cohomologies of the Bruhat cells in higher degrees, we need one more combinatorial lemma. For this the Weyl group \(W\) is identified with the symmetric group \( S_n\) viewed as the group of permutations of  the set \( \{ 1,...,n \} \) as in (\ref{w-def}). In the notation of Remark \ref{rem:J-set-def} we let for any $w\in W$
\[\cJ^\flat_{a/o}(w)= \left\{ (i,j)\,|\, 1\leq i < j \leq n, \,w(j)<w(i)\leq j\right\}\]
Recall that $\n(w)=|\cJ^\flat_{a/o}(w)|=\dim(\bar U^\flat_{a/o})$.

\begin{lem}\label{lem-nw-max}
Let \(e\) be a positive integer such that \(e \leq \lfloor n/2\rfloor \) and \(w_e\in S_n\) be the involution for which
\[
w_e(j)=\begin{cases}
    n-j+1, & \text{ if } j=1,...,e\\
    j, & \text{ if }  j=e+1,..n-e.
\end{cases}
\]
We then have the following.
\begin{enumerate}
    \item If \(w^2=I\) and \(w\) is a product of \(e\) disjoint transpositions then \(\n(w)\leq \n(w_e)\) with equality only if \(w=w_e\).
    \item \(\n(w_e)= e(n-e)\). In particular \(\n(w)\) is maximal for the long element \(w_{\lfloor n/2\rfloor}\), and we have \(\n(w_{\lfloor n/2\rfloor}) = (n^2-\delta(n))/4\).
\end{enumerate}
\end{lem}

\begin{proof}

We proceed by induction on $e$. Let $k=w(n)$ and first assume that \(k>1\). Let $v=(12\dots k)\in S_n$, that is \[v(j)=\left\{\begin{array}{ll}j+1, & \mathrm{if~} j<k, \\ 1, & \mathrm{if~} j=k, \\ j, & \mathrm{if~} j>k.\end{array}\right.\]
and $w'=v^{-1}wv\in S_n$. We claim that $(i,j)\in \cJ^\flat_{a/o}(w)\Rightarrow (v(i),v(j))\in \cJ^\flat_{a/o}(w')$ thus $\n(w)\leq \n(w')$.

To see this, first assume that  $\{i,j\}\cap\{k,n\}=\emptyset$. In this case the claim is clear since $v$ respects the ordering of $X=\{1,2,\dots,n\}\setminus\{k,n\}$ and $w(X)\subseteq X$.

Now the case \((i,k)\in \cJ^\flat_{a/o}(w)\) does not arise since \(w(k)=n\). There is a  single $j$ such $(k,j)\in \cJ^\flat_{a/o}(w)$, namely $j=n$, but then $(v(k),v(n))=(1,n)\in \cJ^\flat_{a/o}(w')$. Finally $(i,n)\in \cJ^\flat_{a/o}(w')$ for all $i$, thus for all $(i,n)\in \cJ^\flat_{a/o}(w)$ we have $(v(i),v(n))\in \cJ^\flat_{a/o}(w')$.

\smallskip
Note that  if $w(n)\neq 1$, then the last case of the above argument shows $\n(w)$ is strictly smaller than $\n(w')$.

We now move to the induction step. Let \(w'\) be as above if \(w(n)\neq 1\) and \(w'=w\) otherwise. Let also $w''\in S_{n-2}$ be the element which we arises from the permutation \(w'\) restricted to \(\{2,...,n-1\}\) which we identify with \(\{1,...,n-2\}\) using \(j\mapsto j-1\). Then \(w''\) is a product of \(e-1\) transpositions and the induction hypothesis shows that \(\n(w'')\) is maximal if and only if \(w''\) arises from \(w'=w_e\).

To prove the second part note that if \(w(n)=1\), then \((i,n)\in \cJ^\flat_{a/o}(w) \) for all \(i<n\). Again let now $w''\in S_{n-2}$ be the element which we obtain by deleting the first and last rows and columns of $w'$. Then \( \n(w'')=\n(w)+n-1\). This time induction again shows that \( \n(w_e)=(n-1)+(n-3)+...+(n-2e+1)\) from which the statements follow.
\end{proof}

The following lemma enables us to work on the individual groups $\GL_{n_j}$:

\begin{cor}\label{cor:summary}
Let \(G=\GL_m\) for some \(m\) and  \(g:\mathrm{\GL}_{m}(\mathbb{F}_q)\to\mathbb{A}^1,x\mapsto\tr(ax+x^{-1})\) where \(a=\alpha \I+ \overline{a}\), with
\( \overline{a}\) nilpotent. Then $H^i(\GL_m,g)$ vanishes if $i>m^2+(m^2-\delta(m))/2$.

\end{cor}

\begin{proof}
We may assume that \(\bar{a}\) is upper triangular. For an involution \(w^2=I\) again let \(e=e(w)\) be the number of involution pairs in \(w\). By Theorem \ref{thm-weights} and Lemma~\ref{lem-nw-max}
\begin{equation}\label{eq:H^i(C_w)}
H^i_c(C_w,g) = \begin{cases}
    0 &\text{ for any } i,\text{ if } w^2\neq I \\
    0 &\text{ for } i> m^2+ 2e(m-e),  \text{ if } w\neq I, w^2=I\\
    0 & \text{ for } i \neq m^2 \text{ if } w=I.
\end{cases}
\end{equation}

Let \(l\) be the standard length function on \(W=S_m\) (\cite{Borel}, 21.21) and consider
\[
Y_l = \bigsqcup_{l(w)= l} C_w.
\]
We also let \(Y_0=B\) corresponding to the unit element of \(W\).

Note that if \(l(w)=l\) then \(C_w\) is an open subscheme of
\[
X_l = \sqcup_{l(w)\leq l} C_w= Y_l \sqcup X_{l-1}.
\]

Clearly  for any \(w\) we have
\(H^i_c(C_w,g)=0 \)
if \(i > m^2+(m^2-\delta(m))/2\) and so this remain true for \(Y_l\):
\[
H^i_c(Y_l,g)=0 \text{ if } i > m^2+(m^2-\delta(m))/2.
\]
We will now apply induction in the excision long exact sequence on the disjoint union \(X_l= Y_l \sqcup X_{l-1}\):
\[\dots\to H^i_c(X_{l-1},f)\to H^i_c(X_l,f)\to H^i_c(Y_l,f)\to H^{i+1}_c(X_{l-1},f)\to\dots\]

For \(l=0\) we have that \(X_0=Y_0\) and so
\[
H^i_c(X_0,g)=0
\]
already for \(i>m^2\). Assume now that \(H^i_c(X_{l-1},g)=0\) if \(i>m^2+(m^2-\delta(m))/2\). From the excision long exact sequence we then also have that \( H^i_c(X_l,g) = 0 \) for \( i> m^2+(m^2-\delta(m))/2\).

\end{proof}

\begin{proof}[Proof of Theorem \ref{thm-general-bound}]
Fix $a\in M_n$. As the weights do not change after base change, we might work over a sufficiently large finite extension of $\mathbb{F}_q$, say $\mathbb{F}_{q^{m(a)}}$, over which $a$ is conjugate to a block diagonal matrix, where the blocks on the diagonal are square matrices $a_j\in M_{n_j}$ in Jordan normal form with a unique eigenvalue $\alpha_j$.

Then by Proposition \ref{felsoblokkH} we have
\[H^\bullet_c(G,g)\simeq\left(\bigotimes_{j=1}^rH^\bullet_c(\mathrm{\GL}_{n_j},g_j)\right)\otimes\left(\bigotimes_{1\leq i<j\leq r}\bigotimes_{k=1}^{n_in_j}H^\bullet_c(\mathbb{A}^1,0)\right),\]
where  $g_j:\mathrm{\GL}_{n_j}(\mathbb{F}_q)\to\mathbb{A}^1,x\mapsto\tr(a_jx+x^{-1})$.

We apply the lemma with $m=n_i$ and $g=g_i$ for $1\leq i\leq r$ respectively.  By (\ref{eq:deligne-bound}) and Corollary~\ref{cor:summary}
we have that \(K_n(a)\ll q^d\), with
\[
d=\sum_{i=1}^r((3n_i^2-\delta(n_i)/4+\sum_{1\leq i<j\leq r}n_in_j.
\]
To conclude the proof note that
\[
\max\left(\sum_{i=1}^r((3n_i^2-\delta(n_i)/4+\sum_{1\leq i<j\leq r}n_in_j\Bigg|r,n_i\in\mathbb{N}: \sum_{i=1}^rn_i=n\right)=(3n^2-\delta(n))/4.\]
\end{proof}

\section{Degenarate cases}



\subsection{Preliminary observations on \(K_n(a,b)\)}\label{sec-degenarete-1}

The results in this section are combinatorial in nature and will not require cohomology. Therefore from now on we work solely over $\mathbb{F}_q$ and  $M_n=M_n(\mathbb{F}_q)$.

Let $a$ and $b$ singular $n\times n$ matrices such that
\begin{equation}\label{eq-rank-ab}
r=\rk(b)\geq s=\rk(a).\end{equation}

This section contains some elementary observations about the generalized Kloosterman sums
\[
K_n(a,b)=\sum_{x\in \GL_n(\mathbb{F}_q)}\psi(ax+bx^{-1}).
\]
First, we clearly have
\begin{equation}\label{eq-K(a,b)-action}
K_n(a,b)=K_n(c_1ac_2^{-1},c_2bc_1^{-1}) \text{ for any }c_1,c_2\in \GL_n(\mathbb{F}_q).
\end{equation}

In this, and the following sections, when we write a block matrix $a=\left(\begin{smallmatrix}
a_{11}& a_{12} \\ a_{21} & a_{22}
\end{smallmatrix}\right)
\in\mathbb{F}_q^{n\times n}$, we always mean the blocks to correspond to the  partition  $\{1,2,\dots, n\}=\{1,\dots, r\}\sqcup\{r+1,\dots, n\}$, \(r\) as in (\ref{eq-rank-ab}). For example by (\ref{eq-K(a,b)-action}) we may assume that \(b=\p_r\), where \begin{equation}\label{eq-e_r}
\p_r=\left(\begin{array}{cc}\I_r & 0 \\ 0 & 0 \end{array}\right)
\end{equation}
is a standard idempotent, but an exact description of  the equivalence classes is a delicate question. However all we need is a reasonable set of representatives for the action \((a,b) \mapsto ( c_1ac_2^{-1},c_2bc_1^{-1})\) that are suitable for handling the Kloosterman sums. This is most conveniently achieved via a parabolic Bruhat decomposition of \(G=\GL_n(\mathbb{F}_q)\) with respect to the subgroups \(P_r=P_r(\mathbb{F}_q)\) consisting of elements that when acting on row vectors map the subspace \(V_r= \langle\mathbf{e}_{r+1},\dots,\mathbf{e}_n\rangle \) to itself. In block matrix notation we have
\[
P_r=P_r(\mathbb{F}_q)=\{g\in \GL_n(\mathbb{F}_q)\,|\, g=\left(\begin{smallmatrix}
g_{11}& g_{12} \\ 0 & g_{22}
\end{smallmatrix}\right) \}.
\]

Let \(Q_r=P_r^T=\{g\in \GL_n(\mathbb{F}_q)\,|\, g=\left(\begin{smallmatrix}
g_{11}& 0 \\ g_{21} & g_{22}
\end{smallmatrix}\right) \}\) be the stabilizer of the columns space \(\langle \mathbf{e}_{1}^T,...,\mathbf{e}_r^T\rangle\), the subspace of linear functionals vanishing on \(V_r\). Then we have the following Bruhat decomposition for the group \(G=GL_n(\mathbb{F}_q)\).
\begin{pro}\label{pro-P_r-Q_r-Bruhat}
Let \(G=\GL_n(\mathbb{F}_q)\) with \(P_r,Q_r\) as above. Then
\[
G = \bigcup_{w\in W_P\backslash W /W_P} Q_r w P_r
\]
where \(W_P=W\cap P_r=W\cap Q_r\) and $W_P\backslash W /W_P$ denotes the set of double cosets.
\end{pro}
\begin{proof} From \(G=\cup_{w\in W} BwB\) it is clear that \(G=\cup_{w\in W} P_{n-r} w P_r\). Let \(w_l=\sum_{i=1}^n \e_{i,n-i}\) be the matrix that corresponds to the longest element \((1,n)(2,n-1)\dots\in W=S_n\). Then  \(w_lP_{n-r}w_l=Q_r\), from which \(G=w_lG=\cup_{w\in W} Q_r wP_r\). It is obvious that if \(w'=w_1ww_2\) with \(w_1,w_2\in W_P\), then \(Q_rw'P_r=Q_rwP_r\).
\end{proof}

The following lemma is well known, see for example section 1.3 in \cite{James-Kerber}.

\begin{lem}\label{lem-Young-subgroup}
Let $W=S_{1,2,\dots,n}$, $W_P=S_{1,2,\dots,r}\times S_{r+1,r+2,\dots,n}$ and $m=\min(r,n-r)$.
A set of double coset representatives of $W_P\setminus W/W_P$ is
\[\Wr =\{w_k\,|\,k=1,...,m\}
\]
where for \(k \leq m\), \(w_k\) is the permutation matrix
\begin{equation}\label{eq-w_k-def}
w_k=\sum_{i=1}^k \e_{i,i+r}+\sum_{i=k+1}^{r} \e_{i,i}+\sum_{i=r+k+1}^{n} \e_{i,i}
\end{equation}
which we also identify with  \(w_k={(1,r+1)(2,r+2)}\dots (k,r+k)\in W=S_n\).

%
\end{lem}

Recall that we have \(\p_r=\left(\begin{smallmatrix}\I_r & 0 \\ 0 & 0 \end{smallmatrix}\right)\).
\begin{pro}
Let $a$ and $b$ as in (\ref{eq-rank-ab}). Then there exist matrices $d$ and $w\in \Wr $ such that
$$K_n(a,b)=K_n(\p_r\cdot d,\p_r\cdot w).$$
\end{pro}

\begin{proof}
First we can write $a=c_1\p_{s}d_1$ and $b=d_2\p_r c_2$ for some $c_i,d_j\in\GL_n$ and thus $K_n(a,b)=K_n(\p_{r} d_0,\p_rc_0)$, where \(d_0=\p_{s} d_1d_2,\,c_0= c_2c_1\). Here we have used that \(\p_r \p_{s}=\p_{s}\), since \(r\geq s\).

Now we have that $c_0=qwp$, where $q\in Q_r$, $p\in P_r$ are as in Proposition~\ref{pro-P_r-Q_r-Bruhat}, and $w\in \Wr$ as in Lemma~\ref{lem-Young-subgroup}. Let
\begin{equation}\label{eq-Ur-def}
U_r=\{ g\in \GL_n(\mathbb{F}_q)\,|\, g= \left(\begin{smallmatrix}
\I_r & g_{12} \\ 0 & \I_{n-r}
\end{smallmatrix}\right)\}\end{equation}
be the unipotent radical of \(P_r\), and
\[L_r=P_r\cap Q_r=\{ g\in \GL_n(\mathbb{F}_q)\,|\, g= \left(\begin{smallmatrix}
g_{11}& 0 \\ 0 & g_{22}
\end{smallmatrix}\right)\},\]
so that \(P_r=L_rU_r\) and \(Q_r=L_r U_r^T\).

We have \(L_r=H_1H_2\), where
\begin{equation}\label{eq-L_r=H_1-H_2}
H_1= \{ g\in \GL_n(\mathbb{F}_q)\,|\, g= \left(\begin{smallmatrix}
g_{11}& 0 \\ 0 & \I_{n-r}
\end{smallmatrix}\right)\} \quad H_2 = \{ g\in \GL_n(\mathbb{F}_q)\,|\, g= \left(\begin{smallmatrix}
\I_r & 0 \\ 0 & g_{22}
\end{smallmatrix}\right)\}.
\end{equation}

Note that for \(g \in L_r\), \[\p_r g=g\p_r,\] and that for \(u \in U_r\) we have \[u\p_r = \p_r \text{  and } \p_r u^T=\p_r.\]
Therefore writing \(q=g_1u_1^T, p=g_2 u_2\) we have
\[
K_n(\p_r d_0, \p_r g_1u_1^T wg_2 u_2)=K_n( g_2 u_2\p_r d_0, g_1\p_r u_1^Tw) = K_n(\p_r d,w)
\]
where \(d = g_2d_0g_1\).
%
%
\end{proof}


\subsection{The Proof of Theorem \ref{thm-projections} and a preliminary bound}\label{sec-degenerate-K(E_r,0)}

In this section we will give a proof of Theorem~\ref{thm-projections} on \(K_n(a,0)\) with $a$ of rank \(r\), namely
\begin{equation}\label{eq-K_n(E_r,0)-eval}
K(a,0)=K_n(\p_r,0)=(-1)^rq^{-\binom{r+1}{2}}q^{rn}|\GL_{n-r}(\mathbb{F}_q)|.
\end{equation}
While this evaluation is trivial, for singular \(a,b\) it is the basis of our bounds for the general Kloosterman sums given in Proposition~\ref{pro-reduce-to-K_n(E-r,0)} below.

\begin{proof}[The Proof of Theorem~\ref{thm-projections}]
We have that
\[
K_n(\p_r,0)=\sum_{x\in \GL_n(\mathbb{F}_q)} \psi(\p_r x)=
\frac{1}{q^{r(n-r)}}\sum_{u \in U_r}\sum_{x\in \GL_n(\mathbb{F}_q)} \psi(\p_r u x)
\]
where \(U_r= \{ u\in \GL_n(\mathbb{F}_q)\,|\, u= \left(\begin{smallmatrix}
\I_r & u_{12} \\ 0 & \I_{n-r}
\end{smallmatrix}\right)\}\)  as in (\ref{eq-Ur-def}). Let \(x=\left(\begin{smallmatrix}
x_{11}& x_{12} \\ x_{21} & x_{22}
\end{smallmatrix}\right)\). It is clear that when summing over \(u\) first,
\[
\sum_{u \in U_r} \psi(\p_r u x)= \begin{cases} 0 & \text{ if } x_{21}\neq 0\\
q^{r(n-r)} & \text{ if } x_{21}=0. \end{cases}
\]
Therefore \[K_n(\p_r,0)=q^{r(n-r)}|\GL_{n-r}(\mathbb{F}_q)|K_r(\I_r,0)\] which leads immediately to the claim, in view of Theorem~\ref{thm-nil}.
\end{proof}

\begin{pro}\label{pro-reduce-to-K_n(E-r,0)} Let \(w=w_{k}\in \Wr\) be as in Lemma~\ref{lem-Young-subgroup}, and \(d \in M_n\). Then
\[|K_n(\p_r d,\p_r w)|\leq\sum_{j =0}^{k}
q^{j(n-r-j)-\binom{j}{2}}\frac{|\GL_{n-r-j}(\mathbb{F}_q)|}{|\GL_{n-r}(\mathbb{F}_q)|} R_w(j), \]
where $R_w(j) =\big| \{x\in \GL_n(\mathbb{F}_q)\ |\ x \p_r w=\left(\begin{smallmatrix}y_{11} & y_{12} \\ 0 & y_{22} \end{smallmatrix}\right), \rk(y_{22})=j \}\big|$.
\end{pro}

\begin{proof}
First swap the parameters
\[
K_n(\p_r d,\p_r w) = K_n(\p_r w,\p_r d)
\]
and then use the action of \(U_r\) as above to get
\begin{multline*}
K_n(\p_r w,\p_r d)=\sum_{x\in \GL_n(\mathbb{F}_q)} \psi(\p_r w x+\p_r d x^{-1}) =\\
\frac{1}{q^{r(n-r)}} \sum_{u \in U_r} \sum_{x\in \GL_n(\mathbb{F}_q)} \psi(\p_r wux + \p_r d x^{-1}u^{-1})=\\
\frac{1}{q^{r(n-r)}}  \sum_{x\in \GL_n(\mathbb{F}_q)}\sum_{u \in U_r} \psi(\p_r wux + \p_r d x^{-1}).
\end{multline*}
This shows that
\[
K_n(\p_r w,\p_r d)= \sum_{x\in \mathcal{R}_w} \psi(\p_r wx + \p_r d x^{-1})
\]
where \(\mathcal{R}_w= \{ x \in \GL_n(\mathbb{F}_q)\,|\, x\p_r w = \left(\begin{smallmatrix}
y_{11}& y_{12} \\ 0 & y_{22}
\end{smallmatrix}\right) \text{ for some } y_{11},y_{12}, y_{22}\}\).

Let
\begin{equation}\label{eq-R_w(j)}
\mathcal{R}_w(j)= \{ x \in \GL_n(\mathbb{F}_q)\,|\, x\p_r w = \left(\begin{smallmatrix}
y_{11}& y_{12} \\ 0 & y_{22}
\end{smallmatrix}\right), \rk(y_{22})=j\}
\end{equation}
so that \(\mathcal{R}_w=\sqcup_j \mathcal{R}_w(j)\). Since in $x\p_r w$ the last $n-r-k$ columns are 0,  \(\mathcal{R}_w(j)\) is empty if \(j>k\) and so
\[
K_n(\p_r w,\p_r d)=
\sum_{j=1}^k  \sum_{x\in \mathcal{R}_w(j)} \psi(\p_r wx + \p_r d x^{-1}).
\]

Clearly if \(x \in \mathcal{R}_w(j)\) then \(gx\in \mathcal{R}_w(j)\) for any \(g \in P_r\). Therefore let  \[H_2=\{g\in L_r\,|\, g=\left(\begin{smallmatrix}
\I_r& 0 \\ 0 & h
\end{smallmatrix}\right),\ h \in \GL_{n-r}(\mathbb{F}_q)  \}\] as in (\ref{eq-L_r=H_1-H_2}) and note that for \(g\in H_2\), \(x\in \mathcal{R}_w(j)\)
\[
\tr(\p_r w gx)= \tr^{(r)}(y_{11})+\tr^{(n-r)}(y_{22}h) \text{ and } \tr(\p_r d(gx)^{-1})=\tr(\p_r dx^{-1})
\]
where $\tr^{(j)}$ is the $j\times j$ matrix trace and \(y_{11},y_{22}\) are as in (\ref{eq-R_w(j)}). This immediately implies that for \(x \in \mathcal{R}_w(j)\)
\[
\sum_{g \in H_2} \psi( \p_r w gx + \p_r d(gx)^{-1})=K_{n-r}(\p_j,0)\varphi(\tr^{(r)}(y_{11})+\tr(\p_r dx^{-1})),
\]
and this gives
\begin{multline*}
\sum_{x\in \mathcal{R}_w(j)} \psi(\p_r wx + \p_r d x^{-1})=\\
\frac{1}{|\GL_{n-r}(\mathbb{F}_q)|}\sum_{g\in H}	 \sum_{x\in \mathcal{R}_w(j)}  \psi(\p_r wxg + \p_r d (xg)^{-1})
=\\
\sum_{x\in \mathcal{R}_w(j)}  \varphi(\tr^{(r)}(y_{11})+\tr(\p_r dx^{-1}))
\frac{K_{n-r}(\p_{j},0)}{|\GL_{n-r}(\mathbb{F}_q)|}.
\end{multline*}
The proposition follows from trivially estimating the last sum using \(|\varphi(\cdot)|\leq 1\) and the evaluation (\ref{eq-K_n(E_r,0)-eval}).
\end{proof}


\subsection{The Proof of Theorem~\ref{thm-degenerate}}\label{sec-degenarate-bound}

We restate the theorem and its corollary. We need to prove that
if $a$ and $b$ are singular $n\times n$ matrices such that \(s=\rk(a)\leq r=\rk(b)<n\) and \(m=\min(r,n-r)\), then
\begin{enumerate}
\item $K_n(a,b)\leq 2q^{n^2-rn+r^2+\binom{m}{2}}$,
\item If \(a,b\) are not both \(0\) then $K_n(a,b)\leq 2q^{n^2-n+1}$ and
\item this bound is sharp, since \[K_n(\e_{1,n},\e_{1,n})=q^{2n-2}|\GL_{n-2}(\mathbb{F}_q)|+(q-1)q^{n-1}|\GL_{n-1}(\mathbb{F}_q)|\sim q^{n^2-n+1}.\]
\end{enumerate}

Here (ii) is an obvious corollary of (i), and we start with the proof of that claim (Theorem~\ref{thm-degenerate}). By Proposition~\ref{pro-reduce-to-K_n(E-r,0)} this will require some estimates for the number \(R_w(j)=|\mathcal{R}_w(j)|\).

\begin{lem}\label{lem-R_w(j)-transitive}
Let \(w=w_k\in \Wr\) and \(\mathcal{R}_w(j)\) be as in (\ref{eq-R_w(j)}). Then
\[
\mathcal{R}_w(j)=P_r w_j P_s
\]
where \(P_r=L_rU_r\) as in (\ref{eq-Ur-def}), and \(P_s=H_1'H_2U_r\)
with \(H_1,H_2\) as in (\ref{eq-L_r=H_1-H_2}) and \( H_1'=H_1\cap w_k H_1 w_k\).
\end{lem}

\begin{proof} First note that if \( y \in \mathcal{R}_w(j)\p_r w\) then \(gy\in\mathcal{R}_w(j)\p_rw\) for any \(g \in P_r\). This immediately shows that \(P_r \mathcal{R}_w(j) = \mathcal{R}_w(j)\).

On the other hand if \(g=h_1h_2u\) with \(h_1 \in H_1'\), \(h_2\in H_2\) and \(u \in U_r\) then
\[
xg\p_r w= xh_1 \p_r w = x\p_r h_1w =  (x\p_rw) wh_1w = yh_1'
\]
for some \(h_1'\in H_1\) and $y=\left(\begin{smallmatrix}
y_{11}& y_{12} \\ 0 & y_{22}
\end{smallmatrix}\right), \rk(y_{22})=j$ as in (\ref{eq-R_w(j)}), which shows that \( \mathcal{R}_w(j)P_s = \mathcal{R}_w(j)\) as well.

The fact that there is a unique orbit represented by \(w_j\) is a direct calculation based on the definition of \(\mathcal{R}_w(j)\) which implies that \(x_{21}\) has rank \(j\), and the last \(r-k\) columns of \(x_{21}\) are identically \(0\).
\end{proof}

\begin{lem}\label{lem-R_w(j)-count}
In the notation above
\[
R_w(j)=|\mathcal{R}_w(j)|=
c_{n-r,k}(j) c_{r,r-j}(r-j) q^{r(j+n-r)}|\GL_{n-r}(\mathbb{F}_q)|
\]
where \(c_{k,l}(j)=|\{ x\in\mathbb{F}_q^{k\times l}\,|\,\rk(x)=j \}|\).
\end{lem}
\begin{proof}
Assume that \(x=\left(\begin{smallmatrix}
x_{11}& x_{12} \\ x_{21} & x_{22}
\end{smallmatrix}\right) \in \mathcal{R}_w(j)\), and that \( x' = \left(\begin{smallmatrix}
x_{11}& x'_{12} \\ x_{21} & x'_{22}
\end{smallmatrix}\right) \in \GL_n(\mathbb{F}_q)\). Then \(x'\in \mathcal{R}_w(j)\) as well, and there is \(u\in U_r\), \(h\in  H_2\), such that \(x'=xuh\).

It follows that \(R_w(j)=q^{r(n-r)}|\GL_n(\mathbb{F}_q)|\cdot |\mathcal{R}'_w(j)|\), where
\( \mathcal{R}'_w(j)\) consists of those \(n\times r\) matrices \( \{\left( \begin{smallmatrix}
x_{11} \\ x_{21}
\end{smallmatrix}\right)  \} \), which have rank \(r\), and for which the \((n-r)\times r\) matrix \(x_{21}\) is such that it has rank \(j\) and its last \(r-k\) columns are identically \(0\).

The number of choices for \(x_{21}\) for \(x=\left( \begin{smallmatrix}
x_{11} \\ x_{21}
\end{smallmatrix} \right) \in  \mathcal{R}'_w(j) \) is \(c_{n-r,k}(j)\).

From the transitivity $\mathcal{R}_w(j)=P_r\mathcal{R}_w(j)$ in Lemma~\ref{lem-R_w(j)-transitive} for each \(x_{21}\) there are the same number of possible \(x_{11}\).  For \(x_{21}\) the matrix with a \(\I_j\) in the top left corner and zeros everywhere else, it is readily seen that the number of possible \(x_{11}\)-s is \(q^{rj} c_{r,r-j}(r-j)\) which proves the claim.
\end{proof}

In view of the description in Lemma~\ref{lem-R_w(j)-transitive} it may seem that Proposition~\ref{pro-reduce-to-K_n(E-r,0)} is wasteful and a more exact evaluation is possible. While it is true that one can push this approach to get more precise information, we will see below that there are cases when the estimates are of the right order of magnitude. Still we will use some enumerative combinatorics, but merely for getting a good constant to match the \(q\)-power in the estimate that arises from Proposition~\ref{pro-reduce-to-K_n(E-r,0)}.
To do this it is convenient to use the Gaussian binomial coefficients ($q$ binomials) \cite{Cameron}. For $k\in\mathbb{N}$ let
\[
[k]_q=\dfrac{q^k-1}{q-1}, \quad [k]_q!=\prod_{j=1}^k[j]_q, \text{ and } {\qbinom{k}{l}=\frac{[k]_q!}{[l]_q!\cdot[k-l]_q!}}.
\]
%

With this notation we have $|\GL_k(\mathbb{F}_q)|=(q-1)^kq^{\binom{k}2}[k]_q!$ and the number of matrices of fixed size and rank (\cite{Landsberg} Formula ($\mathcal{B}$), see also \cite{Morrison} section 1.7.)
\begin{equation}\label{eq-rk=j}
c_{k,l}(j)=	\big|\{x\in\mathbb{F}_q^{k\times l}\,|\,\rk(x)=j \}\big|=(q-1)^j  q^{\binom{j }2}\dfrac{[k]_q![l]_q!}{[k-j ]_q![l-j ]_q![j]_q!}.
\end{equation}

%

We may therefore rephrase Lemma~\ref{lem-R_w(j)-count} as follows
\begin{equation}\label{lem-R_w(j)-count-pro}
R_w(j)=|\mathcal{R}_w(j)|= (q-1)^n q^{\binom{n}{2}+j ^2}\frac{[k]_q![r]_q![n-r]_q!^2}{[k-j ]_q![n-r-j ]_q![j ]_q!^2}.
\end{equation}
(Here \(w=w_k\in \Wr\) and \(\mathcal{R}_w(j)\) as in (\ref{eq-R_w(j)}).)%

\begin{proof}[The Proof of Theorem~\ref{thm-degenerate}]

By Lemma \ref{lem-R_w(j)-count}, (\ref{eq-rk=j}) and (\ref{lem-R_w(j)-count-pro}) we have that the summands in Proposition \ref{pro-reduce-to-K_n(E-r,0)} are equal to
\[(q-1)^{n-j } q^{\binom{n}{2}+j ^2} \frac{[k]_q![r]_q![n-r]_q!}{[k-j ]_q! [j ]_q!^2}\]
and thus
\begin{equation*}
|K_n(\p_r  d,\p_r  w)| \leq
(q-1)^{n-r}q^{\binom{n}{2}}[n-r]_q! \sum_{j =0}^{k}q^{j ^2} \qbinom{k}{j }\prod_{j'=j +1}^{r}(q^{j'}-1).
\end{equation*}
Using the trivial identity $q^{j'}-1\leq q^{j'}$ we have for the inner sum
\begin{multline*}
\sum_{j =0}^{k}q^{j ^2} \qbinom{k}{j } \prod_{j'=j +1}^{r}(q^{j'}-1) <
\sum_{j =0}^{k}q^{j ^2} \qbinom{k}{j }  q^{\binom{r+1}{2}-\binom{j +1}{2}}=	\\   q^{\binom{r+1}{2}} \sum_{j =0}^{k}q^{\binom{j }{2}}\qbinom{k}{j } =  2 q^{\binom{r+1}{2}}
\prod_{j =1}^{k-1} (1+q^j)
\end{multline*}
by the $q$-binomial theorem (\cite{Stanley} Formula (1.87)). From this
\begin{multline*}
|K_n(\p_r  d,\p_r  w)|\leq 2q^{\binom{n}{2}+\binom{r+1}{2}}\prod_{j =1}^{k-1}\left(q^{2j }-1\right)\prod_{j'=k}^{n-r}\left(q^{j'}-1\right)<\\
2q^{\binom{n}{2}+\binom{r+1}{2}+2\binom{ k}{2}+  k(n-r-  k+1)+\binom{n-r-  k+1}{2}}=2q^{n^2-rn+r^2+\binom{ k}{2}}.
\end{multline*}

Recall that \(m=\min(r,n-r)\) and thus by Lemma \ref{lem-Young-subgroup} we have \(1\leq k\leq m\), so
\[|K_n(\p_r  d,\p_r  w)|\leq 2q^{n^2-rn+r^2+\binom{m}{2}}.\]

\end{proof}

Finally we prove the claim about \(K_n(\e_{1,n},\e_{1,n})\). Using the Bruhat decomposition with the maximal parabolic subgroup \(P\) as in (\ref{eq-P-def}) we get
\begin{equation*}
K_n(\e_{1,n},\e_{1,n})=\sum_{k=1}^n\sum_{u\in U_k}\sum_{p\in P}\varphi\left((uw_{(kn)}p)_{n,1}+(uw_{(kn)}p)^{-1}_{n,1}\right)
\end{equation*}

Write $p=\left(\begin{array}{cc}h & v \\ 0 & \lambda \end{array}\right)$ with $h\in\GL_{n-1}(\mathbb{F}_q)$, $v\in\mathbb{F}_q^{(n-1)\times 1}$ and $\lambda\in\mathbb{F}_q^*$. Then
\[(uw_{(kn)}p)_{n,1}=\begin{cases}h_{k1}, & \mathrm{if~} k<n\\ 0, & \mathrm{if~} k=n \end{cases} \qquad (uw_{(kn)}p)^{-1}_{n,1}=\begin{cases}\lambda^{-1}, & \mathrm{if~} k=1\\ 0, & \mathrm{if~} k>1 \end{cases}\]
and thus
\[\sum_{g\in U_kw_{(kn)}P}\varphi\left((x)_{n,1}+(x)^{-1}_{n,1}\right)=\begin{cases}-q^{n-1}|U_1|K_{n-1}(\e_{1,1},0), &\mathrm{if~}k=1\\ (q-1)q^{n-1}|U_k|K_{n-1}(\e_{1,k},0), & \mathrm{if~} 1<k<n\\ |P|, & \mathrm{if~} k=n.\end{cases}\]

Since $K_{n-1}(\e_{1,k},0)=K_{n-1}(\e_{1,1},0)=-q^{n-2}|\GL_{n-2}(\mathbb{F}_q)|$ by Theorem \ref{thm-projections}, we get
\begin{multline*}
K_n(\e_{1,n},\e_{1,n})=\\
-q^{2n-3}|\GL_{n-2}(\mathbb{F}_q)|\left(-q^{n-1}+(q-1)\sum_{k=2}^{n-1}q^{n-k}\right)+(q-1)q^{n-1}|\GL_{n-1}(\mathbb{F}_q)|=\\
q^{2n-2}|\GL_{n-2}(\mathbb{F}_q)|+(q-1)q^{n-1}|\GL_{n-1}(\mathbb{F}_q)|\sim q^{n^2-n+1}.
\end{multline*}

%

\section{Examples}


\subsection{Kloosterman sums of \(2\times 2\) matrices}\label{sec-ex-n=2}

Let \(a\in M_2(\mathbb{F}_q)\). Since the Kloosterman sum is  invariant under the conjugation, \(K_2(a)=K_2(gag^{-1})\), for any \(g \in \GL_2(\mathbb{F}_q)\), we may assume that \(a\) is in Frobenius normal form,  \(a = \left( \begin{smallmatrix}
0&1\\-d&t
\end{smallmatrix}\right) \), where \(t=\tr(a), d=\det(a)\). The Kloosterman sum is then
\[
K_2(a)=\sum_{x_{11}x_{22}-x_{12}x_{21}\neq 0} \varphi(-d x_{12}+tx_{22})\varphi(x_{21}+(x_{11}+x_{22})/(x_{11}x_{22}-x_{12}x_{21}) ).
\]
From this presentation it is not at all clear that this sum should behave differently depending on whether \(t^2-4d\) is or is not a non-zero square or 0, showing that brute force calculations without using the finer group structure are unlikely to highlight any of the features of these sums.

For \(n=2\) the maximal parabolic of subsection~\ref{sec-parabolic-Bruhat} is also the Borel subgroup, and the two approaches given earlier are the same. They lead in an elementary way to the following evaluations.
\renewcommand{\arraystretch}{1.2}
\[\begin{array}{|c|c|c|c|c|c|c|}
\hline
\rule[-1ex]{0pt}{2.5ex}
a  & \left(\begin{smallmatrix}
\alpha	& 0\\0 &\beta
\end{smallmatrix}\right), \,\alpha \neq \beta &
\left(\begin{smallmatrix}
\alpha & 0\\0 &\alpha
\end{smallmatrix}\right), \alpha\neq 0 &
\left(\begin{smallmatrix}
\alpha	& 1\\0 & \alpha
\end{smallmatrix}\right), \alpha\neq 0 &
\left(\begin{smallmatrix}
0	& 0\\ 0& 0
\end{smallmatrix}\right) &
\left(\begin{smallmatrix}
0	& 1\\ 0& 0
\end{smallmatrix}\right)
\\
\hline
\rule[-1ex]{0pt}{2.5ex} K_2(a) & qK_1(\alpha)K_1(\beta) & q^3-q^2+K_1(\alpha)^2q & -q^2+K_1(\alpha)^2q & q & q
\\
\hline
\end{array}
\]
The non-split case, \( a= \left(\begin{smallmatrix}
\alpha 	& \delta \beta\\ \beta& \alpha
\end{smallmatrix}\right)\), with \(\beta\neq 0\) and \(\delta\notin (\mathbb{F}_q^*)^2 \) can also be evaluated explicitly, although this requires some effort. We have
\[
K_2(a)=-qK_1(\alpha+\beta\sqrt{\delta},\mathbb{F}_{q^2}^*).
\]
See the proof of Proposition~\ref{pro-n=2-semisimple-split} below.
%
%
%
%

It is also possible to deal directly with the more general sum \(K_2(a,b)\). Of course if either \(\rk(a)\) or \(\rk(b)\) is 2, say \(\rk(b)=2\) this leads to the previous evaluation by
\[
K_2(a,b)=K_2(ab^{-1},\I_2).
\]
If one of them, but not both are 0, say \(b=0\), and \(\rk(a)=1\), then
\[
K_2(a,0)=-q(q-1).
\]
If both of \(a\) and \(b\) have rank 1, then we may assume that \(a=e_1=\left( \begin{smallmatrix}
1&0\\0&0
\end{smallmatrix}\right)\), and that \(b\) is one of \( b_1=\left( \begin{smallmatrix}
\alpha&0\\0&0
\end{smallmatrix}\right)\) for which \[
K_2(e_1,b_1)=K_1(\alpha) q(q-1),
\] or \( b_2=\left( \begin{smallmatrix}
0& 1\\0&0
\end{smallmatrix}\right)\) 
for which
\[
K_2(e_1,b_2)=-q(q-1),
\]
or \(b_3=\left( \begin{smallmatrix}
0 & 0\\0 & 1
\end{smallmatrix}\right)\) for which
\[
K_2(e_1,b_3)=q^3-q^2+q.
\]
Finally one trivially has that \(K_2(0,0)=(q^2-1)(q^2-q)\).

\subsection{The recursion in closed form}\label{sec-ex-recursion}

In this subsection we will describe an algorithm for calculating the polynomials that express \(K_n(a)\) when \(a=\alpha\I_n+\bar{a}(\lambda)\) for some fix \(\alpha \neq 0\) and some partition \(\lambda=[n_1,...,n_l]\). As usual the fact that \(n_1+...+n_l=n\) is denoted by \(\lambda \vdash n \). We will rely on the notations of sections~\ref{sec-parabolic-Bruhat} and \ref{sec-recursion}, where we made the assumption that \(n_i\leq n_{i+1}\). If an element \(n_i\) repeats \(k\) times we will write \([...,n_i^k,...]\) instead of \([...,n_i,...,n_i,...]\), so for example we will write \([1^n]\) for the partition that corresponds to the matrix \(\alpha I\).

Also we will denote the polynomials by \(K_\lambda\) as well, instead of the notation \(P_\lambda\) in Theorem \ref{thm-motivic}. However we will still write \(K\) for \( K_{[1]} = K_{[1]} (\alpha)=K_1(\alpha)\), so for example Theorem~\ref{thm-recursion} can be stated as
\[
K_{[1^n] }= q^{n-1}K K_{[1^{n-1}]} + q^{2n-2}(q^{n-1}-1) K_{[1^{n-2}]}.
\]

It is clear from the proof of Theorem~\ref{thm-red-alg-last-step} below that if \(n_{l-1}<n_l\) the recursion is particularly simple, and has only two terms corresponding to the sums over the cells \(X_{n-1}\) and \(X_n\). Therefore in the case \(n_{l-1}<n_l\)
\[
K_\lambda =q^{n-1} K K_{\lambda'}-q^{2n-2}K_{\lambda''},
\]
where \[\lambda'=[n_1,...,n_{l-1},n_l-1] \text{ and } \lambda''=[n_1,...,n_{l-1},n_l-2]\] possibly reordered into a monotonic sequence, if \(n_{l-1}=n_l-1.\) If \(n_l=2\), then the entries corresponding to \(n_l-2=0\) are simply deleted. For example for \(\lambda =[1,2]\) this gives
\[
K_{[1,2]}=q^2 K K_{[1^2]}-q^4 K.
\]

The situation is more interesting when the last entry is repeated. This can be handled by the following explicit recursion, which gives an alternative proof of Theorem~\ref{thm-motivic}.
\begin{thm}[(Recursion algorithm)]\label{thm-red-alg-last-step}
Assume that \(\lambda=[n_1^{k_1},..., n_{l-1}^{k_{l-1}},n_{l}^{k_l}]\), with \(k_l>1\). Then
\[
K_\lambda =q^{n-1} K  K_{\lambda'}-q^{2n-2}K_{\lambda''}-(q^{k_l-1}-1)q^{2n-2}\left( K_{\lambda''}-K_{\lambda'''}\right),
\]
where \(\lambda'=[n_1^{k_1},..., n_{l-1}^{k_{l-1}},n_l-1,n_{l}^{k_l-1}] \), \(\lambda''=[n_1^{k_1},..., n_{l-1}^{k_{l-1}},n_l-2,n_{l}^{k_l-1}]\) and\break \(\lambda'''=[n_1^{k_1},..., n_{l-1}^{k_{l-1}},(n_l-1)^2,n_{l}^{k_l-2}]\) reordered into a monotonic sequence, if needed.
\end{thm}
Note that  \(\lambda'\vdash n-1 \) and \(\lambda'', \lambda'''\vdash n-2\). It is essential that the ``reduced'' partitions \(\lambda',\lambda'',\lambda'''\) are put into the canonical non-decreasing form we are using. This process is somewhat inconvenient to express in notation, but easy to do so in practice. For example if \(\lambda=[1,2,3^2]\) then \(\lambda'=[1,2^2,3]\), \(\lambda''=[1^2,2,3]\) and \(\lambda'''=[1,2^3]\), while for \(\lambda=[1,2,4^2]\), \(\lambda'=[1,2,3,4]\), \(\lambda''=[1,2^2,4]\) and \(\lambda'''=[1,2,3^2]\) etc.

The algorithm can be used to express \(K_n(a)\) for any \(a\) with a split characteristic polynomial when \(n\) is small, by first using Theorem~\ref{thm-split-semisimple}. For example for \(n=3\), \(a=\alpha\I+\bar{a}(\lambda)\), \(\alpha\neq 0\), it gives
\[
\begin{array}{|c|c|c|c|}
\hline
\lambda	& [1^3] & [1,2] & [3] \\
\hline
K_\lambda
& q^3K^3+(q^5+2q^4)(q-1)K & q^3K^3 + q^4(q-2)K & q^3K^3-2q^4K\\
\hline
\end{array}.
\]
%

The first non-trivial example when \(\lambda'''\) appears is \([2^2] \vdash 4\), for which
\[K_{[2^2]}=q^3 K K_{[1,2]} -q^6K_{[2]} -q^6(q-1)(K_{[2]}-K_{[1^2]})=q^6K^4 + q^7(q-3)K^2+q^8(q^2-q+1).\]

To see a more intricate situation we illustrate the algorithm for \(n=6\) and \(\lambda = [1,2,3]\), when we have
\[
K_{[1,2,3]}=q^5 K K_{[1,2^2]}-q^{10}K_{[1^2,2]}.
\]
Furthermore
\[
K_{[1,2^2]}=q^{4}K_{[1^2,2]}-2q^{8}K_{[1,2]}
\]
and so on.

There are many families where the recursion may be stated in simple terms, for example if there is only one block, \(\lambda=[n] \), when we have
\[
K_{[n]} =  K q^{n-1}K_{[n-1]} - q^{2n-2}K_{[n-2]} .
\]
From this one can get a closed formula for \(K_{[n]}\) see  subsection~\ref{sec--ex-purity} below.



Theorem~\ref{thm-red-alg-last-step} is an easy corollary of the following two propositions. As usual we assume \(\alpha\neq 0\), and \(a=\alpha \I + \bar{a}(\lambda)
\) whith $\varepsilon_j\in \{0,1\}$ as in (\ref{eq-lambda-2-b}) with \(\varepsilon_{n-1}=1\) and use \(a''\) to denote \(a''_{\not{k},\not{n}}\).

\begin{pro}\label{pro-red-step1}
Let \(Z=\left\{ z= \sum_{j=i+1}^{l-1} \xi_j \e_{k-1,N_j-1}\,|\, \xi_j \in \mathbb{F}_q \right\}\). In the above notation
\begin{enumerate}
\item \(a''+z\) and \(a''\) are conjugate over \(\mathbb{F}_q\).
\item \(a''+z+\e_{k-1,n-2}\) and \(a''+\e_{k-1,n-2}\) are conjugate over \(\mathbb{F}_q\).
\end{enumerate}
\end{pro}

As an immediate corollary of (\ref{eq-red-try}) we get that $\sum_{x\in X_k} \psi(ax+x^{-1})=0$, unless \(k=N_i\) for some $k<l$ such that $n_k=n_l$ and then

\[\label{eq-red-final}
\sum_{x\in X_k} \psi(ax+x^{-1})=
q^{2n+k-l-3}(q-1)\Big(K_{n-2}(a'')-K_{n-2}(a''+\e_{k-1,n-2})\Big)
\]

%
%

The second step in the proof is the following
\begin{pro}\label{pro-red-step-2}
\begin{enumerate}
\item If \(n_i<n_l\) then  \(a''+\e_{k-1,n-2}\) and \(a''\) are conjugate over \(\mathbb{F}_q\).
\item
If \(n_i=n_l\), \(a''+\e_{k-1,n-2}\) is conjugate to \(a'''\), where \(a'''\) is built from the partition \(\lambda'''\) as in (\ref{eq-lambda-2-b}).
\end{enumerate}
\end{pro}

While not needed, we remark that the proposition remains true over \(\mathbb{Z}\).

The proof of the claims in the propositions will use the linear transformation interpretation form subsection~\ref{sec-proof-motivic}.
We start with an easy observation.
\begin{lem}\label{lem-L_a-struct}
Let \(\L_A=\vgen{v_0}\oplus...\oplus\vgen{v_l} \simeq \cC_{n_0}\oplus...\oplus \cC_{n_l}\). If \(v\in \L_A\) is such that \(A^kv=0\) for \(k<n_0\), then \(\L_A=\vgen{v_0+v}\oplus...\oplus\vgen{v_l} \) as well.
\end{lem}
\begin{proof} One easily checks that
\( \vgen{v_0+v} \cap \big(\vgen{v_1}\oplus...\oplus\vgen{v_l}\big)=\{0\} \) and that
\( \vgen{v_0+v} + \big(\vgen{v_1}\oplus...\oplus\vgen{v_l}\big)=\L_A\).
\end{proof}

\begin{rem}
It follows that there is an isomorphism \(\phi: \L_A \to \L_A\) which is trivial on \( \vgen{v_1}\oplus...\oplus\vgen{v_l} \) and extends \(v_0\mapsto v_0+v\). Clearly, it satisfies \(A\phi=\phi A\).
\end{rem}

The question for us is to determine how a module structure given by \(A:\L\to \L\) changes if \(A\) is perturbed by another map \(Z:\L\to \L\). For example Proposition~\ref{pro-red-step1} and (i) of Proposition~\ref{pro-red-step-2} are an easy consequence of the following lemma.

\begin{lem}\label{lem-earlier-blocks}
Assume that \( \L_A \simeq \vgen{v_0} \oplus \vgen{v_1} ...\oplus \vgen{v_l} \simeq \cC_{n_0} \oplus \cC_{n_1} \oplus  ... \oplus\cC_{n_l}\) and that
\(Z:\L \to \L \) is such that for \(i=1,...,l\)
\[
Z(A^j v_i) =  0  \quad\text{ for  } j=0,...,n_i-1
\]
%
and that moreover for \(i=0\) we have
\begin{align*}
Z(A^j v_0) & =  0 \quad \text{ for } j=0,...,n_0-2,\\
Z(A^{n_0-1} v_0) & = \sum_{i=1}^l  \xi_i A^{n_i-1} v_i.
\end{align*}

If \(n_0<n_i\) for all \(i=1,...,l\) then \(\L_{A+Z} \simeq \L_{A}\) as  \(\mathbb{F}_q[T]\)-modules.
\end{lem}
\begin{proof} Let \(v= \sum_{i=1}^l  \xi_i A^{n_i-n_0} v_i\). Clearly \(A^{n_0}v=0\), and so by the Lemma \ref{lem-L_a-struct} there is an \(A\)-linear isomorphism \(\phi\), for which \(\phi(v_0)=v_0-v\). One easily checks that \(Z\circ \phi=0\) and so \(\phi\) provides the claimed isomorphism.
\end{proof}

\begin{lem}
Assume that \( \L_A \simeq \vgen{v_0} \oplus \vgen{v_1} \simeq \cC_{m} \oplus \cC_{m} \) and that
\(Z:\L \to \L \) is such that
\begin{align*}
Z(A^j v_1) & = 0 \quad \text{ for } j=0,...,m-1,\\
Z(A^j v_0) & = 0 \quad \text{ for } j=0,...,m-2,\\
Z(A^{m-1}v_0) & =\xi A^{m-1}v_1
\end{align*}
for some \(\xi\in \mathbb{F}_q^*\). Then \(\L_{A+Z} \simeq \mathcal{C}_{m-1}\oplus \mathcal{C}_{m+1}\) as  \(\mathbb{F}_q[T]\)-modules.
\end{lem}
\begin{proof}
It is easy to see that \((A+Z)^jv_0=A^jv_0\), for \(j=0,...,m-1\) and that \((A+Z)^mv_0=\xi A^{m-1}v_1\). Moreover if we replace \(v_1\) by \(v_1'=v_1-\frac{1}{\xi}v_0\), then \((A+Z)^{m-1}v_1'=0\).
\end{proof}

\begin{proof}[The Proofs of Proposition~\ref{pro-red-step1} and \ref{pro-red-step-2}]
The two lemmas above give exactly this.
\end{proof}

\subsection{Examples for the Kloosterman sums over Borel Bruhat cells}\label{sec-ex-borel-poly}

Let \(B\) be the standard Borel subgroup of invertible upper triangular matrices, \(w \in W\) an element of the Weyl group, and \(C_w=BwB\). We will consider here the sums
\[
K_n^{(w)}(a)=\sum_{x \in C_w}\psi(ax+x^{-1}),
\]
where \(a=\alpha \I +\upr{a}\), with \(\upr{a}\in \bar{U}\), where \(\bar{U}\) is the set of  strictly upper triangular matrices. We will first comment on the nature on these sums and then derive some of the properties that will be used below in the section on purity.

As in the proof of Theorem~\ref{thm-motivic} in subsection \ref{sec-red-2-GL_n-2} one can show that these sums satisfy a recursion that connects them to similar sums of rank \(n-1\) or \(n-2\), depending on whether \(w(n)=n\) or \(w(n)<n\). To see this, recall from Proposition~\ref{pro-K_n^w-eval} that
\begin{equation}\label{eq-K-eval-repeat}
K_n^{(w)}(a)= q^{n_a+n_o} K_1(\alpha)^f S_n^{(w)}(\overline{a}),
\end{equation}
where \(n_o\) is the number of involution pairs in \(w\), \(f\) is the number of fixed points of \(w\) and where the auxiliary sum is given by
\[
S_n^{(w)}(\overline{a})=\sum_{v,{d,u_1}} \psi(\bar{a}^vwdu).
\]
with \(v \in \Um, u \in \Upb\),  \(d\in D(w)=\{ \sum_{i<w(i)} t_i \e_{i,i}\,|\, t_i\in \mathbb{F}_q^*\}\) and \(\overline{a}^v=v^{-1}\overline{a}v\). The recursion then proceeds on the sum \(S_n^{(w)}(a)\). For example when \(w(n)=k<n\) we again let  \(m''=m''_{\not{k},\not{n}}\) denote the matrix one gets by deleting the \(k\)-th and \(n\)-th rows and columns of an \(n\times n\) matrix \(m\). To describe the
set of perturbations \(Z\subset M_{n-2}\) that arise in the reduction, we let \(\overline{a}_{(i)}\) and \(\overline{a}^{(i)}\) denote the \(i\)-th row, respectively the $i$th column, of \(\overline{a}\) and consider the set
\begin{equation*}
Y=\{ y\in \mathbb{F}_q^n\,|\, y_i=0 \text{ for } i\leq k, \text{ and } y\overline{a}=\overline{a}_{(k)}  \}.
\end{equation*}
Using this notation we have the following proposition whose proof goes along the lines of  Proposition~\ref{pro-red-using-blocks} and is omitted here.
\begin{pro}\label{eqn-borel-red} The following equation holds
\begin{equation}
S_n^{(w)}(\overline{a})=q^{n-k-1} \sum_{y\in Y} S_{n-2}^{(w'')}(\overline{a}''+(\overline{a}^{(k)} y)'')\sum_{t\in \mathbb{F}_q^*} \varphi(t y\overline{a}^{(n)}).
\end{equation}
\end{pro}
Note that the set \(Y\) may be empty, and then so is the set of perturbations which arise from the collection of rank 1 matrices
\begin{equation*}
Z=\{ \overline{a}^{(k)} y \in M_n\,|\, y\in Y \},
\end{equation*}
in which case the sum is interpreted as 0.


What makes the sums \(S_n^{(w)}(\overline{a})\) harder to deal with is that as functions of \(a\) they are no longer invariant under conjugation by \(\GL_n\). They are invariant under conjugation by elements of \(B\) in virtue of (\ref{eq-K-eval-repeat}) and the fact that
\[
K_n^{(w)}(a)=K_n^{(w)}(b^{-1}ab)
\]
for any \(b \in B\), since \(b^{-1}C_wb=C_w\). However the \(B\)-orbits in the set \(\bar U\) of strictly upper triangular matrices under the adjoint action are not well-understood. It is easy to see that one can no longer straighten out partitions into a non-decreasing order, or even assume that an orbit is represented by a matrix in Jordan normal form. For example \(\{t \e_{1,3}\,|\, t \neq 0\}\) is one of the orbits in the \(3 \times 3 \) case. When \(n= 6\), there is even a one-parameter family of orbits, found by Kashin \cite{Kashin}. In general a full description of the orbits is hard even in low ranks \cite{B-H, Hille-Rohrle}.

Returning to the sums \(K_n^{(w)}(a)\) it is still quite likely that  these can be expressed as polynomials in \(q\) and \(K\), independently of the characteristic \(p\) since the perturbations arising in the reduction calculations are of a very special nature. In what follows we will provide some low rank examples, when the independence is easy to establish directly. We will do this by stating certain special cases when the above reduction is sufficient, most importantly the case when \(w=(i\,j)\).
There are a number of other special cases when the reduction for \(S_n^{(w)}(a)\) can be treated in a simple manner, for example when \(w(n)=1\) or \(n-1\).

Again we assume that $a=\alpha I+\bar a$ is of the form as in (\ref{a-def}), $\bar a=\sum_{j=1}^{n-1} \varepsilon_j \e_{j,j+1}$, and $\lambda \vdash n$ is the partition corresponding to $a$. Denote

\[
K_\lambda^{(w)}(\alpha)=\sum_{x \in C_w}\psi(ax+x^{-1}).
\]


\begin{thm}\label{thm-inv-S}
Let $i<j$ and $(ij)\in W$ and \(\alpha\neq 0\). Then
\[K^{(ij)}_\lambda(\alpha)=
q^{n(n-1)/2}K^{n-2}\cdot\begin{cases}
(q-1)q
, & \mathrm{if~}j=i+1\mathrm{~and~}\varepsilon_i=0\\
-q
, & \mathrm{if~}j=i+1\mathrm{~and~}\varepsilon_i=1\\
(q-1)q^{j-i-d}
, & \mathrm{if~}j>i+1\mathrm{~and~}\varepsilon_i=\varepsilon_{j-1}=0
\end{cases}
\]
where $d=\vert\{k\in\mathbb{N}|i<k<j-1\mathrm{~and~}\varepsilon_k\neq0 \}\vert$.

If \(j>i+1\) and either \( \varepsilon_i\) or \(\varepsilon_{j-1}\neq 0\) then \(K^{(ij)}_\lambda(\alpha)=0\).
\end{thm}

This allows us to compute the Bruhat cell polynomials for $n\leq 3$:
\begin{cor}	$n=2$
\[
\begin{array}{|r|c|c|}
\hline
\lambda\setminus w & (12) & \I \\
\hline
[1,1] & q^2(q-1) & qK^2 \\
\hline
[2] & -q^2 & qK^2\\
\hline
\end{array}\]

$n=3$
\[
\begin{array}[t]{|r|c|c|c|c|}
\hline
\lambda\setminus w & (13) & (12) & (23) & \I \\
\hline
[1,1,1] & q^5(q-1)K & q^4(q-1)K & q^4(q-1)K & q^3K^3 \\
\hline
[2,1] & 0 & q^4(q-1)K & -q^4K & q^3K^3 \\
\hline
[1,2] & 0 & -q^4K & q^4(q-1)K & q^3K^3 \\
\hline
[3] & 0 & -q^4K & -q^4K & q^3K^3 \\
\hline
\end{array}
\]
\end{cor}

With a little more work one can calculate all the Bruhat cell polynomials for $n=4$. We summarize the result in the following table:
\begin{cor}

\[\begin{array}{|r|c|c|c|c|c|}
\hline
\lambda\setminus w & (14)(23) & (13)(24) & (12)(34) & (14) & (13)\\
\hline
[1,1,1,1] & q^{10}(q-1)^2 & q^9(q-1)^2 & q^8(q-1)^2 & q^9(q-1)K^2 & q^8(q-1)K^2 \\
\hline
[2,1,1] & 0 & 0 & -q^8(q-1) & 0 & 0 \\
\hline
[1,2,1] & -q^9(q-1) & 0 & q^8(q-1)^2 & q^8(q-1)K^2 & 0 \\
\hline
[1,1,2] & 0 & 0 & -q^8(q-1) & 0 & q^8(q-1)K^2 \\
\hline
[3,1] & 0 & 0 & -q^8(q-1) & 0 & 0 \\
\hline
[2,2] & 0 & q^9(q-1) & q^8 & 0 & 0 \\
\hline
[1,3] & 0 & 0 & -q^8(q-1) & 0 & 0 \\
\hline
[4] & 0 & 0 & q^8 & 0 & 0  \\
\hline
\end{array}
\]
\[\begin{array}{|r|c|c|c|c|c|}
\hline
\lambda\setminus w & (24) & (12) & (23) & (34) & \I \\
\hline
[1,1,1,1] & q^8(q-1)K^2 & q^7(q-1)K^2 & q^7(q-1)K^2 & q^7(q-1)K^2 & q^6K^4 \\
\hline
[2,1,1] & q^8(q-1)K^2 & -q^7K^2 & q^7(q-1)K^2 & q^7(q-1)K^2 & q^6K^4 \\
\hline
[1,2,1] & 0 & q^7(q-1)K^2 & -q^7K^2 & q^7(q-1)K^2 & q^6K^4  \\
\hline
[1,1,2] & 0 & q^7(q-1)K^2 & q^7(q-1)K^2 & -q^7K^2 & q^6K^4 \\
\hline
[3,1] & 0 & -q^7K^2 & -q^7K^2 & q^7(q-1)K^2 & q^6K^4 \\
\hline
[2,2] & 0 & -q^7K^2 & q^7(q-1)K^2 & -q^7K^2 & q^6K^4 \\
\hline
[1,3] & 0 & q^7(q-1)K^2 & -q^7K^2 & -q^7K^2 & q^6K^4 \\
\hline
[4] & 0 & -q^7K^2 & -q^7K^2 & -q^7K^2 & q^6K^4 \\
\hline
\end{array}
\]
\end{cor}

Note that the polynomials are not merely a permutation for different rearrangements of a partition, see for example the case \(\lambda=[1,2,1]\) and \(w=(14)(23)\).

We finish this section by giving the cell polynomials for the full block ($\lambda=[n]$) case for general $n$.
Let $w=(i_1,j_1)(i_2,j_2)\dots(i_r,j_r)\in W$ such that $i_k<j_k$ for any $k$. Then
\[K^{w}_{[n]}(\alpha)=\begin{cases}(-1)^rK^{n-2r}q^{n(n-1)/2+r}, & \mathrm{if~} j_k-i_k=1 \mathrm{~for~any~}k \\0, & \mathrm{otherwise}\end{cases}\]

Since this is merely for illustration we only give a sketch of the argument. For those $w\in W$ such that $K^{w}_n(\alpha)=0$ one can find $(i,j)\in I$ such that $\tr(v^{-1}avwdu)$ is nonconstant and linear in $v_{i,j}$.


\subsection{Regular semisimple matrices}\label{sec-ex-semisiople-non-split}

We have seen that for an \(n\times n\) matrix \(a\), whose characteristic polynomial \(P_a\) have no multiple roots  the cohomology associated to the Kloosterman sum \(K_n(a)\) is pure. Assume now that this characteristic polynomial  \(P_a\) is irreducible over \(\mathbb{F}_q\). Let \(\alpha \in \mathbb{F}_{q^n}\) be an eigenvalue of \(a\), \(P_a(\alpha)=0\). The argument in subsection~\ref{sec-cohomology} shows that
over $\mathbb{F}_{q^n}$
\[H^\bullet=H^\bullet_c(\GL_n,x\mapsto\tr(ax+x^{-1}))=\left(\bigotimes_{i=1}^{n(n-1)/2}H^\bullet_c(\mathbb{A}^1,0)\right)\otimes\left(\bigotimes_{j=1}^n H^\bullet_c(\mathbb{A}^1\setminus\mathbb{A}^0,f_\alpha)\right) \]
where for \(x\in(\mathbb{A}^1\setminus\mathbb{A}^0)(\mathbb{F}_{q^n}) \),
\(f_\alpha(x) = \alpha x +x^{-1}\) that corresponds to the scalar Kloosterman sum \(K_1(\alpha , \mathbb{F}_{q^n})=\lambda_1+\lambda_2\). Here \(\lambda_1\) and $\lambda_2=\overline\lambda_1$ are the $\mathbb{F}_{q^n}$-Frobenius eigenvalues on \(H^\bullet_c(\mathbb{A}^1\setminus\mathbb{A}^0,f_\alpha)\). 

It is then clear that over \(\mathbb{F}_{q^n}\) the Frobenius eigenvalues on the \(2^n\)-dimensional space \break \(\bigotimes_{j=1}^nH^1_c(\mathbb{A}^1\setminus\mathbb{A}^0,f_\alpha)\) 
are of the form  $(\prod_{i\in I} \lambda_1)(\prod_{i\notin I}\lambda_2)=\lambda_1^{|I|}\lambda_2^{n-|I|}$, where \(I\subset \{1,...,n\}\). (Therefore each of \( \lambda_1^{|I|}\lambda_2^{n-|I|}\) has multiplicity $\binom{n}{j}$.) If we fix some \(n\)-th roots of the \(\lambda_i\), say \(\eta_i^n=\lambda_i\) then we have that the Frobenius eigenvalues on \(H^{n^2}\) are of the form \(\zeta_I q^{n(n-1)/2} \eta_1^{|I|}\eta_2^{n-|I|} \) where again \(I\subset \{1,...,n\}\), and the \(\zeta_I\) are \(n\)-th roots of unity, \(\zeta_I^n=1\) for all \(I\). It is natural to make the following\footnote{While this paper was in print Elad Zelingher announced a proof of this conjecture, see \cite{Zelingher}. }
\begin{con}\label{conj:reg-ss} If $p$ is large enough, and the characteristic polynomial \(P_a\) is irreducible over \(\mathbb{F}_q\) then
\[K_n(a,\mathbb{F}_q)=(-1)^{n+1}q^{n(n-1)/2} K_1(\alpha,\mathbb{F}_{q^n}).\]
\end{con}
The conjecture would follow if for $I=\emptyset$ and $\{1,...,n\}$ the $\mathbb{F}_q$-Frobenius eigenvalues were \break $(-1)^{n+1}q^{n(n-1)/2}\lambda_1, (-1)^{n+1}q^{n(n-1)/2}\lambda_2$ and the others canceled after summing. For example when \(n=3\), this can happen if the eigenvalues $\mu_i$ for $1\leq i\leq 8$ are the eight summands in the expansion of the product \((\eta_1+\eta_2)(\omega\eta_1+\omega^2\eta_2)(\omega^2\eta_1+\omega\eta_2)q^3\) for \(\omega^3=1\). when it leads to \[K_n(a,\mathbb{F}_q)=\sum_{i=1}^8\mu_i=q^3(\lambda_1+\lambda_2)=(-1)^4q^3K_1(\alpha,\mathbb{F}_{q^3}),\]
exactly as desired.

The conjecture is partly based on the observation that if we let \(K=\mathbb{F}_q[a]\subset M_n\), then \(K\) is a field naturally isomorphic to \(\mathbb{F}_{q^n}\). \(K\) acts on \(M_n\) by left multiplication and it is easy to describe the \(K\)-algebra that arises for any \(n\). We will use this below to handle the case \(n=2\), but such elementary methods get cumbersome and are unlikely to give a proof, or even offer any insight already for \(n=3\).

We now give a few numerical examples for \(M_n(\mathbb{F}_{p^n})\) for small \(n\) and \(p\) checked with computer algebra systems pari/gp and Sage.

\begin{enumerate}
\item
Let \(n=3, p=5\) and \(\alpha \in \mathbb{F}_{125}\) one of the roots of \(x^3+x^2+1\). Then \(K_1(\alpha,\mathbb{F}_{125})=(3+\sqrt{5})/2\). On the other hand if \(A=\left[\begin{smallmatrix}
0&1&0\\0&0&1\\ -1&0&-1
\end{smallmatrix}\right] \) then
\(
K_3(A,\mathbb{F}_5)=327.2542
\)
which agrees with \(125 K_1(\alpha,\mathbb{F}_{125})\).

\item Let \(n=4\) and \(p=3\),  \(\alpha \in \mathbb{F}_{81}\) be a root of \(x^4+2x^3+2\), and \(A = \left[\begin{smallmatrix}
0&1&0&0\\0&0&1&0\\0&0&0&1 \\1&0&0&1
\end{smallmatrix}\right]\). Then
\( K_4(A,\mathbb{F}_3)=11664
\)
which agrees with \(-729 K_1(\alpha,\mathbb{F}_{81})\).

\item 
Let \(n=3\), \(p \equiv 1 \ (3)\) and \(A=\left[\begin{smallmatrix}
0&1&0\\0&0&1\\ \mu&0&0
\end{smallmatrix}\right]\), where \(\mu \in \mathbb{F}_{p}^*\setminus (\mathbb{F}_{p}^*)^3\). If \(\alpha\in \mathbb{F}_{p^3}\) is such that \(\alpha^3=\mu\)
then one can check that \[
K(\alpha)=\sum_{}e(3\mu c+(3a^2 - 3\mu cb)/\Delta(a,b,c)),
\]
where \(e(x)=e^{2\pi ix}\), \(\Delta(a,b,c)=a^3 - 3\mu cba + (\mu b^3 + \mu^2 c^3)\) and where the sum is over \((a,b,c)\in \mathbb{F}_p^3\setminus\{(0,0,0)\}\).

A direct calculation using Bruhat decomposition shows that
the conjecture in this case is equivalent to
\[
K(\alpha)=K_{13}(A)+(1+({\mu}/{p}))q,
\] where \(({\mu}/{p})\)
is the Legendre symbol and where
\[
K_{13}(A)=\sum_{t_1,t_2,t_3\in\mathbb{F}_p^*}
e ( \mu t_1^2 t_3^2 +t_1^2 t_3 +\mu t_3 +1/t_2 -1/(\mu t_1 t_3^2)  ).
\]

Up to about \(p\leq 200\), this can be checked fast even on a personal computer. For example the order of \(2 \mod 199\) is \(99\), and we have that
\[
K(\sqrt[3]{2},\mathbb{F}_{199^3})= 3869.8269
\]
while for \(A=\left(\begin{smallmatrix}
0&1&0\\0&0&1\\ 2&0&0
\end{smallmatrix}\right)\)
\[
K_{13}(A)= 4267.8269
\]
with a difference of 398, which shows that \(K_3(A,\mathbb{F}_{199})=K(\sqrt[3]{2},\mathbb{F}_{199^3})\). On the other hand \(3 \mod 199\) is a primitive root, and we have that for \(A=\left(\begin{smallmatrix}
0&1&0\\0&0&1\\ 3&0&0
\end{smallmatrix}\right)\)
\[
K(\sqrt[3]{3},\mathbb{F}_{199^3})=K_{13}(A)=-2875.1994.
\]
We also checked all \(p=3k+1\leq 200\), for which \(2\) is not a cube mod \(p\) and found that \( K(\sqrt[3]{2},\mathbb{F}_{p^3})=K_3(A, \mathbb{F}_{p})\) holds for all of them.

\end{enumerate}


\begin{pro}\label{pro-n=2-semisimple-split}
Assume that \(q\) is odd and let $\alpha\in \mathbb{F}_{q^2}\setminus\mathbb{F}_q$ and $a=\left(\begin{smallmatrix} 0 & 1 \\ -N(\alpha) & -\Tr(\alpha) \end{smallmatrix}\right)$, where $N$ and $\Tr$ are the norm and trace of the field extension $\mathbb{F}_{q^2}|\mathbb{F}_q$. Then \[K_2(a,M_2(\mathbb{F}_q))=-qK_1(\alpha,\mathbb{F}_{q^2}^*).\]
\end{pro}

\begin{rem}
The cohomology complex $H^\bullet_c$ corresponding to the sum \( K_2(a,M_2(\mathbb{F}_q))\) satisfies
\[\dim H^i_c=\begin{cases}0, & \mathrm{if~}i\neq 4\\ 4, &\mathrm{if~}i=4 \end{cases}\]
and the Frobenius eigenvalues \(\mu_i\) on $H^4_c$ satisfy \(\mu_1^2=q^2\lambda_1^2, \mu_2^2=q^4, \mu_3^2=q^4\), and \(\mu_4^2=q^2\lambda_2^2\) where $\lambda_2=\overline\lambda_1$ are the eigenvalues corresponding to $K_1(\alpha,\mathbb{F}_{q^2}^*)$. Apart from permutations the proposition determines the sign of the square roots, we have $\mu_1=-q\lambda_1$, $\mu_2=q^2$, $\mu_3=-q^2$ and $\mu_4=-q\lambda_2$. Thus
$K_2(a,M_2(\mathbb{F}_{q^m}))=\sum_{i=1}^4\mu_i^m.$ Note that here again we have cancellation: $\mu_2^m+\mu_3^m=0\iff 2\nmid m$.
\end{rem}

\begin{proof}[Proof of Proposition \ref{pro-n=2-semisimple-split}]
As above let \(K=\mathbb{F}_q[a]\) be the subring of \(M_2\) generated by \(a\). \(K\) is isomorphic \(\mathbb{F}_{q^2}\) by the assumption on \(\alpha\). Also, the vectorspace \(\mathbb{F}_q^2\) as an \(\mathbb{F}_q[a]\)-module is isomorphic to the \(\mathbb{F}_q\) vector space  \(\mathbb{F}_{q^2}\), with \(a\) acting via multiplication by \(\alpha\). For the moment denote this \(L_\alpha\), \(L_\alpha: \beta \mapsto \alpha \beta\). 
Let \( \mathbb{F}_{q^2}\langle \tau \rangle\) be the non-commutative ring of twisted polynomials, \(\sum_i \xi_i \tau^i\) subject to \(\tau \xi= F(\xi)\tau\), where \(F(\xi)=\xi^q\) is the Frobenius automorphism of \(\mathbb{F}_{q^2}/\mathbb{F}_q\).

There is an obvious map from \(\mathcal{M}_2=\mathbb{F}_{q^2}\langle \tau \rangle/(\tau^2-1)\) to \(M_2(\mathbb{F}_q)\), sending \(\xi_1+\xi_2 \tau \) to the  \(\mathbb{F}_q\)-linear transformation \( L_{\xi_0}+L_{\xi_1}F\). It is not difficult to see that this linear map
is injective, and so by dimension count, an isomorphism. This identifies \(M_2\) with \( \mathcal{M}_2\), and it is easy to check that under this identification \(\psi(L_{\xi_0}+L_{\xi_1}F)=\varphi_2(\xi_0)\), where $\varphi_2=\varphi\circ\Tr_{\mathbb{F}_{q^2}/\mathbb{F}_q}$. It follows that
\[
\sum_{x\in M_2(\mathbb{F}_q)^*} \psi(ax+x^{-1}) = \sum_{\xi_0+\xi_1\tau \in \mathcal{M}_2^*}\varphi_2(\alpha \xi_0+(\xi_0+\xi_1\tau)^{-1}).
\]
%
%
An easy calculations shows that \((1+\xi \tau)\in \mathcal{M}_2^*\) exactly when \(N(\xi)\neq 1\), and then \((1+\xi \tau)^{-1}= \frac{1}{1-N(\xi)}(1-\xi \tau)\). One also has that \((\xi \tau)^{-1}=F(\xi^{-1})\tau \) and so
\[
\mathcal{M}_2^*=\{\xi_1\tau\,|\,\xi_1 \in \mathbb{F}_{q^2}^*\} \cup \{\xi_0(1+\xi_1\tau)\,|\,\xi_0 \in \mathbb{F}_{q^2}^*, \xi_1\in \mathbb{F}_{q^2}, N(\xi_1)\neq 1\}.
\]
Therefore
\[
\sum_{x\in \GL_2(\mathbb{F}_q)} \psi(ax+x^{-1}) = 
q^2-1+
\sum_{\substack{\xi_0 \in \mathbb{F}_{q^2}^*,\xi_1 \in \mathbb{F}_{q^2}\\ N(\xi_1)\neq 1 }}
\varphi_2\left(\alpha {\xi_0}+ (1-N(\xi_1))^{-1} \xi_0^{-1}  \right).
\]
%
%
%

Now the norm map \(N\) is a surjective homomorphism from \(\mathbb{F}_{q^2}^*\to \mathbb{F}_q^*\), with a kernel of size \(q+1\) and so for \(\gamma \in \mathbb{F}_q\)
\[
\vert\{\xi\in \mathbb{F}_{q^2}\,|\,(1-N(\xi))^{-1}=\gamma\}\vert=
\begin{cases}
0 & \text{ if }\gamma =0 \\
1 & \text{ if }\gamma =1\\
q+1 & \text{ if }\gamma \neq 0,1.
\end{cases}
\]
This gives
\begin{multline*}
\sum_{x\in \GL_2(\mathbb{F}_q)} \psi(ax+x^{-1}) = 
q^2-1+
(q+1) \sum_{\substack{\xi_0 \in \mathbb{F}_{q^2}^*\\\gamma \in \mathbb{F}_{q}^*}}
\varphi_2\left(\alpha {\xi_0}+ \gamma \xi_0^{-1}  \right)
-q\sum_{\xi_0 \in \mathbb{F}_{q^2}^*}
\varphi_2\left(\alpha {\xi_0}+ \xi_0^{-1}  \right).
\end{multline*}
Finally
\begin{multline*}
\sum_{\substack{\xi_0 \in \mathbb{F}_{q^2}^*\\\gamma \in \mathbb{F}_{q}^*}}
\varphi_2\left(\alpha {\xi_0}+ \gamma \xi_0^{-1}  \right)=
\sum_{\substack{\xi_0 \in \mathbb{F}_{q^2}^*\\\gamma \in \mathbb{F}_{q}^*}}
\varphi_2\left(\alpha {\xi_0}\right)\phi\left(\gamma \Tr \xi_0^{-1}  \right)
=
-\sum_{\substack{\xi_0 \in \mathbb{F}_{q^2}^* \\ \Tr \xi_0^{-1}=0}}
\varphi_2\left(\alpha {\xi_0}\right)=-(q-1).
\end{multline*}
%
%

\end{proof}

\subsection{The purity locus}\label{sec--ex-purity}

We have seen that for a regular semi-simple element \(a\in M_n(\mathbb{F}_q)\) the Kloosterman sum \(K_n(a)\) is pure. The tables above already suggest that for a matrix \(a\) with more than one Jordan block for an eigenvalue, that sum can not be pure. This can be seen without reference to cohomology. To see this assume that \(a\) has a single eigenvalue \(\alpha\). Recall that
\[
K_n(a)=P(A,G,K)=\sum_{2f\leq n}^{}c_f(q,q-1)K^{n-2f},
\]
where  \(c_f\) are polynomials. We give the \(A\) and \(G\) weigth 1, and \(K\) weight \(1/2\), so the polynomial \(P\) has a weighted degree, which determines the order of magnitude (in $q$) of its value . Now \(f=1\) corresponds to simple transpositions, and by Theorem~\ref{thm-inv-S} one sees that these sums are too large in magnitude to be pure if not all \(\varepsilon_i\) are 0.

It is an intriguing question what happens  for \(K_{[n]}\) when \(a\) has only one Jordan block. The recursion formula gives
\[
K_{[n]}=q^{n-1} K K_{[n-1]}-q^{2n-2}K_{[n-2]}
\]
where \( K =K_1(\alpha)\). Let \(k_n=q^{-n(n-1)/2}K_{[n]}\), so that we have
\[
k_n= K k_{n-1}-qk_{n-2}.
\]
It follows that there exist \(c_1,c_2\) such that
\[
k_n = c_1 \lambda_1^n + c_2 \lambda_2^n
\]
where \(\lambda_1,\lambda_2\) are the roots of \( X^2 - K  X+q\). These are exactly the eigenvalues of Frobenius acting on the cohomology of the Kloosterman sheaf. Using that \( K =\lambda_1+\lambda_2\), and that \(K_{[2]}=-q^2+K_{1}^2 q\) we get that
\(k_{1}=\lambda_1+\lambda_2\), and that \(k_{2}= k_1^2-q=\lambda_1^2+\lambda_1\lambda_2 +\lambda_2^2\). Therefore \(c_1= \frac{\lambda_1}{\lambda_1-\lambda_2} \), \(c_2= -\frac{\lambda_2}{\lambda_1-\lambda_2} \) and
\begin{equation}\label{eq-U-Chebyshev}
k_n=\frac{\lambda_1^{n+1}-\lambda_2^{n+1}}{\lambda_1-\lambda_2}=\sum_{j=0}^{n} \lambda_1^j \lambda_2^{n-j}.
\end{equation}

This evaluation has an interesting interpretation. Let \(X^2-K_1X+q=(X-\lambda_1)(X-\lambda_2)\), with \(\lambda_{1,2}=q^{1/2} e^{\pm i\theta}\), so that \(K_1=2q^{1/2}\cos \theta\). We have that
\[
k_{n}=\frac{\lambda_1^{n+1}-\lambda_2^{n+1}}{\lambda_1-\lambda_2}=
q^{n/2} \frac{e^{(n+1)\theta}- e^{-(n+1)\theta}} {e^{\theta}-e^{-\theta}}=\frac{\sin(n+1)\theta}{\sin \theta}
\]
Therefore
\begin{equation}\label{eq-K_n-U_n}
K_{[n]}=q^{n(n-1)/2} U_n(\cos \theta)
\end{equation}
where \(U_n\) is the Chebyshev polynomial of the second kind.
The Sato-Tate distribution of the angles of \(K_1(\alpha)\) over the valuations of a global field is then equivalent to non-trivial cancellation in the sums \[\sum_{N(v) \leq x} K_{n}(a,\mathbb{F}_{v})/N(v)^{n(n-1)/2},\] where \(a=\alpha \I+\sum_{i=1}^{n-1}\e_{i,i+1}\).

Getting back to the question of purity these sums are pure from a numerical point of view, but this in itself does not rule out a cohomology with a nilpotent Frobenius action.

For example in the case $n=2$ it is easy to see that the cohomologies corresponding to the Bruhat cells are as follows.

On $C_I=B$ the trace of $\alpha x+x^{-1}$ can be written as a product of two Kloosterman sums over the diagonal elements, thus we have
\[H^\bullet_{C_I}=H^\bullet(C_I,x\mapsto tr(ax+x^{-1}))=H^\bullet_c(\mathbb{A}^1-\mathbb{A}^0,f_\alpha)\otimes H^\bullet_c(\mathbb{A}^1-\mathbb{A}^0,f_\alpha)\otimes H^\bullet_c(\mathbb{A}^1,0).\] That implies $\dim(H^i_{C_I})=0$ unless $i=4$ and $\dim(H^4_{C_I})=4$.

On the nontrivial cell $C_w=UwB$ we have seen that the sum (and the cohomology) cancels on the subvariety $\alpha\det(b)\neq1$ and on the rest we have \[H^\bullet_{C_w}=H^\bullet(C_w,x\mapsto\tr(ax+x^{-1}))=H^\bullet_c(\mathbb{A}^1-\mathbb{A}^0,\mathrm{id})\otimes H^\bullet_c(\mathbb{A}^1,0)\otimes H^\bullet_c(\mathbb{A}^1,0).\]
That implies $\dim(H^i_{C_w})=0$ unless $i=5$ and $\dim(H^5_{C_w})=1$.

Thus the long exact sequence of the excision ($C_w=G\setminus C_I$) gives
\[0\to H^4_G\to H^4_{C_I}\to H^5_{C_w}\to H^5_G\to0.\]

Either $\dim(H^4_G)=4$ and $\dim(H^5_G)=1$ or $\dim(H^4_G)=3$ and $\dim(H^5_G)=0$ seems to be possible.

The same problem exists for higher degree cases.


\end{document}